\documentclass{siamart250211}
\usepackage{titling}
\title{A minimax method for \protect \\ the spectral
fractional Laplacian \protect\\ and related evolution problems}

 \usepackage{authblk}
\usepackage{verbatim}

\author[1]{J.~A. Carrillo}
\author[1]{\authorcr S. Fronzoni}
\author[1]{Y. Nakatsukasa}
\author[1]{E. S\"{u}li}

\affil[1]{
	Mathematical Institute, University of Oxford
}
\affil[ ]{
	OX2 6GG Oxford, United Kingdom
}
\affil[ ]{\textit{carrillo@maths.ox.ac.uk, fronzoni@maths.ox.ac.uk, nakatsukasa@maths.ox.ac.uk, suli@maths.ox.ac.uk}}

\affil[ ]{}

\usepackage[utf8]{inputenc}
\usepackage[T1]{fontenc}
\usepackage{amsmath}
\usepackage[showonlyrefs]{mathtools}
\usepackage{amsfonts}
\usepackage{amssymb}
\usepackage{graphicx}
\usepackage{enumitem}
\usepackage{ifthen}
\usepackage{dsfont}
\usepackage{float}
\usepackage{lettrine}
\usepackage{lipsum}
\usepackage{subcaption}

\usepackage[UKenglish]{babel}

\definecolor{ppGreen}{HTML}{008000}
\definecolor{ppBlue}{HTML}{0000FF}
\definecolor{ppRed}{HTML}{FF0000}
\definecolor{ppPurple}{HTML}{800080}
\definecolor{lightblue}{rgb}{0.145,0.6666,1}
\definecolor{grey52}{RGB}{52,52,52}
\definecolor{color1}{RGB}{0,62,116}
\definecolor{color2}{RGB}{152,152,152}
\definecolor{color3}{RGB}{52,52,52}
\definecolor{color4}{RGB}{100,100,100}

\definecolor{imperialnavy}{RGB}{0,33,71}
\definecolor{imperialblue}{RGB}{0,62,116}
\definecolor{imperialgrey}{RGB}{235,238,238}
\definecolor{imperialcoolgrey}{RGB}{157,157,157}

\definecolor{lille}{RGB}{178, 35, 114} 
\newcounter{review}

\newcommand{\ntcreview}[3]{
	\refstepcounter{review}
	{\color{#2}{\textbf{[#1]}: #3}}
}

\newcommand{\creview}[3]{
	\ntcreview{#1}{#2}{#3}
	\addcontentsline{tor}{subsection}{
		\thereview~\textbf{[#1]}:~#3
	}
}

\newcommand{\review}[2]{\creview{#1}{blue}{#2}}

\makeatletter

\newcommand\listreviewname{List of Reviews}
\newcommand\listofreviews{
	\section*{\listreviewname}\@starttoc{tor}}
\makeatother

\usepackage[all]{nowidow}

\usepackage{tikz}

\usepackage{pgfplots}
\pgfplotsset{compat=1.16}

\usepackage{float}
\usepackage{subcaption}
\usepackage[section]{placeins}
\usepackage{rotating}

\usepackage{array}
\newcolumntype{L}[1]{>{\raggedright\let\newline\\\arraybackslash\hspace{0pt}}m{#1}}
\newcolumntype{C}[1]{>{\centering\let\newline\\\arraybackslash\hspace{0pt}}m{#1}}
\newcolumntype{R}[1]{>{\raggedleft\let\newline\\\arraybackslash\hspace{0pt}}m{#1}}
\usepackage{booktabs}
\usepackage{tabularx}
\usepackage{multirow}

\usepackage{titling}

 \usepackage{accents}

\newcommand\term\emph

\numberwithin{equation}{section}

\usepackage{autobreak}

\usepackage{enumerate}

\makeatletter
\def\@maketitle{
	\newpage
	\begin{center}
		\let \footnote \thanks
		{\LARGE\bfseries \@title \par}
		\vskip 2.5em
			{\large
				\lineskip .5em
				\begin{tabular}[t]{c}
					\@author
				\end{tabular}\par}
		\vskip 1em
			{\large \@date}
	\end{center}
	\par
	\vskip 1.5em}
\makeatother


\newtheorem{prop}[theorem]{Proposition}

\newtheorem{remark}[theorem]{\bf Remark}

\usepackage{esint}

\def\XXint#1#2#3{{\setbox0=\hbox{$#1{#2#3}{\int}$ }
			\vcenter{\hbox{$#2#3$ }}\kern-.6\wd0}}

\DeclarePairedDelimiter{\norm}{\|}{\|}

\DeclarePairedDelimiter{\inn}{\langle}{\rangle}

\usepackage{suffix}

\newcommand{\inner}[2]{\inn{#1,#2}}
\WithSuffix\newcommand\inner*[2]{\inn*{#1,#2}}

\DeclarePairedDelimiter{\positive}{(}{)^{+}}
\DeclarePairedDelimiter{\negative}{(}{)^{-}}

\newcommand\pos\positive
\renewcommand\neg\negative

\WithSuffix\newcommand\pos*{\positive*}
\WithSuffix\newcommand\neg*{\negative*}

\newcommand{\pnorm}[2]{\norm{#2}_{\L{#1}}}
\WithSuffix\newcommand\pnorm*[2]{\norm*{#2}_{\L{#1}}}

\newcommand{\psnorm}[3]{\norm{#3}_{\L{#1}(#2)}}
\WithSuffix\newcommand\psnorm*[3]{\norm*{#3}_{\L{#1}(#2)}}

\newcommand{\pnormp}[2]{\pnorm{#1}{#2}^{#1}}
\WithSuffix\newcommand\pnormp*[2]{\pnorm*{#1}{#2}^{#1}}

\newcommand{\psnormp}[3]{\psnorm{#1}{#2}{#3}^{#1}}
\WithSuffix\newcommand\psnormp*[3]{\psnorm*{#1}{#2}{#3}^{#1}}

\newcommand\svec\vec

\renewcommand{\vec}{\mathbf}
\renewcommand{\svec}{\boldsymbol}

\renewcommand{\d}{\mathrm{d}}
\newcommand{\dd}{\mathop{}\!\d}

\newcommand{\dt}{\dd t}

\newcommand{\dx}{\dd x}
\newcommand{\dy}{\dd y}

\newcommand{\ppr}{(r)}

\newcommand{\Wr}{^{W,\,\ppr}}

\newlength{\dhatheight}

\ifthenelse{\isundefined{\Wr}}{
	\newcommand{\Wr}{^{W,\,\ppr}}
}{
	\renewcommand{\Wr}{^{W,\,\ppr}}
}

\newcommand{\s}{s}

\newcommand{\I}{\mathcal{I}}
\newcommand{\Hs}{\mathbb{H}^{\s}}

\newcommand{\V}{\mathcal{V}}

\newcommand{\ddlambda}{\frac{\mathrm{d}}{\mathrm{d}\lambda}}

\newcommand{\ddtheta}{\frac{\mathrm{d}}{\mathrm{d}\theta}} \usepackage{todonotes}

\newif\ifskiptable

\pgfplotsset{
	colormap={hsv}{
			hsb(0.00cm)=(0.00,0,0.95);
			hsb(0.05cm)=(0.05,1,1);
			hsb(0.10cm)=(0.10,1,1);
			hsb(0.15cm)=(0.15,1,1);
			hsb(0.20cm)=(0.20,1,1);
			hsb(0.25cm)=(0.25,1,1);
			hsb(0.30cm)=(0.30,1,1);
			hsb(0.35cm)=(0.35,1,1);
			hsb(0.40cm)=(0.40,1,1);
			hsb(0.45cm)=(0.45,1,1);
			hsb(0.50cm)=(0.50,1,1);
			hsb(0.55cm)=(0.55,1,1);
			hsb(0.60cm)=(0.60,1,1);
			hsb(0.65cm)=(0.65,1,1);
			hsb(0.70cm)=(0.70,1,1);
			hsb(0.75cm)=(0.75,1,1);
			hsb(0.80cm)=(0.80,1,1);
			hsb(0.85cm)=(0.85,1,1);
			hsb(0.90cm)=(0.90,1,1);
			hsb(0.95cm)=(0.95,1,1);
			hsb(1.00cm)=(1.00,1,1);
		}
}

\pgfplotsset{
	colormap={hsvSoft}{
			hsb(0.00cm)=(0.00,0,0.95);
			hsb(0.05cm)=(0.05,1,1);
			hsb(0.10cm)=(0.10,1,1);
			hsb(0.15cm)=(0.15,1,1);
			hsb(0.20cm)=(0.20,1,1);
			hsb(0.25cm)=(0.25,1,1);
			hsb(0.30cm)=(0.30,1,1);
			hsb(0.35cm)=(0.35,1,1);
			hsb(0.40cm)=(0.40,1,1);
			hsb(0.45cm)=(0.45,1,1);
			hsb(0.50cm)=(0.50,1,1);
			hsb(0.55cm)=(0.55,1,1);
			hsb(0.60cm)=(0.60,1,1);
			hsb(0.65cm)=(0.65,1,1);
			hsb(0.70cm)=(0.70,1,1);
			hsb(0.75cm)=(0.75,1,1);
			hsb(0.80cm)=(0.80,1,1);
			hsb(0.85cm)=(0.85,1,1);
			hsb(0.90cm)=(0.90,1,1);
			hsb(0.95cm)=(0.95,1,1);
			hsb(1.00cm)=(0.00,0,0.95);
		}
}

\pgfplotsset{
	colormap={viridisFull}{
			rgb=(0.26700401, 0.00487433, 0.32941519)
			rgb=(0.26851048, 0.00960483, 0.33542652)
			rgb=(0.26994384, 0.01462494, 0.34137895)
			rgb=(0.27130489, 0.01994186, 0.34726862)
			rgb=(0.27259384, 0.02556309, 0.35309303)
			rgb=(0.27380934, 0.03149748, 0.35885256)
			rgb=(0.27495242, 0.03775181, 0.36454323)
			rgb=(0.27602238, 0.04416723, 0.37016418)
			rgb=(0.2770184 , 0.05034437, 0.37571452)
			rgb=(0.27794143, 0.05632444, 0.38119074)
			rgb=(0.27879067, 0.06214536, 0.38659204)
			rgb=(0.2795655 , 0.06783587, 0.39191723)
			rgb=(0.28026658, 0.07341724, 0.39716349)
			rgb=(0.28089358, 0.07890703, 0.40232944)
			rgb=(0.28144581, 0.0843197 , 0.40741404)
			rgb=(0.28192358, 0.08966622, 0.41241521)
			rgb=(0.28232739, 0.09495545, 0.41733086)
			rgb=(0.28265633, 0.10019576, 0.42216032)
			rgb=(0.28291049, 0.10539345, 0.42690202)
			rgb=(0.28309095, 0.11055307, 0.43155375)
			rgb=(0.28319704, 0.11567966, 0.43611482)
			rgb=(0.28322882, 0.12077701, 0.44058404)
			rgb=(0.28318684, 0.12584799, 0.44496 )
			rgb=(0.283072 , 0.13089477, 0.44924127)
			rgb=(0.28288389, 0.13592005, 0.45342734)
			rgb=(0.28262297, 0.14092556, 0.45751726)
			rgb=(0.28229037, 0.14591233, 0.46150995)
			rgb=(0.28188676, 0.15088147, 0.46540474)
			rgb=(0.28141228, 0.15583425, 0.46920128)
			rgb=(0.28086773, 0.16077132, 0.47289909)
			rgb=(0.28025468, 0.16569272, 0.47649762)
			rgb=(0.27957399, 0.17059884, 0.47999675)
			rgb=(0.27882618, 0.1754902 , 0.48339654)
			rgb=(0.27801236, 0.18036684, 0.48669702)
			rgb=(0.27713437, 0.18522836, 0.48989831)
			rgb=(0.27619376, 0.19007447, 0.49300074)
			rgb=(0.27519116, 0.1949054 , 0.49600488)
			rgb=(0.27412802, 0.19972086, 0.49891131)
			rgb=(0.27300596, 0.20452049, 0.50172076)
			rgb=(0.27182812, 0.20930306, 0.50443413)
			rgb=(0.27059473, 0.21406899, 0.50705243)
			rgb=(0.26930756, 0.21881782, 0.50957678)
			rgb=(0.26796846, 0.22354911, 0.5120084 )
			rgb=(0.26657984, 0.2282621 , 0.5143487 )
			rgb=(0.2651445 , 0.23295593, 0.5165993 )
			rgb=(0.2636632 , 0.23763078, 0.51876163)
			rgb=(0.26213801, 0.24228619, 0.52083736)
			rgb=(0.26057103, 0.2469217 , 0.52282822)
			rgb=(0.25896451, 0.25153685, 0.52473609)
			rgb=(0.25732244, 0.2561304 , 0.52656332)
			rgb=(0.25564519, 0.26070284, 0.52831152)
			rgb=(0.25393498, 0.26525384, 0.52998273)
			rgb=(0.25219404, 0.26978306, 0.53157905)
			rgb=(0.25042462, 0.27429024, 0.53310261)
			rgb=(0.24862899, 0.27877509, 0.53455561)
			rgb=(0.2468114 , 0.28323662, 0.53594093)
			rgb=(0.24497208, 0.28767547, 0.53726018)
			rgb=(0.24311324, 0.29209154, 0.53851561)
			rgb=(0.24123708, 0.29648471, 0.53970946)
			rgb=(0.23934575, 0.30085494, 0.54084398)
			rgb=(0.23744138, 0.30520222, 0.5419214 )
			rgb=(0.23552606, 0.30952657, 0.54294396)
			rgb=(0.23360277, 0.31382773, 0.54391424)
			rgb=(0.2316735 , 0.3181058 , 0.54483444)
			rgb=(0.22973926, 0.32236127, 0.54570633)
			rgb=(0.22780192, 0.32659432, 0.546532 )
			rgb=(0.2258633 , 0.33080515, 0.54731353)
			rgb=(0.22392515, 0.334994 , 0.54805291)
			rgb=(0.22198915, 0.33916114, 0.54875211)
			rgb=(0.22005691, 0.34330688, 0.54941304)
			rgb=(0.21812995, 0.34743154, 0.55003755)
			rgb=(0.21620971, 0.35153548, 0.55062743)
			rgb=(0.21429757, 0.35561907, 0.5511844 )
			rgb=(0.21239477, 0.35968273, 0.55171011)
			rgb=(0.2105031 , 0.36372671, 0.55220646)
			rgb=(0.20862342, 0.36775151, 0.55267486)
			rgb=(0.20675628, 0.37175775, 0.55311653)
			rgb=(0.20490257, 0.37574589, 0.55353282)
			rgb=(0.20306309, 0.37971644, 0.55392505)
			rgb=(0.20123854, 0.38366989, 0.55429441)
			rgb=(0.1994295 , 0.38760678, 0.55464205)
			rgb=(0.1976365 , 0.39152762, 0.55496905)
			rgb=(0.19585993, 0.39543297, 0.55527637)
			rgb=(0.19410009, 0.39932336, 0.55556494)
			rgb=(0.19235719, 0.40319934, 0.55583559)
			rgb=(0.19063135, 0.40706148, 0.55608907)
			rgb=(0.18892259, 0.41091033, 0.55632606)
			rgb=(0.18723083, 0.41474645, 0.55654717)
			rgb=(0.18555593, 0.4185704 , 0.55675292)
			rgb=(0.18389763, 0.42238275, 0.55694377)
			rgb=(0.18225561, 0.42618405, 0.5571201 )
			rgb=(0.18062949, 0.42997486, 0.55728221)
			rgb=(0.17901879, 0.43375572, 0.55743035)
			rgb=(0.17742298, 0.4375272 , 0.55756466)
			rgb=(0.17584148, 0.44128981, 0.55768526)
			rgb=(0.17427363, 0.4450441 , 0.55779216)
			rgb=(0.17271876, 0.4487906 , 0.55788532)
			rgb=(0.17117615, 0.4525298 , 0.55796464)
			rgb=(0.16964573, 0.45626209, 0.55803034)
			rgb=(0.16812641, 0.45998802, 0.55808199)
			rgb=(0.1666171 , 0.46370813, 0.55811913)
			rgb=(0.16511703, 0.4674229 , 0.55814141)
			rgb=(0.16362543, 0.47113278, 0.55814842)
			rgb=(0.16214155, 0.47483821, 0.55813967)
			rgb=(0.16066467, 0.47853961, 0.55811466)
			rgb=(0.15919413, 0.4822374 , 0.5580728 )
			rgb=(0.15772933, 0.48593197, 0.55801347)
			rgb=(0.15626973, 0.4896237 , 0.557936 )
			rgb=(0.15481488, 0.49331293, 0.55783967)
			rgb=(0.15336445, 0.49700003, 0.55772371)
			rgb=(0.1519182 , 0.50068529, 0.55758733)
			rgb=(0.15047605, 0.50436904, 0.55742968)
			rgb=(0.14903918, 0.50805136, 0.5572505 )
			rgb=(0.14760731, 0.51173263, 0.55704861)
			rgb=(0.14618026, 0.51541316, 0.55682271)
			rgb=(0.14475863, 0.51909319, 0.55657181)
			rgb=(0.14334327, 0.52277292, 0.55629491)
			rgb=(0.14193527, 0.52645254, 0.55599097)
			rgb=(0.14053599, 0.53013219, 0.55565893)
			rgb=(0.13914708, 0.53381201, 0.55529773)
			rgb=(0.13777048, 0.53749213, 0.55490625)
			rgb=(0.1364085 , 0.54117264, 0.55448339)
			rgb=(0.13506561, 0.54485335, 0.55402906)
			rgb=(0.13374299, 0.54853458, 0.55354108)
			rgb=(0.13244401, 0.55221637, 0.55301828)
			rgb=(0.13117249, 0.55589872, 0.55245948)
			rgb=(0.1299327 , 0.55958162, 0.55186354)
			rgb=(0.12872938, 0.56326503, 0.55122927)
			rgb=(0.12756771, 0.56694891, 0.55055551)
			rgb=(0.12645338, 0.57063316, 0.5498411 )
			rgb=(0.12539383, 0.57431754, 0.54908564)
			rgb=(0.12439474, 0.57800205, 0.5482874 )
			rgb=(0.12346281, 0.58168661, 0.54744498)
			rgb=(0.12260562, 0.58537105, 0.54655722)
			rgb=(0.12183122, 0.58905521, 0.54562298)
			rgb=(0.12114807, 0.59273889, 0.54464114)
			rgb=(0.12056501, 0.59642187, 0.54361058)
			rgb=(0.12009154, 0.60010387, 0.54253043)
			rgb=(0.11973756, 0.60378459, 0.54139999)
			rgb=(0.11951163, 0.60746388, 0.54021751)
			rgb=(0.11942341, 0.61114146, 0.53898192)
			rgb=(0.11948255, 0.61481702, 0.53769219)
			rgb=(0.11969858, 0.61849025, 0.53634733)
			rgb=(0.12008079, 0.62216081, 0.53494633)
			rgb=(0.12063824, 0.62582833, 0.53348834)
			rgb=(0.12137972, 0.62949242, 0.53197275)
			rgb=(0.12231244, 0.63315277, 0.53039808)
			rgb=(0.12344358, 0.63680899, 0.52876343)
			rgb=(0.12477953, 0.64046069, 0.52706792)
			rgb=(0.12632581, 0.64410744, 0.52531069)
			rgb=(0.12808703, 0.64774881, 0.52349092)
			rgb=(0.13006688, 0.65138436, 0.52160791)
			rgb=(0.13226797, 0.65501363, 0.51966086)
			rgb=(0.13469183, 0.65863619, 0.5176488 )
			rgb=(0.13733921, 0.66225157, 0.51557101)
			rgb=(0.14020991, 0.66585927, 0.5134268 )
			rgb=(0.14330291, 0.66945881, 0.51121549)
			rgb=(0.1466164 , 0.67304968, 0.50893644)
			rgb=(0.15014782, 0.67663139, 0.5065889 )
			rgb=(0.15389405, 0.68020343, 0.50417217)
			rgb=(0.15785146, 0.68376525, 0.50168574)
			rgb=(0.16201598, 0.68731632, 0.49912906)
			rgb=(0.1663832 , 0.69085611, 0.49650163)
			rgb=(0.1709484 , 0.69438405, 0.49380294)
			rgb=(0.17570671, 0.6978996 , 0.49103252)
			rgb=(0.18065314, 0.70140222, 0.48818938)
			rgb=(0.18578266, 0.70489133, 0.48527326)
			rgb=(0.19109018, 0.70836635, 0.48228395)
			rgb=(0.19657063, 0.71182668, 0.47922108)
			rgb=(0.20221902, 0.71527175, 0.47608431)
			rgb=(0.20803045, 0.71870095, 0.4728733 )
			rgb=(0.21400015, 0.72211371, 0.46958774)
			rgb=(0.22012381, 0.72550945, 0.46622638)
			rgb=(0.2263969 , 0.72888753, 0.46278934)
			rgb=(0.23281498, 0.73224735, 0.45927675)
			rgb=(0.2393739 , 0.73558828, 0.45568838)
			rgb=(0.24606968, 0.73890972, 0.45202405)
			rgb=(0.25289851, 0.74221104, 0.44828355)
			rgb=(0.25985676, 0.74549162, 0.44446673)
			rgb=(0.26694127, 0.74875084, 0.44057284)
			rgb=(0.27414922, 0.75198807, 0.4366009 )
			rgb=(0.28147681, 0.75520266, 0.43255207)
			rgb=(0.28892102, 0.75839399, 0.42842626)
			rgb=(0.29647899, 0.76156142, 0.42422341)
			rgb=(0.30414796, 0.76470433, 0.41994346)
			rgb=(0.31192534, 0.76782207, 0.41558638)
			rgb=(0.3198086 , 0.77091403, 0.41115215)
			rgb=(0.3277958 , 0.77397953, 0.40664011)
			rgb=(0.33588539, 0.7770179 , 0.40204917)
			rgb=(0.34407411, 0.78002855, 0.39738103)
			rgb=(0.35235985, 0.78301086, 0.39263579)
			rgb=(0.36074053, 0.78596419, 0.38781353)
			rgb=(0.3692142 , 0.78888793, 0.38291438)
			rgb=(0.37777892, 0.79178146, 0.3779385 )
			rgb=(0.38643282, 0.79464415, 0.37288606)
			rgb=(0.39517408, 0.79747541, 0.36775726)
			rgb=(0.40400101, 0.80027461, 0.36255223)
			rgb=(0.4129135 , 0.80304099, 0.35726893)
			rgb=(0.42190813, 0.80577412, 0.35191009)
			rgb=(0.43098317, 0.80847343, 0.34647607)
			rgb=(0.44013691, 0.81113836, 0.3409673 )
			rgb=(0.44936763, 0.81376835, 0.33538426)
			rgb=(0.45867362, 0.81636288, 0.32972749)
			rgb=(0.46805314, 0.81892143, 0.32399761)
			rgb=(0.47750446, 0.82144351, 0.31819529)
			rgb=(0.4870258 , 0.82392862, 0.31232133)
			rgb=(0.49661536, 0.82637633, 0.30637661)
			rgb=(0.5062713 , 0.82878621, 0.30036211)
			rgb=(0.51599182, 0.83115784, 0.29427888)
			rgb=(0.52577622, 0.83349064, 0.2881265 )
			rgb=(0.5356211 , 0.83578452, 0.28190832)
			rgb=(0.5455244 , 0.83803918, 0.27562602)
			rgb=(0.55548397, 0.84025437, 0.26928147)
			rgb=(0.5654976 , 0.8424299 , 0.26287683)
			rgb=(0.57556297, 0.84456561, 0.25641457)
			rgb=(0.58567772, 0.84666139, 0.24989748)
			rgb=(0.59583934, 0.84871722, 0.24332878)
			rgb=(0.60604528, 0.8507331 , 0.23671214)
			rgb=(0.61629283, 0.85270912, 0.23005179)
			rgb=(0.62657923, 0.85464543, 0.22335258)
			rgb=(0.63690157, 0.85654226, 0.21662012)
			rgb=(0.64725685, 0.85839991, 0.20986086)
			rgb=(0.65764197, 0.86021878, 0.20308229)
			rgb=(0.66805369, 0.86199932, 0.19629307)
			rgb=(0.67848868, 0.86374211, 0.18950326)
			rgb=(0.68894351, 0.86544779, 0.18272455)
			rgb=(0.69941463, 0.86711711, 0.17597055)
			rgb=(0.70989842, 0.86875092, 0.16925712)
			rgb=(0.72039115, 0.87035015, 0.16260273)
			rgb=(0.73088902, 0.87191584, 0.15602894)
			rgb=(0.74138803, 0.87344918, 0.14956101)
			rgb=(0.75188414, 0.87495143, 0.14322828)
			rgb=(0.76237342, 0.87642392, 0.13706449)
			rgb=(0.77285183, 0.87786808, 0.13110864)
			rgb=(0.78331535, 0.87928545, 0.12540538)
			rgb=(0.79375994, 0.88067763, 0.12000532)
			rgb=(0.80418159, 0.88204632, 0.11496505)
			rgb=(0.81457634, 0.88339329, 0.11034678)
			rgb=(0.82494028, 0.88472036, 0.10621724)
			rgb=(0.83526959, 0.88602943, 0.1026459 )
			rgb=(0.84556056, 0.88732243, 0.09970219)
			rgb=(0.8558096 , 0.88860134, 0.09745186)
			rgb=(0.86601325, 0.88986815, 0.09595277)
			rgb=(0.87616824, 0.89112487, 0.09525046)
			rgb=(0.88627146, 0.89237353, 0.09537439)
			rgb=(0.89632002, 0.89361614, 0.09633538)
			rgb=(0.90631121, 0.89485467, 0.09812496)
			rgb=(0.91624212, 0.89609127, 0.1007168 )
			rgb=(0.92610579, 0.89732977, 0.10407067)
			rgb=(0.93590444, 0.8985704 , 0.10813094)
			rgb=(0.94563626, 0.899815 , 0.11283773)
			rgb=(0.95529972, 0.90106534, 0.11812832)
			rgb=(0.96489353, 0.90232311, 0.12394051)
			rgb=(0.97441665, 0.90358991, 0.13021494)
			rgb=(0.98386829, 0.90486726, 0.13689671)
			rgb=(0.99324789, 0.90615657, 0.1439362 )
		}
}

\pgfplotsset{
	colormap={viridisSoft}{
			rgb255=(242, 242, 242);
			rgb=(0.28026,0.1657,0.4765);
			rgb=(0.26366,0.23763,0.51877);
			rgb=(0.23744,0.3052,0.54192);
			rgb=(0.20862,0.36775,0.55267);
			rgb=(0.18225,0.42618,0.55711);
			rgb=(0.1592,0.48224,0.55807);
			rgb=(0.13777,0.53749,0.5549);
			rgb=(0.12115,0.59274,0.54465);
			rgb=(0.12808,0.64775,0.5235);
			rgb=(0.18065,0.7014,0.48819);
			rgb=(0.27415,0.75198,0.4366);
			rgb=(0.39517,0.79747,0.36775);
			rgb=(0.53561,0.83578,0.2819);
			rgb=(0.68895,0.86545,0.18272);
			rgb=(0.84557,0.88733,0.0997);
			rgb=(0.99324,0.90616,0.14394)
		}
}

\pgfplotsset{
	colormap={cellRed}{
			rgb255=(242.0,242.0,242.0);
			rgb255=(241.63157894736844,234.47368421052633,234.47368421052633);
			rgb255=(241.26315789473685,226.94736842105266,226.94736842105266);
			rgb255=(240.89473684210526,219.42105263157893,219.42105263157893);
			rgb255=(240.5263157894737,211.89473684210526,211.89473684210526);
			rgb255=(240.1578947368421,204.3684210526316,204.3684210526316);
			rgb255=(239.78947368421052,196.84210526315792,196.84210526315792);
			rgb255=(239.42105263157896,189.31578947368422,189.31578947368422);
			rgb255=(239.05263157894737,181.78947368421052,181.78947368421052);
			rgb255=(238.6842105263158,174.26315789473688,174.26315789473688);
			rgb255=(238.31578947368422,166.73684210526315,166.73684210526315);
			rgb255=(237.94736842105263,159.21052631578948,159.21052631578948);
			rgb255=(237.57894736842104,151.68421052631578,151.68421052631578);
			rgb255=(237.21052631578948,144.1578947368421,144.1578947368421);
			rgb255=(236.84210526315792,136.63157894736844,136.63157894736844);
			rgb255=(236.47368421052633,129.10526315789474,129.10526315789474);
			rgb255=(236.10526315789474,121.57894736842107,121.57894736842107);
			rgb255=(235.73684210526318,114.05263157894737,114.05263157894737);
			rgb255=(235.3684210526316,106.52631578947368,106.52631578947368);
			rgb255=(235.0,99.0,99.0);
		}
}

\pgfplotsset{
	colormap={cellGreen}{
			rgb255=(242.0,242.0,242.0);
			rgb255=(236.21052631578948,239.5263157894737,234.26315789473685);
			rgb255=(230.42105263157896,237.05263157894737,226.5263157894737);
			rgb255=(224.6315789473684,234.57894736842104,218.78947368421052);
			rgb255=(218.8421052631579,232.10526315789474,211.05263157894737);
			rgb255=(213.05263157894737,229.63157894736844,203.31578947368422);
			rgb255=(207.26315789473685,227.1578947368421,195.57894736842107);
			rgb255=(201.4736842105263,224.68421052631578,187.8421052631579);
			rgb255=(195.68421052631578,222.21052631578948,180.10526315789474);
			rgb255=(189.8947368421053,219.73684210526318,172.36842105263162);
			rgb255=(184.10526315789474,217.26315789473682,164.63157894736844);
			rgb255=(178.31578947368422,214.78947368421052,156.89473684210526);
			rgb255=(172.5263157894737,212.31578947368422,149.1578947368421);
			rgb255=(166.73684210526318,209.84210526315792,141.42105263157896);
			rgb255=(160.94736842105263,207.3684210526316,133.6842105263158);
			rgb255=(155.1578947368421,204.89473684210526,125.94736842105263);
			rgb255=(149.3684210526316,202.42105263157893,118.21052631578948);
			rgb255=(143.57894736842104,199.94736842105266,110.47368421052632);
			rgb255=(137.78947368421052,197.47368421052633,102.73684210526316);
			rgb255=(132.0,195.0,95.0);
		}
}

\pgfplotsset{
	colormap={cellRedSquared}{
			rgb255=(242.0,242.0,242.0);
			rgb255=(241.28254847645428,227.34349030470915,227.34349030470915);
			rgb255=(240.60387811634348,213.47922437673128,213.47922437673128);
			rgb255=(239.9639889196676,200.40720221606648,200.40720221606648);
			rgb255=(239.36288088642658,188.1274238227147,188.1274238227147);
			rgb255=(238.8005540166205,176.63988919667594,176.63988919667594);
			rgb255=(238.2770083102493,165.94459833795014,165.94459833795014);
			rgb255=(237.79224376731304,156.04155124653738,156.04155124653738);
			rgb255=(237.34626038781164,146.93074792243766,146.93074792243766);
			rgb255=(236.93905817174516,138.61218836565098,138.61218836565098);
			rgb255=(236.57063711911357,131.0858725761773,131.0858725761773);
			rgb255=(236.2409972299169,124.35180055401662,124.35180055401662);
			rgb255=(235.95013850415512,118.40997229916897,118.40997229916897);
			rgb255=(235.69806094182823,113.26038781163435,113.26038781163435);
			rgb255=(235.4847645429363,108.90304709141274,108.90304709141274);
			rgb255=(235.3102493074792,105.33795013850416,105.33795013850416);
			rgb255=(235.17451523545705,102.56509695290858,102.56509695290858);
			rgb255=(235.0775623268698,100.58448753462605,100.58448753462605);
			rgb255=(235.01939058171746,99.3961218836565,99.3961218836565);
			rgb255=(235.0,99.0,99.0);
		}
}
\pgfplotsset{
	colormap={cellGreenSquared}{
			rgb255=(242.0,242.0,242.0);
			rgb255=(230.7257617728532,237.18282548476455,226.93351800554018);
			rgb255=(220.06094182825484,232.62603878116343,212.6814404432133);
			rgb255=(210.00554016620498,228.32963988919667,199.2437673130194);
			rgb255=(200.5595567867036,224.29362880886427,186.62049861495845);
			rgb255=(191.7229916897507,220.5180055401662,174.8116343490305);
			rgb255=(183.49584487534625,217.0027700831025,163.81717451523545);
			rgb255=(175.87811634349032,213.74792243767314,153.63711911357342);
			rgb255=(168.86980609418282,210.75346260387812,144.27146814404432);
			rgb255=(162.47091412742384,208.01939058171746,135.72022160664818);
			rgb255=(156.68144044321332,205.54570637119116,127.98337950138506);
			rgb255=(151.50138504155126,203.33240997229916,121.06094182825484);
			rgb255=(146.9307479224377,201.37950138504155,114.95290858725764);
			rgb255=(142.96952908587255,199.68698060941827,109.65927977839334);
			rgb255=(139.61772853185596,198.25484764542935,105.18005540166205);
			rgb255=(136.8753462603878,197.0831024930748,101.51523545706371);
			rgb255=(134.74238227146813,196.17174515235456,98.66481994459834);
			rgb255=(133.21883656509695,195.5207756232687,96.62880886426592);
			rgb255=(132.30470914127426,195.13019390581718,95.40720221606648);
			rgb255=(132.0,195.0,95.0);
		}
} 
\usepgfplotslibrary{patchplots}

\pgfplotsset{every axis/.append style={
			grid=both,
			grid style={white, line width=.1pt},
			major grid style={white, line width=1.5pt},
			axis background/.style={fill=gray!10},
			axis line style={draw=none},
			tick style={draw=none},
			xlabel = $x$,
			line width=1pt,
			legend style={
					line width = 1pt,
					draw=none,
					/tikz/every even column/.append style={column sep=0.5cm}
				},
		}}

\usepgfplotslibrary{groupplots}

\usepackage{calc}

\definecolor{gg0}{HTML}{E24A33}
\definecolor{gg1}{HTML}{348ABD}
\definecolor{gg2}{HTML}{988ED5}
\definecolor{gg3}{HTML}{777777}
\definecolor{gg4}{HTML}{FBC15E}
\definecolor{gg5}{HTML}{8EBA42}
\definecolor{gg6}{HTML}{FFB5B8}

\pgfplotsset{
	/pgfplots/colormap={bright}{rgb255=(0,0,0) rgb255=(78,3,100) rgb255=(2,74,255)
			rgb255=(255,21,181) rgb255=(255,113,26) rgb255=(147,213,114) rgb255=(230,255,0)
			rgb255=(255,255,255)}
}

\newcommand{\addappendix}{
	\section*{\appendixname}
	\addcontentsline{toc}{section}{\appendixname}
	\counterwithin*{figure}{section}
	\stepcounter{section}
	\renewcommand{\thesection}{A}
	\renewcommand{\thefigure}{\thesection.\arabic{figure}}
}

\definecolor{brandeisblue}{rgb}{0.0, 0.44, 1.0}
\definecolor{lincolngreen}{rgb}{0.11, 0.35, 0.02}
\definecolor{indiagreen}{rgb}{0.07, 0.53, 0.03}
\definecolor{venetianred}{rgb}{0.78, 0.03, 0.08}
\definecolor{darkorange}{rgb}{1.0, 0.55, 0.0}
\definecolor{burntorange}{rgb}{0.8, 0.33, 0.0}
\definecolor{flame}{rgb}{0.89, 0.35, 0.13}
\definecolor{non-photoblue}{rgb}{0.64, 0.87, 0.93}

\usepackage{tikz}
\usepackage{tkz-base}
\usetikzlibrary{calc}
\usepackage{tkz-euclide}

\renewcommand{\review}[2]{}
\renewcommand{\creview}[3]{}
\renewcommand{\ntcreview}[3]{}

\renewcommand{\tableofcontents}{}
\renewcommand{\listofreviews}{}

\definecolor{revisionColourOne}{RGB}{180,0,0}
\definecolor{revisionColourTwo}{RGB}{0,0,180}

\usepackage{fullpage}
\usepackage[compatibility=false]{caption}

\usepackage{setspace}
\begin{document}
\begin{singlespace}\maketitle\end{singlespace}
\begin{abstract}
We present a numerical method for the computation of the inverse of the fractional Laplacian, $(-\Delta)^{s}$, based on its spectral definition, using rational functions to approximate the fractional power $A^{-s}$ of a matrix $A$, for $0<s<1$.  The proposed numerical method is fast and accurate, benefiting from the fact that the matrix $A$ arises from a finite element approximation of the Laplacian $-\Delta$, which makes it applicable to a wide range of computational domains with potentially irregular shapes. We make use of state-of-the-art software to compute the best rational approximation of a fractional power. We analyze the convergence rate of our method and validate our findings through a series of numerical experiments with a range of exponents $s \in (0,1)$. Additionally, we apply the proposed numerical method to different evolution problems that involve the fractional Laplacian through an interaction potential: the fractional porous medium equation and the fractional Keller--Segel equation. We then investigate the accuracy of the resulting numerical method, focusing in particular on the accurate reproduction of qualitative properties of the associated analytical solutions to these partial differential equations.
\end{abstract}
 
\section{Introduction}\label{Sect:1}

The fractional Laplacian denoted, for $s \in (0,1)$, by $(-\Delta)^s$ is a nonlocal operator based on taking the $s$-th power of the classical Laplacian $-\Delta$. The fractional Laplacian has become a cornerstone in the modeling of phenomena that exhibit anomalous diffusion and nonlocal effects, such as long-jump diffusion; we refer the reader to \cite{Valdinoci2009, Bucur2016, Vazquez2017}. Unlike the classical Laplacian, the fractional Laplacian is a nonlocal operator, requiring information from the entire domain of definition of the function to which it is applied in order to evaluate its action to the function at any point of the domain. This nonlocal nature introduces specific theoretical and computational challenges, making its numerical approximation a thriving area of research. Over the years, several techniques for the numerical approximation of the fractional Laplacian have been developed, based on different mathematical definitions, see \cite{Kwasnicki2017, whatis2020, Bonito2018}. The choice of definition of the fractional Laplacian often dictates the numerical approach, as various definitions give rise to distinct computational requirements.
One of the most commonly used definitions of the fractional Laplacian involves its integral representation, that is 
\begin{equation} \label{FracLapInt}
(-\Delta)^s u(x) = C_{d, s} \,\textrm{P.V.}\! \int_{\mathbb{R}^{d}} \frac{u(x) - u(y)}{|x-y|^{d+2s}} \dy,\quad x \in \mathbb{R}^d,
\end{equation}
where $C_{d,s}$ is a positive real constant and $\textrm{P.V.}$ stands for the principal value of the integral over $\mathbb{R}^{d}$. 

For the integral representation of the fractional Laplacian, finite difference schemes approximate the principal-value integral by discretizing the associated singular convolution. These methods are simple to construct and are widely used on regular meshes over axiparallel domains.
In \cite{Huang2014} the authors proposed finite difference methods for spatially fractional diffusion equations, demonstrating robust convergence for smooth solutions. The authors of \cite{DelTeso2018, DelTeso2019} analyzed and developed further finite difference methods for fractional-order differential equations. 
These techniques benefit from their conceptual simplicity and ease of implementation, but they also have disadvantages, including limited flexibility on
nonuniform grids and
irregular domain geometries.

In the context of the integral representation \eqref{FracLapInt}, finite element methods (FEM) provide a variational framework for problems involving the fractional Laplacian. Through the use of suitable weak formulations, these techniques deliver accurate numerical solutions, particularly for irregular domains or problems whose solutions exhibit singular behavior near domain boundaries. These methods have been developed to handle singularities in the kernel, adaptive refinement of the computational mesh, and incorporate efficient quadrature schemes for numerical integration. In \cite{Acosta2017}, a FEM for the integral fractional Laplacian in bounded domains was developed, focusing on accurate truncation and error control. Subsequently, adaptive quadrature rules for finite element discretizations were introduced in \cite{Ainsworth2018}, ensuring efficient computation in irregular geometries, while \cite{Zhang2019} proposed a nonlocal FEM with adaptive mesh refinement to address the singularity of the kernel. These techniques retain all the advantages of FEM, however, because they rely on the integral representation of the fractional Laplacian, they still face issues related to their high computational and implementation complexity due to dense matrix representations and numerical integration. 

In order to avoid dealing with the singular kernel in the expression for the integral representation of the fractional Laplacian, the Caffarelli--Silvestre extension \cite{caffarelli2016fractional} reformulated the fractional Laplacian as a Dirichlet-to-Neumann map for an elliptic problem in a higher-dimensional domain. This formulation transforms the nonlocal problem in $\mathbb{R}^d$ into a local PDE in a higher-dimensional space $\mathbb{R}^{d+1}$, enabling the use of standard FEM techniques. Following this reformulation, the authors of \cite{Nochetto2015} developed a weighted FEM, providing rigorous error analysis and convergence rates. Although these methods have the disadvantage of requiring additional computational resources to solve a PDE in higher dimensions, as well as a sufficiently refined graded mesh, some of these issues were subsequently addressed in the work \cite{Schwab2019} and in the recent contribution \cite{Schwab2023}. 

For problems defined on a bounded domain $\Omega$, with a suitable (e.g. Dirichlet or Neumann) boundary condition imposed on $\partial \Omega$, an equally natural definition of the fractional Laplacian is based on its spectral definition, involving eigenvalues and eigenfunctions of the standard Laplacian. An appealing feature of it is that it avoids the presence of the singular kernel appearing in the integral representation of the fractional Laplacian. The precise definition will be given in the next section.

Several previous approaches in the literature have used explicit representations of the spectral fractional Laplacian involving the heat semigroup (cf., for example, \cite{Cusimano2018} for further details), while finite element discretizations have resulted in fractional powers $A^{s}$ of the matrix representation $A$ of the discrete Laplacian. Of previous contributions that focus on the latter approach, we highlight in particular \cite{bonito2015}, with subsequent developments discussed in \cite{bonito2017, bonito2021, bonito2022}. Motivated by these, in \cite{carrillosuli2024} a finite element method was developed and analyzed for a drift-diffusion PDE involving the spectral definition of the fractional Laplacian. Our main goal here is to focus on the development of an efficient and accurate implementation of that method through a fast algorithm, 
with a provable error bound.
The algorithm is based on rational approximation of functions and we mention \cite{Miroslav2023} for some recent work that has also followed this direction, for a different PDE model unrelated to the fractional Laplacian, which is the primary focus of the present work.

This technique has the advantage that it comes from a finite element approximation, it is capable of dealing with different domain geometries, and it is accurate, despite the fact that it does not require full knowledge of the spectrum of the matrix representation $A$ of the discrete Laplacian, as in \cite{carrillofronzoni2024}. 
The proposed numerical approximation can be easily generalized to the fractional power of any positive self-adjoint linear operator. These advantages provide a fast and accurate numerical approximation of the fractional Laplacian, which can then be used in the numerical solution of nonlocal evolutionary PDEs.

The outline of the paper is as follows. In the next section we introduce the relevant notation and our numerical method for the approximation of the spectral fractional Laplacian, based on rational functions. We describe the algorithm, together with its computational benefits. We also provide an error estimate for our method and we validate the predicted order of convergence through a series of numerical experiments. In Section 3 we apply our computational method to two nonlinear evolution problems of relevance in applications involving the fractional Laplacian: the fractional porous medium equation and the fractional Keller--Segel model. For these models we confirm some of the qualitative properties that are known to hold, and we further explore certain additional features.

\section{A new computational method for the spectral fractional Laplacian}

The aim of this section is to describe the numerical approximation that we use for the spectral fractional Laplacian, $(-\Delta)^s$, in the context of solving a fractional Poisson equation.

\subsection{Spectral fractional operator}
We begin by introducing the definitions and notational conventions used for the spectral fractional Laplacian  operator; the reader is referred to \cite{caffarelli2016fractional,antil2018} for further details. 

Suppose that $\Omega$ is a bounded open Lipschitz domain in $\mathbb{R}^d$.
Let us denote by $L^{2}_{\ast}(\Omega)$ and  $H^{1}_{\ast}(\Omega)$ the closed linear subspace of $L^{2}(\Omega)$ and $H^{1}(\Omega)$ with zero integral average on $\Omega$, respectively; that is, 
$$L^{2}_{\ast}(\Omega) := \bigg\{ v \in L^{2}(\Omega) \text{ such that } \int_{\Omega}v \dx = 0 \bigg\}\quad
\mbox{and}  
\quad
H^{1}_{\ast}(\Omega) := \bigg\{ v \in H^{1}(\Omega) \textrm{ such that } \int_{\Omega} v \dx= 0 \bigg\}.$$
Let $0<s<1$; we consider the following fractional Poisson equation on  $\Omega \subset \mathbb{R}^{d}$:
\begin{equation} \label{FracPoi}
(-\Delta)^{s} u = f \quad \textrm{in } \Omega,
\end{equation}
subject to a suitable boundary condition. The precise definition of $(-\Delta)^{s}$ will be stated below. We let $\mathcal{V}:=H^{1}_{0}(\Omega)$ in the case of a homogeneous Dirichlet boundary condition, $u=0$ on $\partial \Omega$, and $\mathcal{V}:=H^{1}_\ast(\Omega)$ in the case of a homogeneous Neumann boundary condition, $\partial_{n}u = 0$ on $\partial \Omega$. In addition, we shall assume that $f\in L^2(\Omega)$ in the case of a Dirichlet boundary condition and $f\in L^2_\ast(\Omega)$ in the case of a Neumann boundary condition. These assumptions ensure the existence of a unique solution $u$ (see \cite{antil2018}). 

We denote by $(\cdot, \cdot)$ the $L^{2}(\Omega)$ inner product and  by $\|\cdot\|_{L^{p}(\Omega)}$ and $\|\cdot\|_{H^{1}(\Omega)}$ the standard $L^{p}(\Omega)$ and $H^{1}(\Omega)$ norm, respectively, with $p \in [1,\infty]$.

The fractional Laplacian $(-\Delta)^{\s}$ in \eqref{FracPoi} is understood to be based on its spectral definition, using the eigenvalues and eigenfunctions of the Dirichlet Laplacian $-\Delta_{\mathrm{D}}$ or the Neumann Laplacian $-\Delta_{\mathrm{N}}$, depending on the boundary condition. Let $-\Delta_{\mathcal{B}}$ denote  either the Dirichlet Laplacian or the Neumann Laplacian.  Then, the pair $(\mathrm{Dom}(-\Delta_\mathcal{B}), -\Delta_{\mathcal{B}})$ is defined
as follows: 
\begin{align*}
 \mathrm{Dom}(-\Delta_{\mathcal{B}}) := \bigg\{ u \in \mathcal{V}\,:\, H^1(\Omega)\ni & \,v \mapsto \!\int_\Omega \nabla u \cdot \nabla v \dx\quad \mbox{is a continuous linear funcional on } L^2(\Omega)\bigg \},\\
 \int_\Omega (-\Delta_{\mathcal{B}} u)\, v &:= \int_\Omega \nabla u \cdot \nabla v \dx, \quad u \in \mathrm{Dom}(-\Delta_{\mathcal{B}}),\, \, v \in \mathcal{V}.
\end{align*}
Thus, $-\Delta_{\mathcal{B}}$ is a nonnegative and selfadjoint linear operator that is densely defined in $L^2(\Omega)$ in the case of the Dirichlet Laplacian, and in $L^2_\ast(\Omega)$ in the case of the Neumann Laplacian; thanks to the Hilbert--Schmidt theorem (cf., for example, Lemma 5.1 in \cite{FS}),
there exists an orthonormal basis of $L^2(\Omega)$, respectively $L^2_\ast(\Omega)$, consisting of eigenfunctions $\psi_k \in \mathcal{V}\setminus\{0\}$, $k = 1,2,\ldots$, with corresponding eigenvalues $0 <  \lambda_1 \leq \lambda_2 \leq \cdots \nearrow +\infty$. 
Thus, by writing  $\mathcal{B} (\psi_k) = \psi_{k}$ if we have a Dirichlet boundary condition and $\mathcal{B} (\psi_k) = \partial_{n} \psi_{k}$ for a Neumann boundary condition,
\begin{equation}\label{eigen}
    \left \{ \begin{aligned}
        -\Delta \psi_k &= \lambda_k \psi_k &&\quad \textrm{in } \Omega,  \\
        \mathcal{B} (\psi_k) &= 0 &&\quad\textrm{on } \partial \Omega, 
    \end{aligned} \right.
\end{equation}
with the partial differential equation in the first line of \eqref{eigen} understood as an equality in $L^2(\Omega)$, and the homogeneous boundary condition appearing in the second line of \eqref{eigen} as an equality in $H^{\frac{1}{2}}(\partial\Omega)$ in the case of a Dirichlet boundary condition, and as an equality in $H^{-\frac{1}{2}}(\partial\Omega)$ in the case of a Neumann boundary condition.
For $\s \in (0, 1)$ the domain of the spectral fractional Laplacian $(-\Delta_{\mathcal{B}})^s$ is defined as the Hilbert space
\begin{equation} \label{DirFracLap}
\Hs(\Omega) := \Bigg \{ u(\cdot) = \sum_{k=1}^{\infty} u_{k} \psi_{k}(\cdot) \in L^{2}(\Omega): \|u\|^{2}_{\Hs(\Omega)} := \sum_{k=1}^{\infty} \lambda_{k}^{\s} u_{k}^{2} < \infty \Bigg \} \quad \textrm{with } u_{k} := \int_{\Omega} u(x) \psi_{k}(x) \dx
\end{equation}
for $k=1,2, \ldots$ in the case of a homogeneous Dirichlet boundary condition, and 
\begin{equation} \label{NeuFracLap}
\Hs(\Omega) := \Bigg \{ u(\cdot) = \sum_{k=1}^{\infty} u_{k} \psi_{k}(\cdot) \in L^{2}_\ast(\Omega): \|u\|^{2}_{\Hs(\Omega)} := \sum_{k=1}^{\infty} \lambda_{k}^{\s} u_{k}^{2} < \infty \Bigg \} \quad \textrm{with } u_{k} := \int_{\Omega} u(x) \psi_{k}(x) \dx
\end{equation}
for $k=1,2,\dots$ if we are considering a homogeneous Neumann boundary condition. In the remainder of the paper, with the exception of the Appendix, we will use $\mathbb{H}^s(\Omega)$ as a unified notation for both of these cases; the distinction between them will be implicit in the type of Laplacian (i.e., Dirichlet or Neumann) that will be used and will be clear from the context. We denote by $\mathbb{H}^{-s}(\Omega)$ the corresponding dual space. 

For $u \in \mathbb{H}^s(\Omega)$, we have  
\begin{equation} \label{SpectralFracLap}
(-\Delta_{\mathcal{B}})^s u(x) : = \sum_{k=1}^\infty \lambda_k^s u_k \psi_k(x),
\end{equation}
as an element of $\mathbb{H}^{-s}(\Omega)$. Furthermore,  one has 
     $\|u \|_{\Hs(\Omega)} = \| (-\Delta_\mathcal{B})^{\s / 2} u\|_{L^{2}(\Omega)}$ for all $u \in \mathbb{H}^s(\Omega)$.

\subsection{Finite element approximation}

Hereafter, for the sake of simplicity, we shall suppose that $\Omega \subset \mathbb{R}^d$ is a bounded open Lipschitz polytope. Given a quasi-uniform and shape regular triangulation  $\mathcal{T}_{h} = \{ K_{n} \}_{n=1}^{M_{h}}$ of the domain, where $K_{n}$, $n = 1, \dots, M_{h}$, are closed simplices
with mutually disjoint interiors, such that $\overline{\Omega} = \cup_{n=1}^{M_{h}} K_{n}$, let $\V_{h}$ be the subspace of $\mathcal{V}$, consisting of all continuous piecewise affine functions defined on $\mathcal{T}_h$:
\begin{equation} \label{FEMspace}
    \mathcal{V}_{h} := \{ v_{h} \in C(\overline{\Omega}) \cap \mathcal{V} \textrm{ such that } v_{h}\big|_{K} \in \mathbb{P}^{1} \textrm{ for all } K \in \mathcal{T}_{h} \},
\end{equation}
and let $N_{h}$ be its dimension, $\dim \V_{h} =: N_{h}$. For this space, it is possible to define a nodal basis, which we denote by $\{ \phi_{i} \}_{i = 1}^{N_{h}}$, such that $\phi_{i}(P_{j}) = \delta_{i,j}$, where $\{ P_{j} \}_{j=1}^{N_{h}}$ are the mesh points (vertices) of $\mathcal{T}_{h}$ contained in $\Omega$ in the case of a homogeneous Dirichlet boundary condition, and in $\overline{\Omega}$ in the case of a homogeneous Neumann boundary condition. Moreover let $\mathcal{I}_{h}: \mathbb{R}^{N_{h}} \to \mathcal{V}_{h}$ be the interpolation operator defined as 
\begin{equation} \label{InterpOp} \mathcal{I}_{h} V = \sum_{i=1}^{N_{h}} V_{i} \phi_{i}, \quad \text{for } V = (V_{1}, \dots, V_{N_{h}})^{\mathrm{T}} \in \mathbb{R}^{N_{h}}, \end{equation}
which takes a vector $V$ in $\mathbb{R}^{N_{h}}$ and maps it into the function $\mathcal{I}_{h} V$ in $\mathcal{V}_{h}$.

Let us consider the eigenvalue problem for the Laplacian, stated in its weak form: 
\begin{equation} \label{EigenWeakFormMain} 
\text{Find } \varphi \in \V \setminus \{0\} \text{ and } \lambda \in \mathbb{R} \quad \text{such that} \quad (\nabla \varphi, \nabla v) = \lambda (\varphi, v) \quad \textrm{for all } v \in \V.\end{equation}

The finite element approximation of this eigenvalue problem is obtained by considering the restriction of the weak formulation (\ref{EigenWeakFormMain}) to the finite element space $\V_{h}$:
\begin{equation} \label{EigenProb1} \text{Find } \varphi_{h}  \in \V_{h}\setminus \{0\} \text{ and } \lambda^{h} \in \mathbb{R} \quad \text{such that} \quad(\nabla \varphi_{h}, \nabla v_{h}) = \lambda^{h} (\varphi_{h}, v_{h}) \quad \textrm{for all } v_{h} \in \V_{h}. \end{equation}

We define the \textit{discrete Laplacian} $ L_{h} : \V_{h} \rightarrow \V_{h}$
by the identity $(L_{h} w_{h}, v_{h}) := (\nabla w_{h}, \nabla v_{h})$ for all $v_{h} \in \V_{h}$. The linear operator $L_{h}$ is positive and selfadjoint, and has a set of  eigenvalues $0 <  \lambda_1^h \leq \lambda_2^h \leq \cdots \leq \lambda_{N_{h}}$, with corresponding eigenfunctions $\{ \varphi_{k}^{h} \}_{k=1}^{N_{h}}$, which form an orthonormal basis (with respect to the inner product $(\cdot,\cdot)$) for the space $\V_{h}$. This allows us to define a discrete version of the spectral fractional Laplacian
\begin{equation} \label{FracLapFE}
(-\Delta_{h})^{\s} u_{h} = \sum_{k = 1}^{N_{h}} (\lambda_{k}^{h})^{\s} u^{h}_{k}\varphi _{k}^{h}, \quad \textrm{where } u_{k}^{h} = \int_{\Omega} u_{h} \varphi_{k}^{h} \dx \quad \mbox{and} \quad s \in (0,1).
\end{equation}
With these definitions in place, the finite element approximation of the fractional Poisson equation is as follows:  
\begin{equation} \label{weakproblem} 
\textrm{Find } u_{h} \in \V_{h} \quad \textrm{such that} \quad ((-\Delta_{h})^{\s} u_{h}, v_{h}) = (f, v_{h}) \quad \textrm{ for all } v_{h} \in \V_{h}. 
\end{equation}
This can be rewritten in the following equivalent form: 
\[ \textrm{Find } u_{h} \in \V_{h} \quad \textrm{such that} \quad  \sum_{k=1}^{N_{h}} (\lambda_{k}^{h})^{\s} u^{h}_{k} (\varphi_{k}^{h}, v_{h}) = (f, v_{h}) \quad \textrm{ for all } v_{h} \in \V_{h}. \]
Since $\{ \varphi_{j}^{h} \}_{j=1}^{N_{h}}$ is an orthonormal basis for $\V_{h}$ with respect to the inner product $(\cdot,\cdot)$, we have that
\begin{equation} \label{EigenEq1} 
(\lambda_{j}^{h})^{\s} u^{h}_{j}   = (f, \varphi_{j}^{h}) \quad \textrm{ for all } j = 1, \dots, N_{h}. 
\end{equation}
Writing (\ref{EigenEq1}) as a linear system gives
\[ \Lambda^{\s} \widetilde{U} = \widetilde{F}, \quad \textrm{with} \quad \widetilde{U} = (u^{h}_{1}, \dots,  u^{h}_{N_{h}})^{\mathrm{T}}, \quad \widetilde{F} = ((f, \varphi_{1}^{h}), \dots, (f, \varphi_{N_{h}}^{h}))^{\mathrm{T}}, \quad \Lambda = \text{diag}(\lambda_{1}, \dots, \lambda_{N_{h}}). \]

Let $X$ be the matrix of change of variables between $\{ \varphi_{j}^{h} \}_{j =1}^{N_{h}}$ and the usual finite element basis $\{ \phi_{i} \}_{i =1}^{N_{h}}$ such that $\widetilde{U} = X U$ and $\widetilde{F} = X F$.
Then the linear system can be rewritten as 
\begin{equation} \label{EigenEq2}
X^{-1} \Lambda^{\s} X U = F.
\end{equation}
\begin{lemma} \label{EigenProbLemma}
Let $\{ \phi_{i} \}_{i = 1}^{ N_{h}}$ be the nodal basis of the space $\mathcal{V}_{h}$ defined in \eqref{FEMspace}, and let $\mathcal{S}$ and $\mathcal{M}$ be the stiffness matrix and mass matrix, respectively, defined as follows: $\mathcal{S}_{i,j} = (\nabla \phi_{i}, \nabla \phi_{j})$, $\mathcal{M}_{i,j} = (\phi_{i}, \phi_{j})$, $i, j = 1, \dots, N_h$. The matrices $\mathcal{S}$ and $\mathcal{M}$ are symmetric and positive definite, and the matrix $X^{-1} \Lambda^{s} X$ in \eqref{EigenEq2} is the $s$-th power of the matrix $A=\mathcal{M}^{-1} \mathcal{S}$.
The eigenvalues of the matrix $A$ are real and positive.
\begin{proof}
The symmetry of the matrices $\mathcal{S}$ and $\mathcal{M}$ is obvious. That these matrices are positive definite is also easy to prove: suppose that $\xi \in \mathbb{R}^{N_h}\setminus\{0\}$ and let $v(x):= \sum_{i=1}^{N_h} \xi_i \phi_i(x)$ for $x \in \overline\Omega$. Then, $\xi^{\mathrm{T}} \mathcal{S} \xi = \|\nabla v\|^2_{L^2(\Omega)}$. Thanks to the classical Poincar\'{e} inequality in $H^1_0(\Omega)$ in the case of $\mathcal{V}_h \subset H^1_0(\Omega)$ and by the Poincar\'{e}--Wirtinger inequality in $H^1_\ast(\Omega)$ in the case of
$\mathcal{V}_h \subset H^1_\ast(\Omega)$, there exists a positive constant $C=C(\Omega,d)$ such that $\|\nabla v\|^2_{L^2(\Omega)} \geq C\|v\|^2_{L^2(\Omega)} > 0$, because $\xi \neq 0$ and the functions $\{\phi\}_{i=1}^{N_h}$ are linearly independent. Hence $\mathcal{S}$ is positive definite. The positive definiteness of the matrix $\mathcal{M}$
follows even more directly: $\xi^{\mathrm{T}}\mathcal{M}\xi = \|v\|^2_{L^2(\Omega)}>0$, again because $\xi \neq 0$ and the functions $\{\phi\}_{i=1}^{N_h}$ are linearly independent. In particular, $\mathcal{S}$ and $\mathcal{M}$ are invertible matrices. 

Let us restate (\ref{EigenProb1}) by expressing the eigenfunction $\varphi_h$ in terms of the nodal basis: $\varphi_{h} = \sum_{i=1}^{N_{h}} 
(\varphi_{h})_{i} \phi_{i}$, with $(\varphi_h)_i \in \mathbb{R}$ for $i=1,\dots, N_h$. The eigenproblem reduces to  the equality $\sum_{i=1}^{N_{h}} (\nabla \phi_{i}, \nabla \phi_{j}) (\varphi_{h})_{i} = \lambda^{h} \sum_{i=1}^{N_{h}} (\phi_{i}, \phi_{j}) (\varphi_{h})_{i}$, 
which can be rewritten in matrix form as the generalized algebraic  eigenvalue problem 
\begin{equation} \label{EigenEq3}
    \mathcal{S} \Phi = \lambda^{h} \mathcal{M} \Phi, 
\end{equation}
where $\Phi = ((\varphi_{h})_{1}, \dots,  (\varphi_{h})_{N_{h}})^{\mathrm{T}} \in \mathbb{R}^{N_h}$. Equation (\ref{EigenEq3}) implies that $\lambda^h_j$ is an eigenvalue of the matrix $\mathcal{M}^{-1}\mathcal{S}$
with associated eigenvector $\Phi_j$, $j=1,\dots,N_h$. Therefore, the matrix $X$ introduced prior to equation \eqref{EigenEq2} has the form $X=[\Phi_1, \dots, \Phi_{N_h}] \in \mathbb{R}^{N_h \times N_h}$. Consequently, the matrix $X^{-1} \Lambda^{s} X$ is the $s$-th power of $A=\mathcal{M}^{-1}\mathcal{S}$. Finally, because the symmetric matrices $\mathcal{S} \in \mathbb{R}^{N_h \times N_h}$ and $\mathcal{M}\in \mathbb{R}^{N_h \times N_h}$ are positive definite, it follows that the eigenvalues $\lambda^h$ of the matrix $A$ are real and positive. 
\end{proof}
\end{lemma}

\subsection{Rational approximation and fractional powers of a matrix} According to 
Lemma \ref{EigenProbLemma} our finite element approximation of the fractional Poisson equation involves the computation of the power of a nonsingular matrix $A$. The use of this method tends to be avoided in practice because the computational cost to obtain an approximation of the power of a matrix is usually high for large matrices, and requires knowledge of the spectral properties of the matrix itself: as the dimension of the matrix increases, one needs to compute an increasing number of eigenvalues and eigenvectors to improve the accuracy of the method, which is, for very large matrices, practically infeasible.

The aim of our work in this context is to develop a computationally efficient procedure to approximate fractional powers of a nonsingular matrix $A$, with positive eigenvalues;
 more precisely, our focus is the approximation of the matrix $A^{-s}$, to directly solve the linear system that arises from the finite element approximation of the fractional Poisson equation. We address this challenge with the idea of considering the minimax algorithm introduced in \cite{NakatsukasaTrefethen2018}. This algorithm was developed to compute the best rational approximation of a given function in one variable, on a given interval, through adaptive barycentric representations. 

We briefly outline the method, and we refer to \cite{NakatsukasaTrefethen2018} for further details. Let 
\begin{equation} \label{RationalFuncs}
\mathcal{R}_{m,n} := \bigg \{ \frac{p}{q}: p \in \mathbb{P}_{m}[x], ~~ q \in \mathbb{P}_{n}[x]\bigg \}
\end{equation}
denote the set of rational functions of type $(m,n)$, where $\mathbb{P}_{n}[x]$ denotes the set of all polynomials in $x$ of degree at most $n$. Given a continuous real-valued function $g$ defined on a closed interval $[a,b] \subset \mathbb{R}$, the algorithm computes the solution of the problem 
\begin{equation} \label{BestRatApprox} \inf_{r \in \mathcal{R}_{m,n}} \| g - r \|_{\infty, [a,b]}, \end{equation} 
where $\| \cdot \|_{\infty, [a,b]}$ denotes the maximum norm over $[a, b]$, i.e., $\| g - r \|_{\infty, [a,b]} = \max_{x \in [a,b]}|g(x) - r(x)|$. We denote by $r_{m, n}^{\ast}$ the solution of the problem, called the \textit{best rational approximation}.

For the sake of simplicity, we shall restrict our attention here to the case $m = n$. Working in this setting will enable us in Section \ref{SectionError} 
to make direct use of a sharp approximation result due to Stahl \cite{stahl2003}, stated in Remark \ref{ErrEstRem}. 
We shall seek the best rational approximation $r$ of type $(n,n)$ to $g$ in its partial fraction representation, i.e., as
\begin{equation*}
r(x) = R_{0} + \frac{R_{1}}{ x- t_{1}} + \dots + \frac{R_{n}}{x - t_{n}},
\end{equation*}
where $t_{1}, \dots, t_{n}$ are the poles of $r$, and the numbers $R_0, R_1, \dots, R_n$ are referred to as the residues of $r$.
The minimax algorithm from \cite{NakatsukasaTrefethen2018} provides the user with the residues and poles of the best rational approximation. 

The method is then suitable for the approximation of the matrix function $M \in \mathbb{R}^{N \times N} \mapsto g(M) \in \mathbb{R}^{N \times N}$ using the corresponding expression
\begin{equation} \label{MatPowApprox}
r(M) = R_{0} I + R_{1} (M - t_{1}I)^{-1} + \dots + R_{n} (M - t_{n}I)^{-1}.
\end{equation}
We will approximate here $g(x) = x^{-s}$ by a rational function $r \in \mathcal{R}_{n,n}$, which will then provide the desired matrix approximation to the fractional power $A^{-s}$ of the nonsingular matrix $A$ with positive eigenvalues  under consideration.

\begin{remark}[Efficiency: computational cost]
Once we have the residues and poles of the rational approximation $r$, and thereby the rational approximation $r(A)$ of the matrix $A^{-s}$, it is then straightforward to compute an approximation to the vector $U=g(A)F$ in (\ref{EigenEq2}) as 
\begin{equation} \label{SolEigenEq}
r(A) F = R_{0} F + R_{1} (A - t_{1}I)^{-1} F + \cdots +   R_{n} (A - t_{n}I)^{-1} F.
\end{equation}
The terms $T_i:= R_i (A-t_iI)^{-1}F$ appearing on the right-hand side of  \eqref{SolEigenEq} can be efficiently computed by solving systems of linear algebraic equations of the form $(A-t_i I) T_i = R_i F$, $i=1,\ldots, n$.  In the present context, $A=\mathcal{M}^{-1}\mathcal{S}$ (cf. Lemma \ref{EigenProbLemma}).
From a computational perspective, we observe that applying mass lumping to the symmetric positive definite mass matrix $\mathcal{M}$, thus replacing $\mathcal{M}$ with a positive definite diagonal matrix $\hat{\mathcal{M}}$, results in a symmetric positive definite matrix $\hat{A}=\hat{\mathcal{M}}^{-1} \mathcal{S}$, which has the same sparsity structure as the stiffness matrix $\mathcal{S}$. As a result, the symmetric matrices $(\hat{A}-t_i I)$ retain the sparsity structure of $\mathcal{S}$, and this enhances the efficiency of solving the associated linear systems when using a finite element nodal basis. The degree $n$ of the rational approximation $r$ is typically a small number (common choices are $n=11, 12, 13$) to ensure sufficient accuracy. Thus, once the minimax algorithm has provided us with the residues and poles, the overall complexity of computing the approximation $r(A)F$ is $\mathcal{O}(n C(N))$, where $C(N)$ is the cost of solving the linear systems specified above.
\end{remark}

\begin{remark}[Efficiency: knowledge of the spectrum]
    We note that, thanks to the nature of the minimax algorithm developed in \cite{NakatsukasaTrefethen2018} to compute the best rational approximation to the function $x^{-s}$,  our method does not require the spectrum of the discrete Laplacian $-\Delta_{h}$. The algorithm only needs the endpoints $x_{min}, x_{max}$ of the interval $[x_{min}, x_{max}]$ on which the user aims to build the approximation. This is of significant benefit in computing $r(A)$,
    since we only require estimates of the smallest and largest eigenvalues of the matrix $A$, and these can be easily calculated using standard algorithms from numerical linear algebra. 
\end{remark}

\begin{remark}[Error estimate] \label{ErrEstRem}
     It was proved by Stahl \cite{stahl2003} that the error of best rational approximation exhibits root-exponential decay with respect to the degree of the rational approximation. More precisely, letting $r^{\ast}_{n,n}(x)$ denote the best rational approximation to $x^{s}$, $s > 0$, with degree $n \in \mathbb{N}$, Stahl proved that,   defining $ E_{n, n}(x^{s}, [0,1]) \coloneqq \| x^{s} -r^{\ast}_{n,n}(x) \|_{\infty,[0,1]} = \inf_{r \in \mathcal{R}_{n, n}} \| x^{s} - r(x) \|_{\infty,[0,1]}$, one has that 
\begin{equation} \label{StahlEstimate}
\lim_{n \to \infty} \mathrm{e}^{2 \pi \sqrt{s n}} E_{n, n}(x^{s}, [0,1]) = 4^{1+s} |\sin(\pi s)|. 
\end{equation}

\end{remark}

\subsection{Error analysis} \label{SectionError}

Our aim in this section is to provide an error estimate for the numerical method we have presented, and to validate the estimate through a series of numerical experiments. Let us consider the error $\|u^{\ast} - u_{h}\|_{L^{2}(\Omega)}$ in the $L^{2}(\Omega)$ norm, where $u^{\ast}$ is the exact solution to the fractional Poisson equation 
\[ (-\Delta_{\mathcal{B}})^{\s} u^{\ast}= f \quad \text{ in } \Omega\]
subject to a homogeneous Dirichlet or homogeneous Neumann boundary condition on $\partial \Omega$, 
and let $u_{h}$ be the numerical solution obtained by our method, based on combining the finite element approximation of $u^\ast$ with a rational approximation $r^\ast(A)$ of $A^{-s}$ where $A:=\mathcal{M}^{-1}\mathcal{S}$, i.e., 
\begin{equation} \label{BestRatSolEq} u_{h} := \mathcal{I}_{h} (r^{\ast}(\mathcal{M}^{-1}\mathcal{S}) F) = \mathcal{I}_{h} (r^{\ast}(A) F) ,
\end{equation}
where $\mathcal{I}_h$ is the interpolation operator defined in \eqref{InterpOp} and $F$ is the vector defined in \eqref{EigenEq2}. 
Since $r^{\ast}(A)$ is an approximation of the inverse power matrix $A^{-\s}$, we split the total error $u^{\ast} - u_{h}$ into two parts, given the two stages of our approximation to $u^\ast$; thus, 
\begin{equation} \label{boundmain}
\|u^{\ast} - u_{h}\|_{L^{2}(\Omega)} \leq \|u^{\ast} - \I_{h}(A^{-\s}F)\|_{L^{2}(\Omega)} + \|\I_{h}(A^{-\s}F) - \I_{h}(r^{\ast}(A)F)\|_{L^{2}(\Omega)} := E_{\textrm{FEM}} + E_{\textrm{MIM}}.
\end{equation}
The first term $E_{\textrm{FEM}}$ exhibits the order of convergence in the limit of $h \to 0_+$ of the finite element approximation $(-\Delta_h)^s$ to the fractional Laplacian, and is governed by \cite[Theorem 4.3]{bonito2015} and subsequent generalizations (see  \cite[Theorem 6.2]{bonito2017} and \cite{bonito2022}). 

Consider the second term $E_{\textrm{MIM}}$. Since $\I_{h}(A^{-\s}F) - \I_{h}(r^{\ast}(A)F) \in \mathcal{V}_h$, by norm-equivalence in finite-dimensional normed linear spaces, we can pass from the $L^{2}(\Omega)$ norm to the Euclidean 2-norm $\| \cdot \|_{2}$, up to a factor $C(h, d)$, depending on the mesh-size $h$ and the dimension $d$ of the computational domain $\Omega$. Therefore, we have 
\begin{align}
E_{\textrm{MIM}} &\leq C(h, d) \| A^{-\s}F - r^{\ast}(A) F\|_{2} \leq C(h, d) \|A^{-\s} - r^{\ast}(A)\|_{2} \; \|F\|_{2} \nonumber \\
& \leq C(h, d)  \|X \Lambda^{-\s} X^{-1} - X  r^{\ast}(\Lambda)X^{-1}\|_{2} \; \|F\|_{2} \leq C(h, d)   \| X\|_{2}  \|X^{-1}\|_{2} \| \Lambda^{-\s} - r^{\ast}(\Lambda)\|_{2} \; \|F\|_{2} \nonumber \\
 &= C(h, d) \kappa_{2}(X)  \left( \max_{1 \leq i \leq N_{h}} |(\lambda_{i}^{h})^{-\s} - r^{\ast}((\lambda_{i}^{h}))|  \right) \; \|F\|_{2}, \label{eq:ActualIntervalErrorMIM}
\end{align}
where $\kappa_2(X)=: \| X\|_{2}  \|X^{-1}\|_{2}$ is the condition number of the matrix $X$, and 
$\{ \lambda_{k}^{h} \}_{k=1}^{N_{h}}$ are the eigenvalues of the matrix $A = \mathcal{M}^{-1}\mathcal{A}$ (cf. Lemma \ref{EigenProbLemma}); 
when applied to a matrix, $\|\cdot\|_2$ signifies the matrix 2-norm. Since the discrete spectral fractional Laplacian $-\Delta_{h}$ is a positive self-adjoint operator, we have $0 < \lambda_1^h \leq \lambda_2^h \leq \cdots \leq \lambda_{N_{h}}$, and $N_{h} \to \infty$ as $h \to 0$. It is known that $\lambda_{1}^{h} \to \lambda_{1}$ as $h \to 0$. 

We shall now use the estimate \eqref{StahlEstimate} from Remark \ref{ErrEstRem}. By rescaling the computational domain $\Omega$, if necessary, one can assume without loss of generality that $\lambda_1>1$ and $\lambda_{1}^{h} > 1$ for all $h>0$; hence,  
\begin{align*}
    E_{\textrm{MIM}} &\leq C(h, d) \kappa_{2}(X) \| x^{-s} - r^{\ast}_{n, n}(x) \|_{[1, \infty]} = C(h, d) \kappa_{2}(X) \inf_{r \in \mathcal{R}_{n,n}} \| x^{-s} - r(x)\|_{[1, \infty]}. 
\end{align*}
By using the change of variable $y = x^{-1}$ it then follows that  
\begin{align} \label{StahlEq1}
    E_{\textrm{MIM}} &\leq C(h, d) \kappa_{2}(X) \inf_{r \in \mathcal{R}_{n,n}} \| y^{s} - r(y^{-1}) \|_{[0, 1]} = C(h,d) \kappa_{2}(X) \inf_{r \in \mathcal{R}_{n,n}} \| y^{s} - \widehat{r}(y) \|_{[0,1]} \nonumber \\
    &= C(h, d) \kappa_{2}(X) \inf_{r \in \mathcal{R}_{n,n}} \| y^{s} - r(y)\|_{[0, 1]},
\end{align}
where in the last equality we have used the fact that $\widehat{r}(y):=r(y^{-1})$ is still a rational function of degree $m=n$ if $r$ is a rational function of degree $m=n$; 
indeed, thanks to the definition \eqref{RationalFuncs}, we have $ \{ \widehat{r}(y): r \in \mathcal{R}_{n,n} \}= \mathcal{R}_{n,n}$. 
Using \eqref{StahlEstimate} we finally have that 
\begin{equation} \label{ErrorEstimate2}
    E_{\textrm{MIM}} \leq 4^{1+s} |\sin(\pi s)| C(h, d) \kappa_{2}(X)\, \mathrm{e} ^{-2 \pi \sqrt{sn}} .
\end{equation}

This result, together with the bound $E_{\mathrm{FEM}}$ on the error committed in the finite element approximation, gives us control of the $L^{2}(\Omega)$ norm error appearing on the left-hand side of the inequality \eqref{boundmain}. The user can thereby prescribe the degree $n$ of the rational approximation and the mesh size $h$ to achieve the desired overall accuracy. Once the mesh size $h$ is fixed, thanks to the root-exponential decay with respect to $n$ of the final factor on the right-hand side of the inequality \eqref{ErrorEstimate2}, it is possible to choose the optimal degree $n$ for the rational approximation to obtain the same level of accuracy in $E_{\textrm{MIM}} $ as in $E_{\textrm{FEM}}$ and compensate for the presence of the factor $C(h, d)$. 

We now test numerically the convergence orders of the two sources of error in \eqref{boundmain}, $E_{\textrm{FEM}}$ and $E_{\textrm{MIM}}$, with our convergence results reported in Figures \ref{FigConv1}, \ref{FigConv2}, \ref{FigConv3}. In particular, in Figures \ref{FigConv1} and \ref{FigConv2} we assess in one and two space dimensions, respectively, the order of convergence of our method as $h \to 0$, governed by the term $E_{\textrm{FEM}}$. The derivation of a bound on the error $E_\textrm{FEM}$ has been studied in recent years, and its rate of convergence in the limit of $h \to 0$ is known (see Theorem 4.3 and Remark 4.1 in \cite{bonito2015}). In particular for $f \in \mathbb{H}^{\frac{1}{2} - \epsilon}(\Omega)$ for small $\epsilon > 0$ 
the authors of \cite{bonito2015} proved that the FEM error $E_{\textrm{FEM}}$ is bounded above by $C_{h} h^{2r}$,
where 
\begin{equation} \label{eq:predictEFEM}
C_{h} = \left \{ \begin{array}{ll} C & \textrm{if } \s > \frac{3}{4}, \\ C \log(1/h) & \textrm{otherwise}, \end{array} \right. \qquad r = \left \{ \begin{array}{ll} 1 & \textrm{if } \s > \frac{3}{4}, \\ \s + \frac{1}{4}  & \textrm{otherwise}. \end{array} \right. 
\end{equation}

In Figure \ref{FigConv3} we validate our estimate of the error $E_{\textrm{MIM}}$ associated with the rational approximation and report the order of convergence with respect to the degree $n$ of the best rational approximation \eqref{BestRatApprox}. 

\begin{table}
\caption{The fractional Poisson equation (\ref{FracPoi}) in $\Omega$ with homogeneous Dirichlet boundary condition; study of the convergence order as a function of the mesh size $h$ for $f \equiv 1$. We compare the observed order of convergence with the predicted order of convergence, as stated in Theorem 4.3  and Remark 4.1 in \cite{bonito2015}. Here, by \textit{order of convergence} we mean the exponent $2r$ in $C_h h^{2r}$, ignoring the presence of the factor $\log(1/h)$ in $C_h$ when $s \leq 3/4$; cf. \eqref{eq:predictEFEM}.}
\label{Table1}
\begin{subtable}{\linewidth}
\[
\begin{array}{cccccccc|ccc}
\toprule
& \multicolumn{7}{c|}{\s < \frac{3}{4}} & \multicolumn{3}{c}{\s > \frac{3}{4}} \\ 
\midrule & 0.02 & 0.05 & 0.1 & 0.1\bar{6} & 0.25 & 0.5 & 0.\bar{6} & 0.8\bar{3} & 0.8571 & 0.9 \\
\midrule 
\textrm{1D $\Omega=(0,1)$} \\
\textrm{observed} & 0.5400 & 0.6000 & 0.6600 & 0.8334 & 1.0003 & 1.5033 & 1.7922 & 1.9227 & 2.0153 & 2.0664 \\
\midrule
\textrm{2D $\Omega=(0,1)^2$} \\
\textrm{observed} & 0.5214 & 0.5816 & 0.6821 & 0.8162 & 0.9843 & 1.4912 & 1.7988 & 1.9598 & 1.9752 & 1.9804 \\
\midrule 
\textrm{predicted} & 0.54 & 0.6 &  0.7 & 0.8\bar{3} & 1.0 & 1.5 & 1.8\bar{3} & 2.0  & 2.0 & 2.0 \\
\bottomrule
\end{array}
\]
\end{subtable} 
\end{table}

\begin{figure}[H]

\centering
	\begin{subfigure}{0.49\textwidth}	
\includegraphics[width=\textwidth]{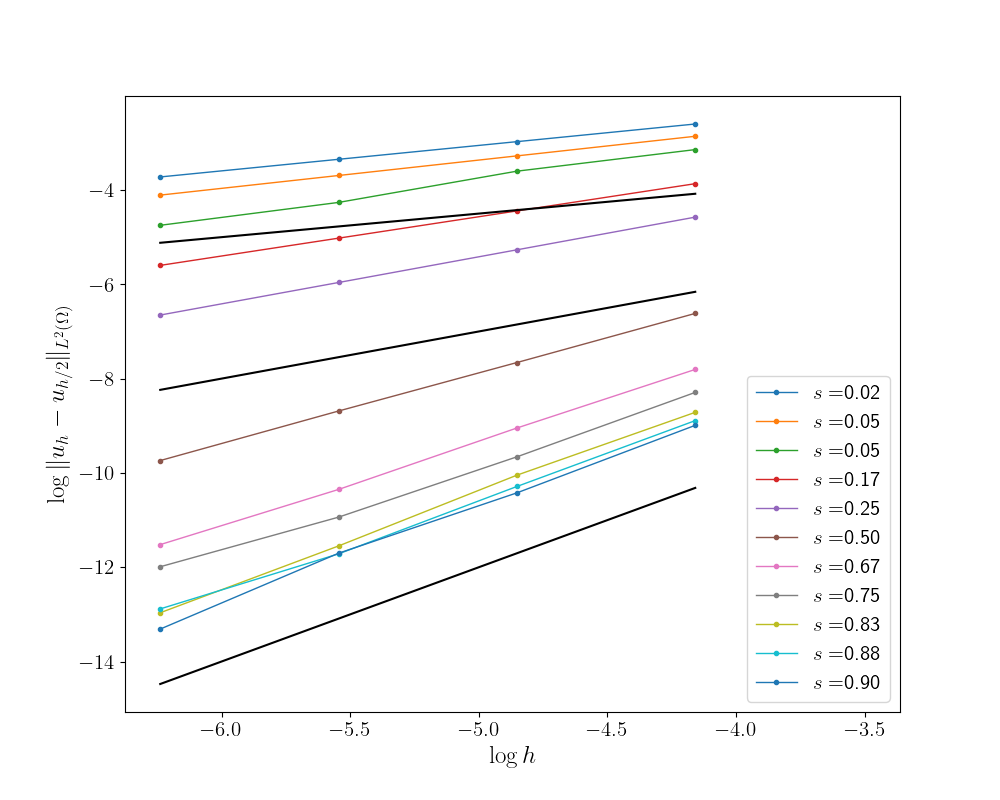}
\caption{One-dimensional case $\Omega=(0,1)$. Mesh sizes $h=1/2^{k}$ with $5 \leq k \leq 9$ were used. } \label{FigConv1}
\end{subfigure}
\hfill
\begin{subfigure}{0.49\textwidth}	
\includegraphics[width = \textwidth]{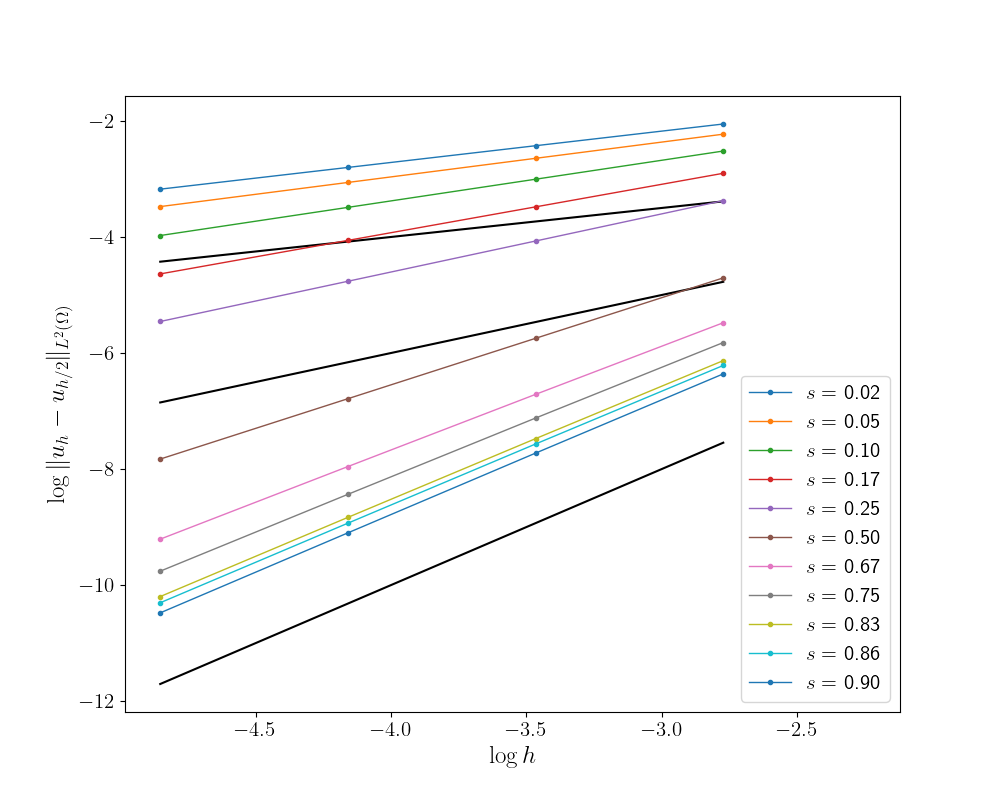}
\caption{Two-dimensional case $\Omega=(0,1)^2$. Mesh sizes $h=1/2^{k}$ with $3 \leq k \leq 7$ were used.}
\label{FigConv2}
\end{subfigure}
\caption{The fractional Poisson equation (\ref{FracPoi}) in $\Omega$ with homogeneous Dirichlet boundary condition; study of the convergence order as a function of the mesh size $h$. We report the convergence of the scheme for $f\equiv 1$. The black reference lines have slope, respectively, 2, 1.5, and 1.}
\end{figure}

\begin{figure}
	\centering

	\begin{subfigure}{0.3\textwidth}			\includegraphics[width=\textwidth]{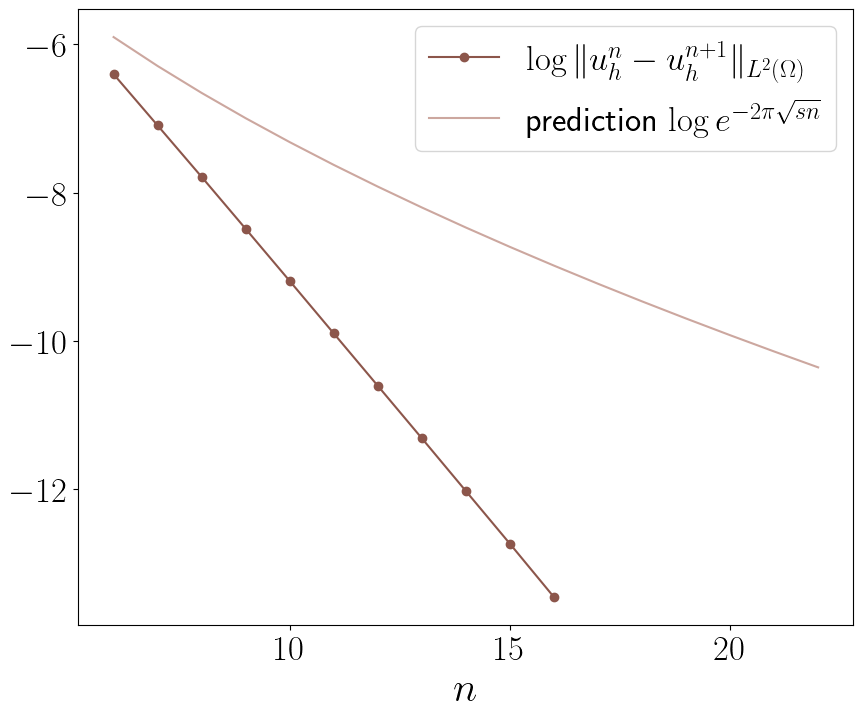}
		\caption{$s=0.1$}
	\end{subfigure}
	\hfill
	\begin{subfigure}{0.3\textwidth}			\includegraphics[width=\textwidth]{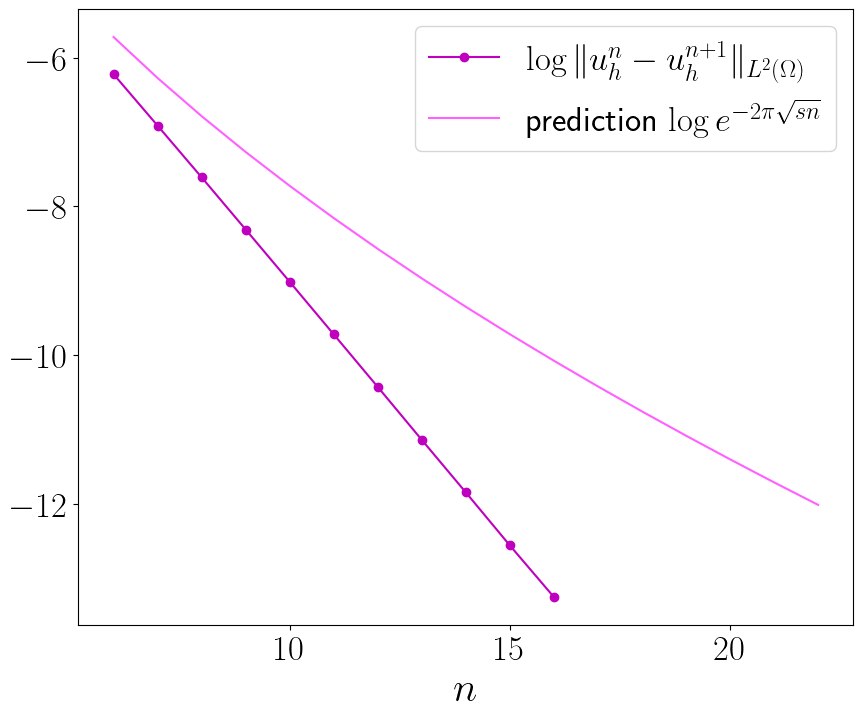}
		\caption{$s=0.2$}
	\end{subfigure}
        \hfill
 	\begin{subfigure}{0.3\textwidth}			\includegraphics[width=\textwidth]{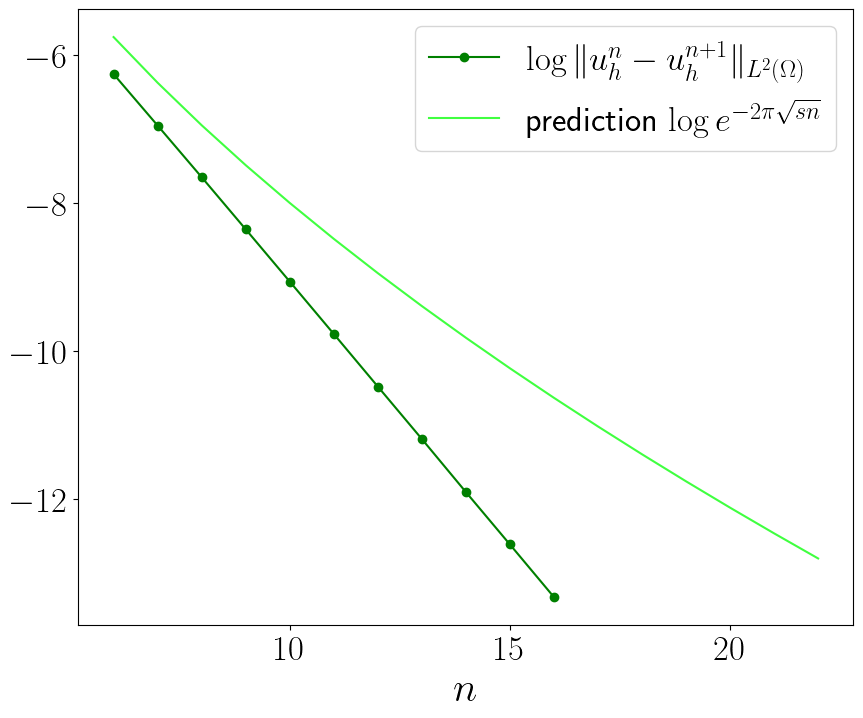}
		\caption{$s=0.25$}
	\end{subfigure}

	\begin{subfigure}{0.3\textwidth}			\includegraphics[width=\textwidth]{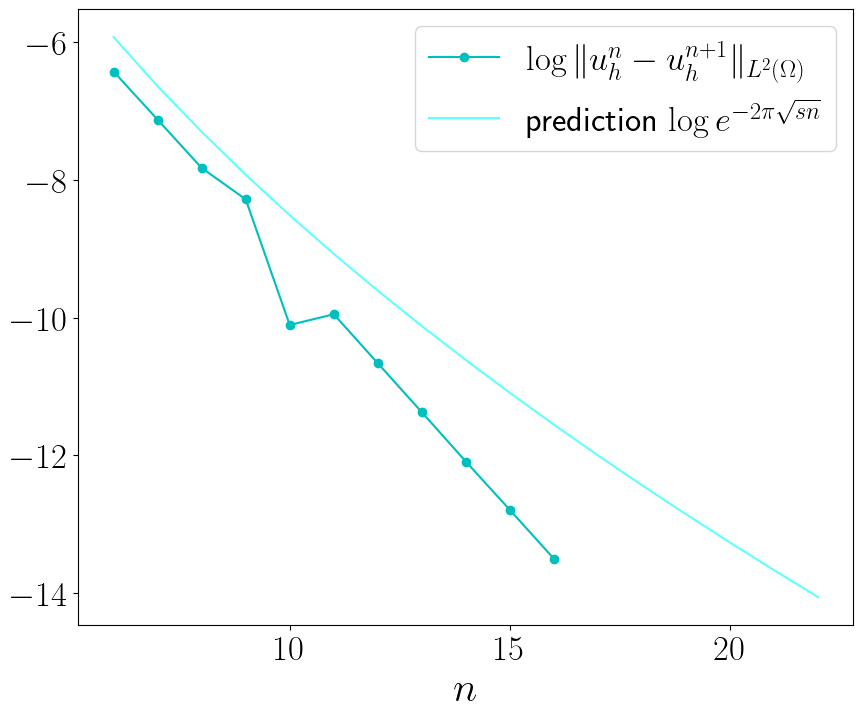}
		\caption{$s=0.\overline{3}$}
	\end{subfigure}
        \hfill
  	\begin{subfigure}{0.3\textwidth}			\includegraphics[width=\textwidth]{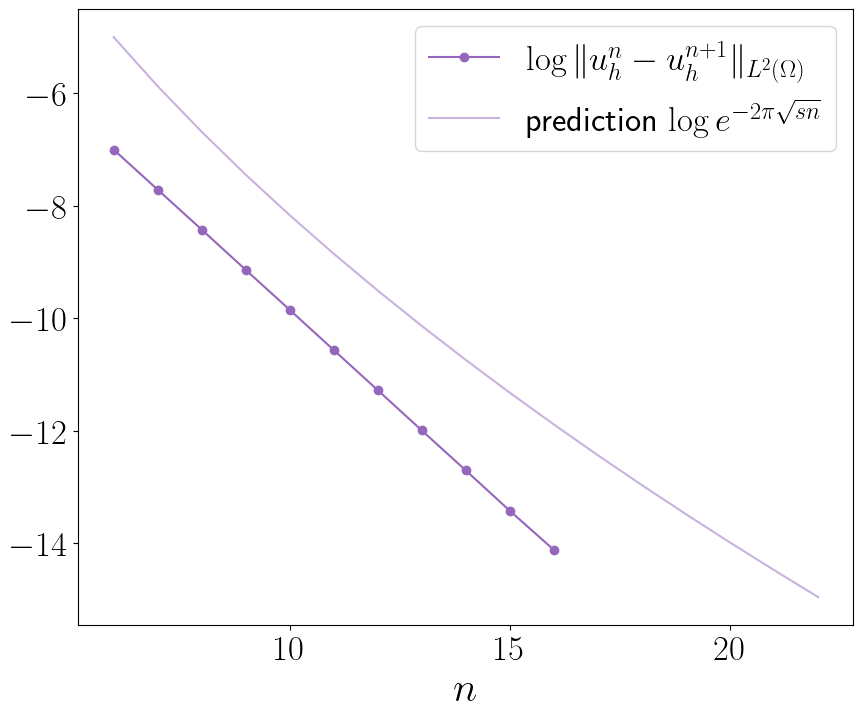}
		\caption{$s=0.5$}
	\end{subfigure}
	\hfill
  	\begin{subfigure}{0.3\textwidth}			\includegraphics[width=\textwidth]{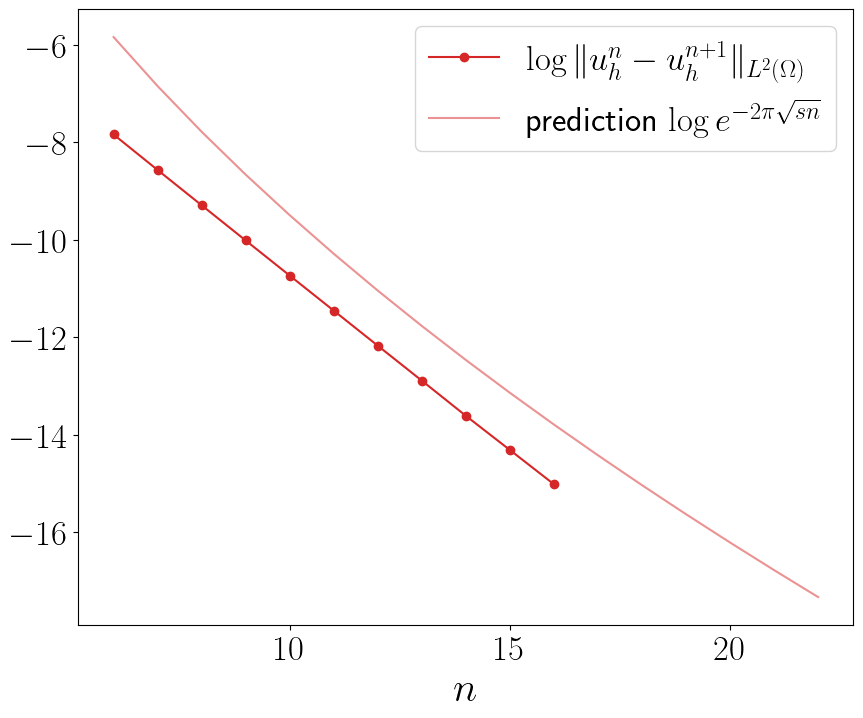}
		\caption{$s=0.\overline{6}$}
	\end{subfigure}

  	\begin{subfigure}{0.3\textwidth}			\includegraphics[width=\textwidth]{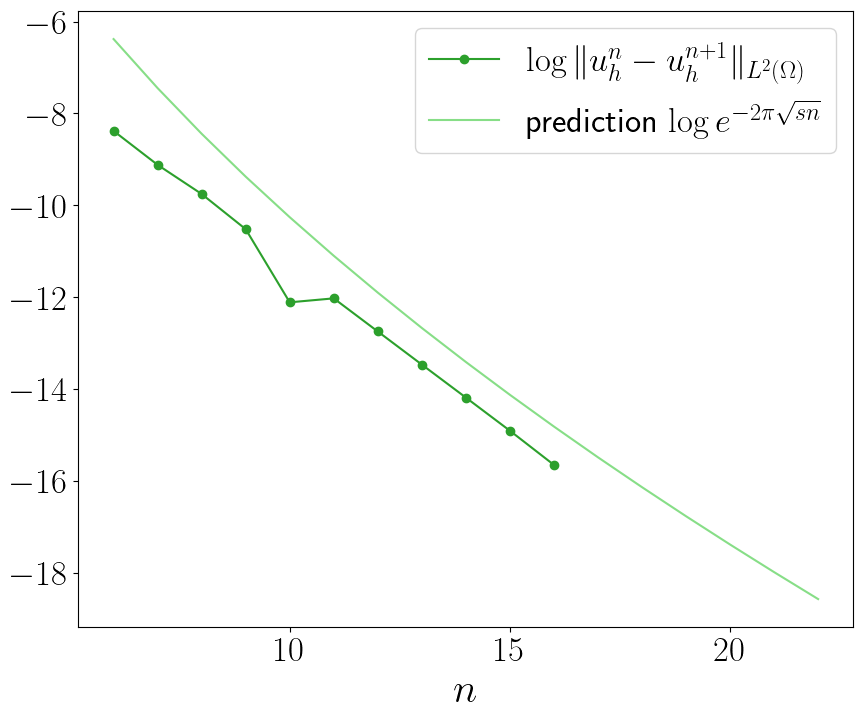}
		\caption{$s=0.75$}
	\end{subfigure}
	\hfill
  	\begin{subfigure}{0.3\textwidth}			\includegraphics[width=\textwidth]{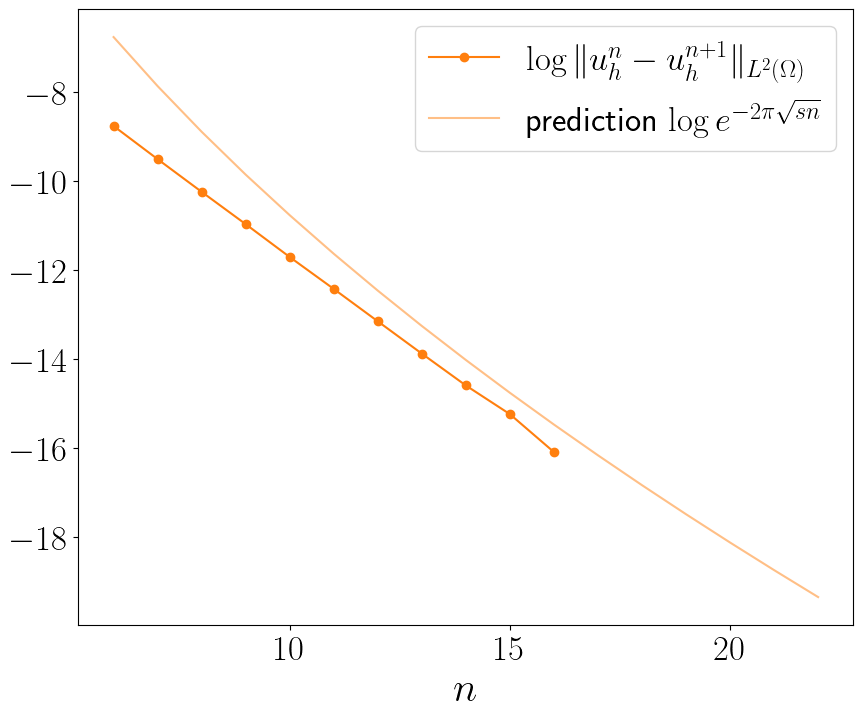}
		\caption{$s=0.8$}
	\end{subfigure}
        \hfill
 	\begin{subfigure}{0.3\textwidth}			\includegraphics[width=\textwidth]{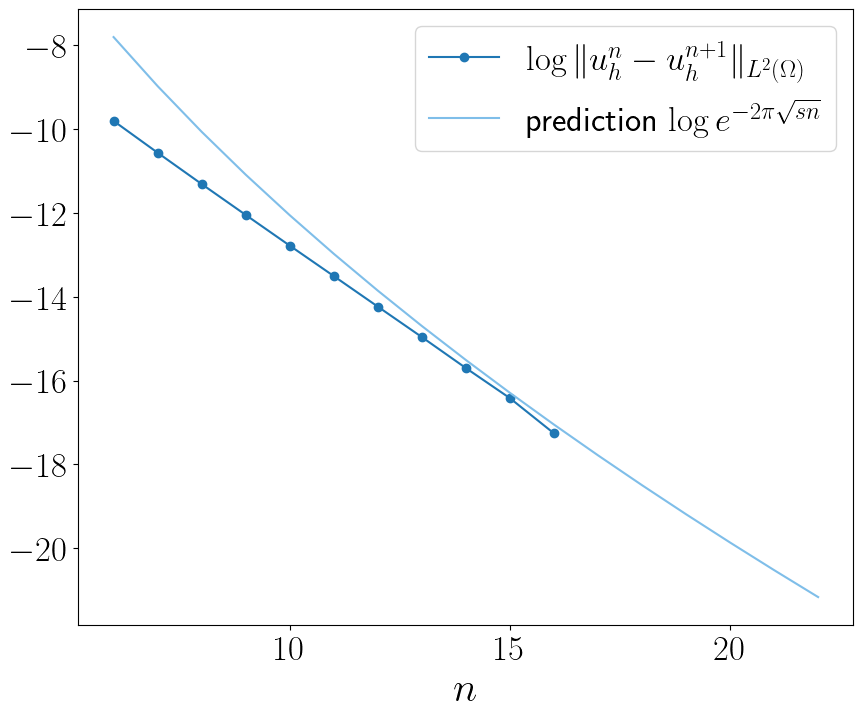}
		\caption{$s=0.9$}
	\end{subfigure}
\caption{The fractional Poisson equation: the observed exponential converge of the best rational approximation \eqref{BestRatApprox}, used in \eqref{SolEigenEq}. The computation of the function $u_h$ defined in  \eqref{BestRatSolEq} was performed for different fractional orders to assess the error as a function of the degree $n$ of rational approximation, in comparison with the theoretical prediction \eqref{ErrorEstimate2}. Each solution was computed on $\Omega=(0,1)^2$ with $f(x,y)=\sin(\pi x) \sin(\pi y)$, and with a mesh of $2^8 \times 2^8$ equally spaced mesh point. }
\label{FigConv3}
\end{figure}

We observe that, with the mesh size $h>0$ fixed, the order of convergence in the limit of $n \to \infty$ is actually exponential rather than root-exponential as predicted by 
\eqref{BestRatApprox}. 
This discrepancy is due to the difference in the domains over which the errors are evaluated. Specifically, the estimate 
\eqref{ErrorEstimate2} was based on bounding the sup-norm over the whole interval $y \in [0,1]$, while, as can be seen from \eqref{eq:ActualIntervalErrorMIM}, our bound on the actual error $E_{\textrm{MIM}}$ only involves the sup-norm over the subinterval $y \in [(\lambda_{N_h}^h)^{-1}, (\lambda_{1}^h)^{-1}] \subset [0,1]$; crucially, the interval 
$[(\lambda_{N_h}^h)^{-1}, (\lambda_{1}^h)^{-1}]$ excludes the `singular point' $y=0$, at which the function $y \in [0,1] \mapsto y^{s}$ has unbounded derivatives for all $s \in (0,1)$. This renders the minimax approximation a very accurate computational technique over the subinterval 
$[(\lambda_{N_h}^h)^{-1}, (\lambda_{1}^h)^{-1}]$ and substantiates the common choice of relatively low values of the degree $n$ of rational approximation, $n = 11, 12, 13$ being typical choices in practice. 
In Figure \ref{FigRootEXp} we numerically validate our reasoning for why exponential rather than slower root-exponential convergence is observed in the results reported in Figure \ref{FigConv3}. We compute rational approximations of the function $f(y) = y^s$ for different fractional orders $s \in (0,1)$ on intervals of the form $[\delta, 1]$ with $\delta>0$. As we can see in Figure \ref{FigRootEXp}, for all fractional orders, as $\delta$ decreases toward zero, the convergence in $n$ becomes slower, reaching the root-exponential convergence predicted by the estimate \eqref{StahlEstimate} for the interval $[0,1]$.

\begin{figure}
\centering
\begin{subfigure}{0.49 \textwidth}
\includegraphics[width = \textwidth]{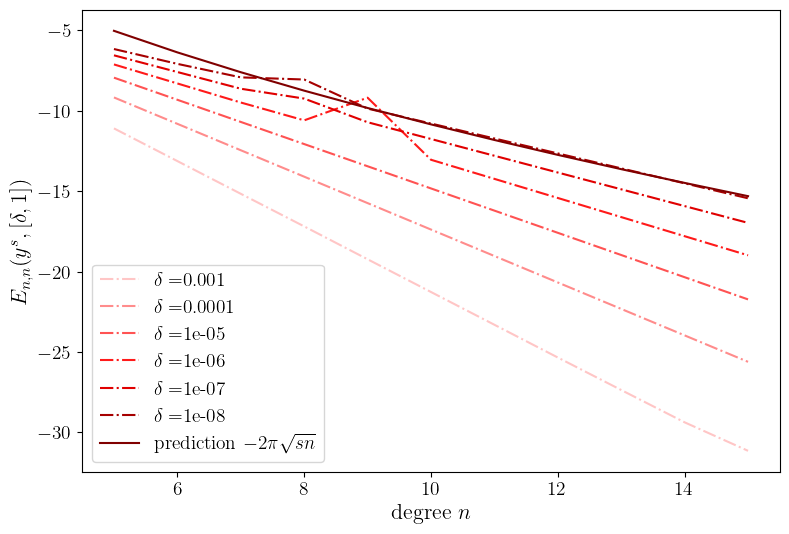}
\caption{$s = 0.2$}
\end{subfigure}
 \hfill
\begin{subfigure}{0.49 \textwidth}
\centering
\includegraphics[width = \textwidth]{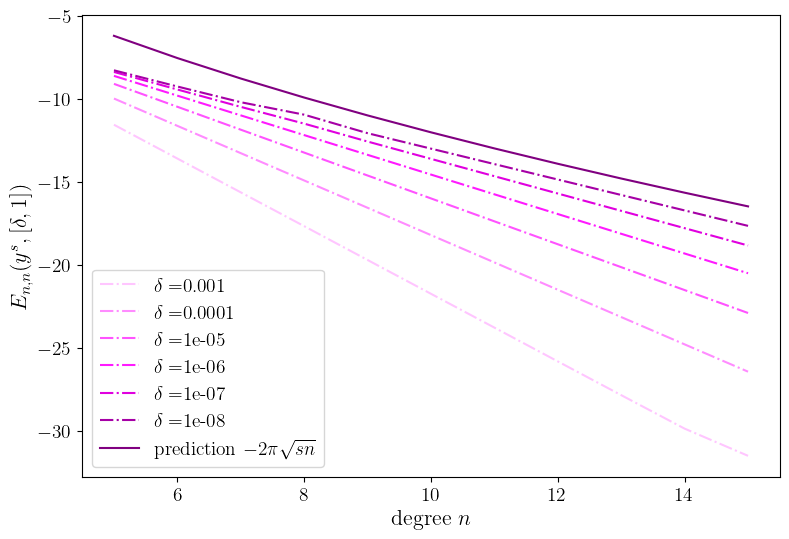}
\caption{$s=0.5$}
\end{subfigure}
\hfill
\begin{subfigure}{0.49 \textwidth}
\includegraphics[width = \textwidth]{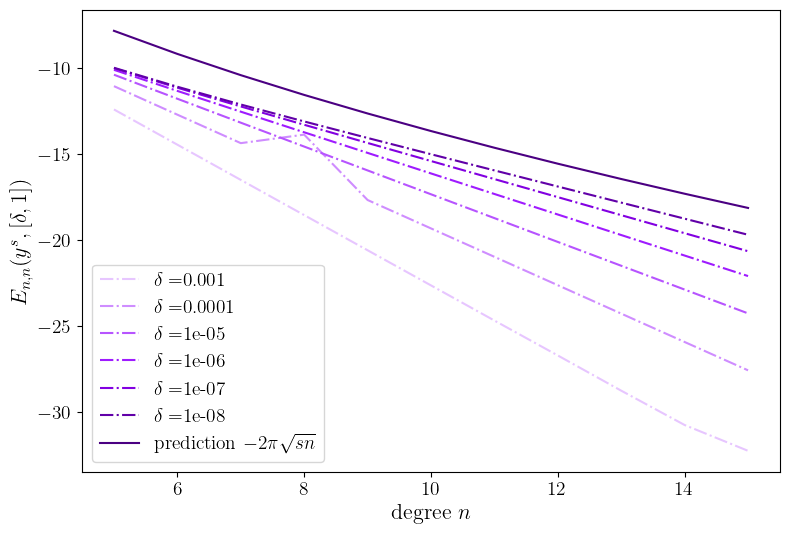}
\caption{$s = 0.7$}
\end{subfigure}
 \hfill
\begin{subfigure}{0.49 \textwidth}
\centering
\includegraphics[width = \textwidth]{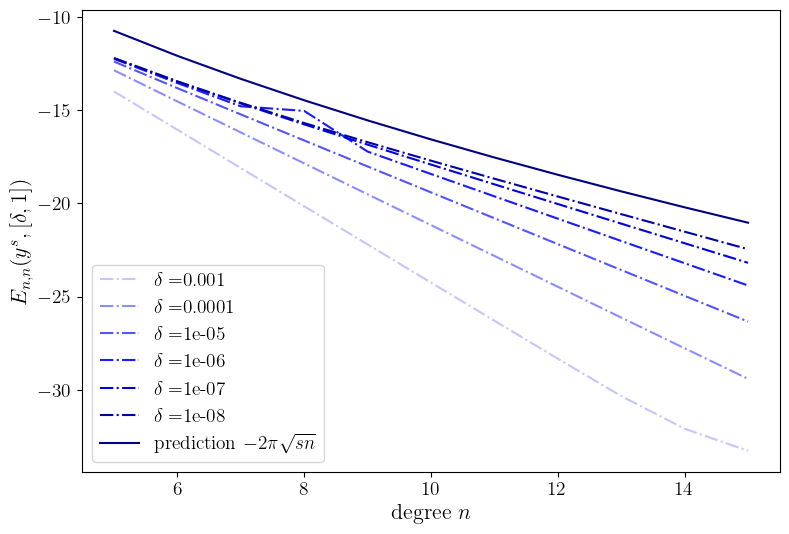}
\caption{$s=0.9$}
\end{subfigure}
\caption{Rational approximation of $f(y) = y^s$ computed through the minimax algorithm on intervals $[\delta, 1]$ for $\delta > 0$. Computation of the error $E_{n, n}(y^{s}, [\delta,1]) \coloneqq \| y^{s} -r^{\ast}_{n,n}(y) \|_{\infty,[\delta,1]} = \inf_{r \in \mathcal{R}_{n, n}} \| y^{s} - r(y) \|_{\infty,[\delta,1]}$ and comparison with the prediction \eqref{StahlEstimate} for $E_{n, n}(y^{s}, [0,1])$.}
\label{FigRootEXp}
\end{figure}

\section{Numerical experiments on related evolution problems}

We now present the results of applying the method described in Section 2 for the finite element approximation of the spectral fractional Laplacian to various fractional-order PDEs. We focus on two aggregation-diffusion problems: the fractional porous medium equation and the fractional Keller--Segel equation. These equations model the spatial and temporal distribution of a density function and are closely related. Both feature a potential term, which is the solution to a fractional Poisson equation. However, in the fractional porous medium equation, the potential is repulsive, promoting the dispersion of the density in space, whereas in the Keller--Segel model, the potential is attractive, leading to aggregation. Both phenomena are of significant interest and have numerous applications.

\subsection{The fractional porous medium equation}

The fractional porous medium equation arises in various scientific and engineering contexts, such as hydrology and subsurface flow, where it models anomalous diffusion in porous geological formations, where long-range interactions are significant, biology and ecology, where the equation is used to describe population dispersal with memory effects or spatial heterogeneity, and materials science, where it characterizes processes such as heat transfer in fractal media or heterogeneous materials. The equation has been studied in several papers; we refer the reader to \cite{Caffarelli2011, Caffarelli2015} for the first comprehensive contributions to the subject, and to \cite{Chen2022} for a more recent study. 

In \cite{carrillosuli2024} the authors presented an analysis of a finite element scheme for a parabolic regularization of the fractional porous medium equation, based on the spectral definition of the fractional Laplacian. We adopt here the same parabolic regularization as in \cite{carrillosuli2024} for the fractional porous medium equation, given by 
\begin{equation} \label{FullProbWSpace}
\frac{\partial \rho}{\partial t} = \sigma \Delta \rho - \nabla \cdot (\rho \nabla c), \quad  - (-\Delta)^{\s} c = \rho, \quad 0< s <1, \quad \sigma > 0.
\end{equation}
As $s \to 0$ the equation tends to a standard porous medium equation with Fickian diffusion.

In the first place, we want to validate our finite element approximation. We shall do so by studying the asymptotic decay towards the equilibrium solution of the following Fokker--Planck equation: 
\begin{equation} \label{PorousMediumFP}
\frac{\partial \rho}{\partial t} = \sigma \Delta \rho + \nabla \cdot (\rho \nabla \rho) + \nabla \cdot(x \rho)=
\nabla \cdot \left( \rho \nabla\left( F(\rho)+\frac{|x|^{2}}{2}\right)\right).
\end{equation}
An explicit steady-state solution to this equation can be computed on the whole of $\mathbb{R}^{d}$ as
\begin{equation} \label{FPStedayState}
    \rho^{\text{Std}}(x) = F^{-1}\bigg( -\frac{|x|^{2}}{2} + C\bigg), \quad \text{where } F(\psi) = \sigma \log(\psi) + \psi 
\end{equation}
and $C$ is a constant whose value is fixed by the mass of the positive initial datum. Denoting by $\rho_{h}$ the approximate steady-state solution that we compute using our finite element approximation, in Figure \ref{Fig1} we report the $L^{2}(\Omega)$ norm of the difference between $\rho^{\text{Std}}$ and $\rho_{h}$. We observe convergence up to the precision provided by the computational mesh. 

\begin{figure}
\centering
\includegraphics[width = 0.8\textwidth]{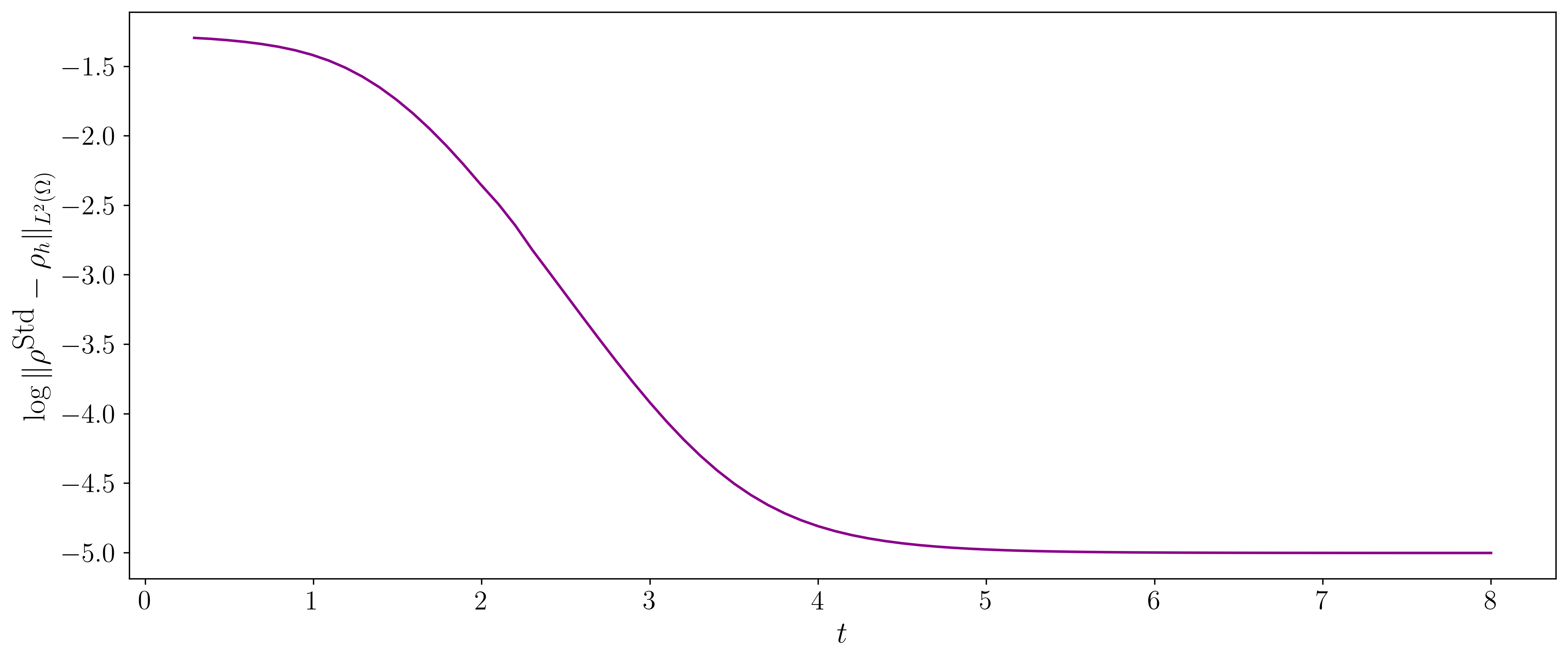}
\caption{The fractional porous medium equation \eqref{PorousMediumFP} in self-similar variables; convergence to the steady state \eqref{FPStedayState} for $s=2^{-10}$. The computational domain is $\Omega = B_{R}(0)$, $R=10$, $\sigma=1$, a mesh of 37267 uniformly spaced mesh points was used with $\Delta t = 0.01$.}
\label{Fig1}
\end{figure}

 Now, let us consider the fractional porous medium equation without the Fickian diffusion term, that is, for $\sigma=0$ in \eqref{FullProbWSpace}:
\begin{equation*}
    \frac{\partial \rho}{\partial t} = - \nabla \cdot (\rho \nabla c), \quad  - (-\Delta)^{\s} c = \rho, \quad 0< s <1, 
\end{equation*}
for which there exists a family of self-similar profiles; cf. \cite{BilerKarch2011, BilerKarch2015}. In order to further validate our method, we consider the solution of the corresponding Fokker--Planck equation
\begin{equation} \label{FractionalPorFP}
\frac{\partial \rho}{\partial t} = - \nabla \cdot (\rho \nabla c) + \mu \nabla \cdot (x \rho), \quad  - (-\Delta)^{\s} c = \rho, \quad 0< s <1, \quad \mu = \frac{1}{d+2-2s}
\end{equation}
that admits a steady state on the whole of $\mathbb{R}^{d}$, given by
\begin{equation} \label{SSFracPM}
\Phi_{s}(y) = k_{s, d} (1 - |y|^{2})^{1-s}_{+}, \quad \text{with } k_{s,d} = \frac{d \Gamma(d/2)}{(d + 2(1-s)) 2^{2(1-s)} \Gamma(2-s) \Gamma(d/2 + 1 s)}.
\end{equation} 
Next, we test our method by applying it to the equation
\begin{equation} \label{FractionalPorFPDiff}
\frac{\partial \rho}{\partial t} = \sigma \Delta \rho - \nabla \cdot (\rho \nabla c) + \mu \nabla \cdot (x \rho), \quad  - (-\Delta)^{\s} c = \rho, \quad 0< s <1, 
\end{equation}
for very small values of $\sigma$ and compare the numerically computed steady states that we compute with \eqref{SSFracPM}. In doing this, however, it needs to be borne in mind that \eqref{SSFracPM} does not belong to $H^{1}(\Omega)$, but only to $H^{1-s}(\Omega)$ with $s \in [0,1)$. This is reflected in our numerical results by the fact that, for a fixed mesh size, as $\sigma$ becomes smaller, our numerical solution develops oscillations in a small neighbourhood of the unit circle $\{y \in \mathbb{R}^{d}, |y| = 1 \}$, where the $H^{1}$ regularity of the solution is lost. This is expected behaviour, since we are using a finite element method that relies on the presence of the parabolic regularization term $\sigma \Delta \rho$. In Figure \ref{Fig2} we report the results for selected values of $\sigma$. While we observe local oscillations developing as $\sigma$ decreases, we also see that away from the unit circle the profile of the numerical solution approaches the analytical solution profile \eqref{SSFracPM}. This example also highlights that our method admits the use of locally refined computational meshes; indeed, the local oscillations are mitigated by local mesh refinement in the vicinity of the unit circle. 

\begin{figure}
	\centering

	\begin{subfigure}{0.49\textwidth}			\includegraphics[width=\textwidth]{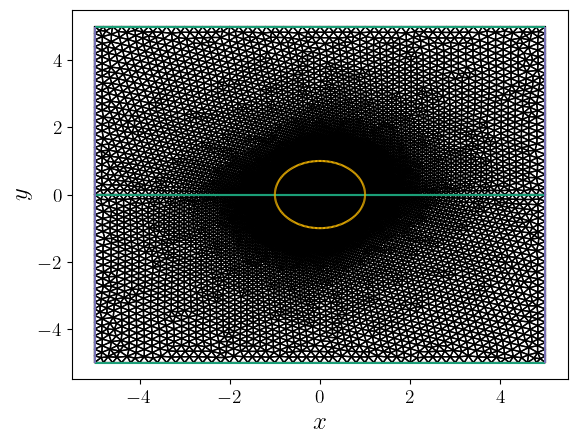}
		\caption{}
	\end{subfigure}
	\hfill
	\begin{subfigure}{0.49\textwidth}
		\includegraphics[width=\textwidth]{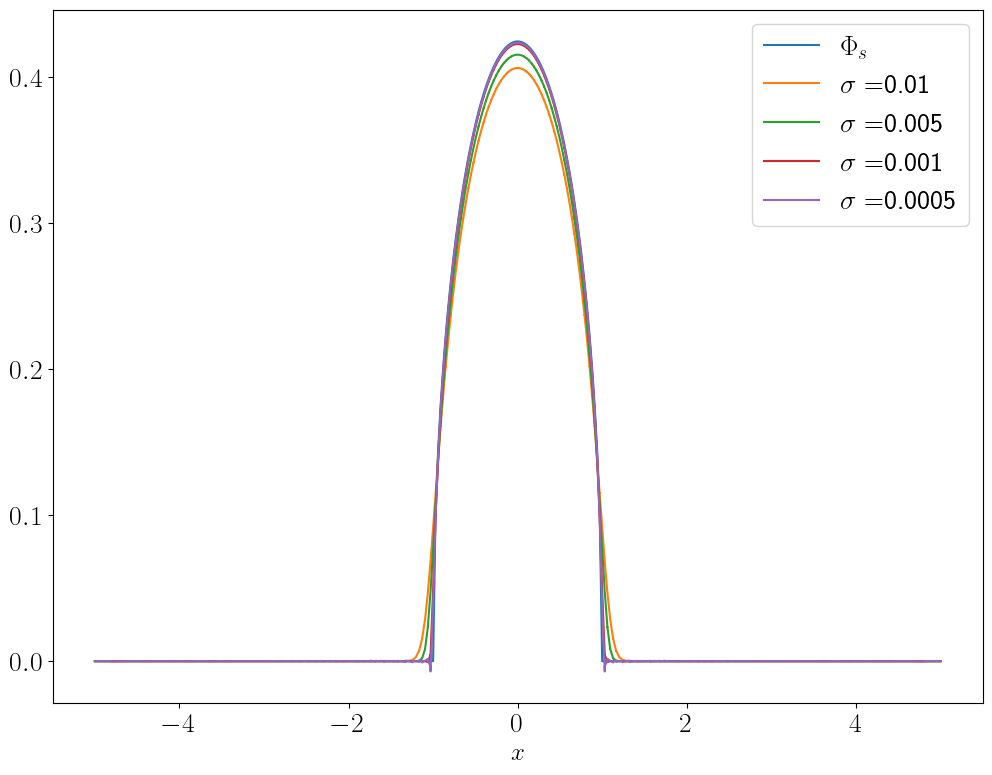}
		\caption{}
	\end{subfigure}

\begin{subfigure}{0.7\textwidth}
		\includegraphics[width=\textwidth]{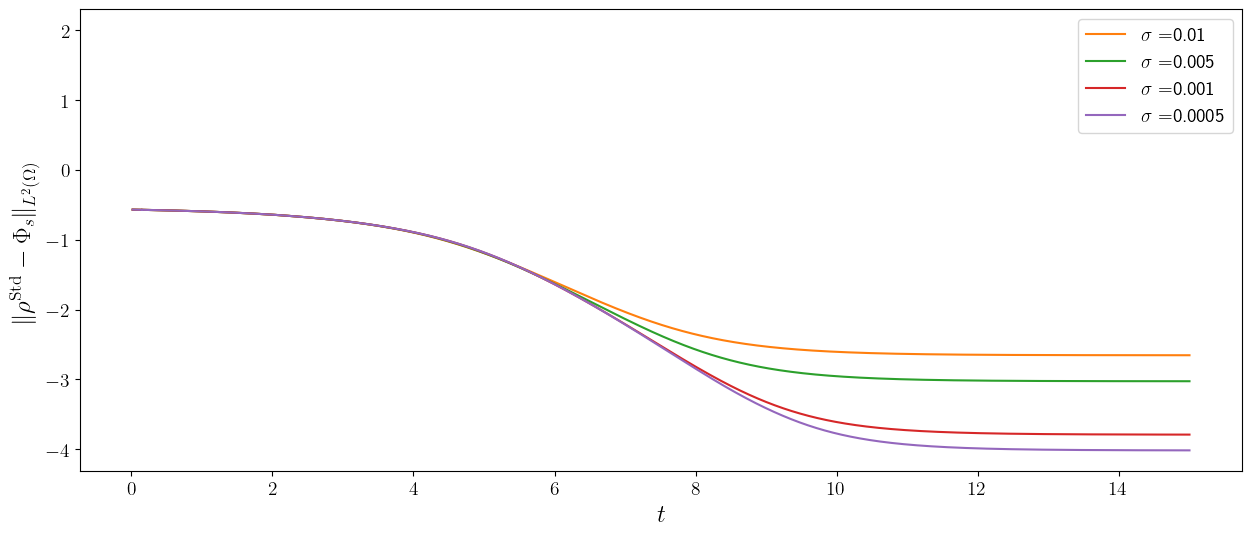}
		\caption{}
	\end{subfigure}

	\caption{The fractional porous medium equation in self-similar variables \eqref{FractionalPorFPDiff}; (a) the locally refined mesh on the computational domain $\Omega$; (b) steady state solution profile for equation \eqref{FractionalPorFPDiff} for different values of $\sigma$; (c) the  $L^{2}(\Omega)$ error between the computed steady state of \eqref{FractionalPorFPDiff} and the analytical steady state $\Phi_s$. The computational domain is $\Omega = (-5,5)^2$, $\rho_0(x) \propto \frac{1}{|\Omega|}$, a mesh  with 13171 mesh points  was used, graded towards the centre of the domain, with the time step taken as $\Delta t = 0.01$.}
	\label{Fig2}
\end{figure}
We shall now explore the behaviour of solutions to \eqref{FractionalPorFPDiff} with different values of the fractional order $s \in (0,1)$. We report the results in Figure \ref{FigProfilesSS}. We observe generic solutions to \eqref{FractionalPorFP} that exhibit the correct qualitative behavior of spreading; see \cite{Caffarelli2011,Caffarelli2015}. The solutions are expected to spread according to the self-similar profile stated in equation \eqref{SSFracPM}. The 2D simulations reported in Figure \ref{FigProfilesSS} show that, as $s\to 0$, the spread becomes more pronounced.
\begin{figure} 
\begin{subfigure}{0.6\textwidth}
\centering
\includegraphics[width = \textwidth]{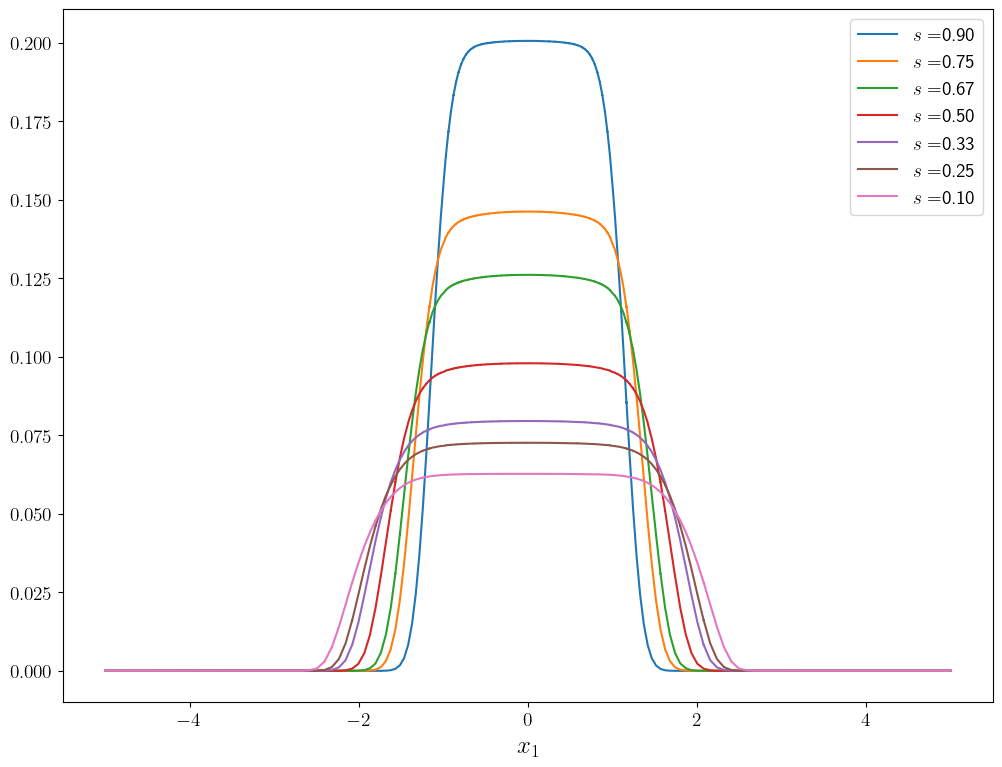}
\caption{}
\end{subfigure}
\centering
\begin{subfigure}{0.7\textwidth}
\includegraphics[width = \textwidth]{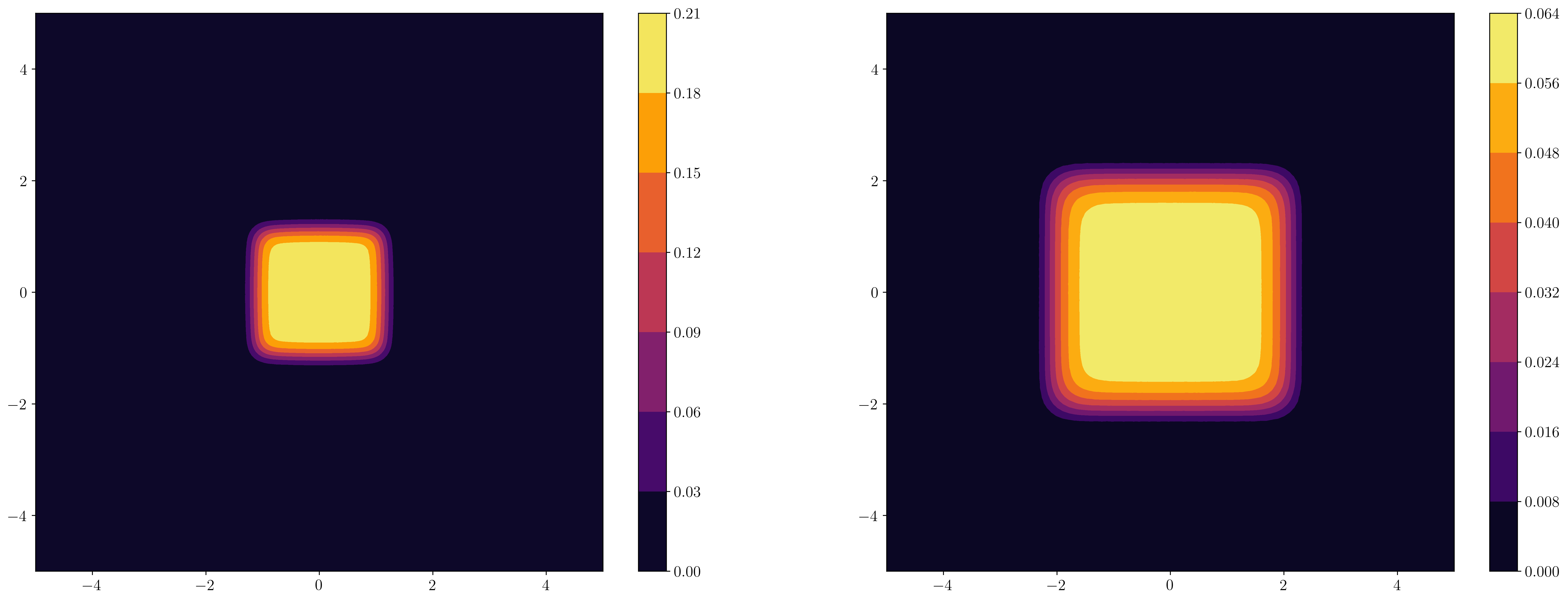}
\caption{}
\end{subfigure}
\caption{Steady state profiles of the fractional porous medium  equation \eqref{FractionalPorFPDiff} in self-similar variables with $\sigma = 0.01$ for different values of $s \in (0,1)$. The computational domain is  $\Omega = (-5,5)^2$,  $\rho_0(x) \propto \frac{1}{|\Omega|}$; a graded mesh towards the centre of the domain, with 13171 mesh points, was used with the time step taken as $\Delta t = 0.01$. (a): cross-section of the numerical solution at $x_2=0$ as a function of $x_1$ for different values of $s$; (b): contour plot for values of $s= 0.9$ (left) and $s=0.1$ (right).}
\label{FigProfilesSS}
\end{figure}

\begin{figure}
\centering
\includegraphics[width = 0.7  \textwidth]{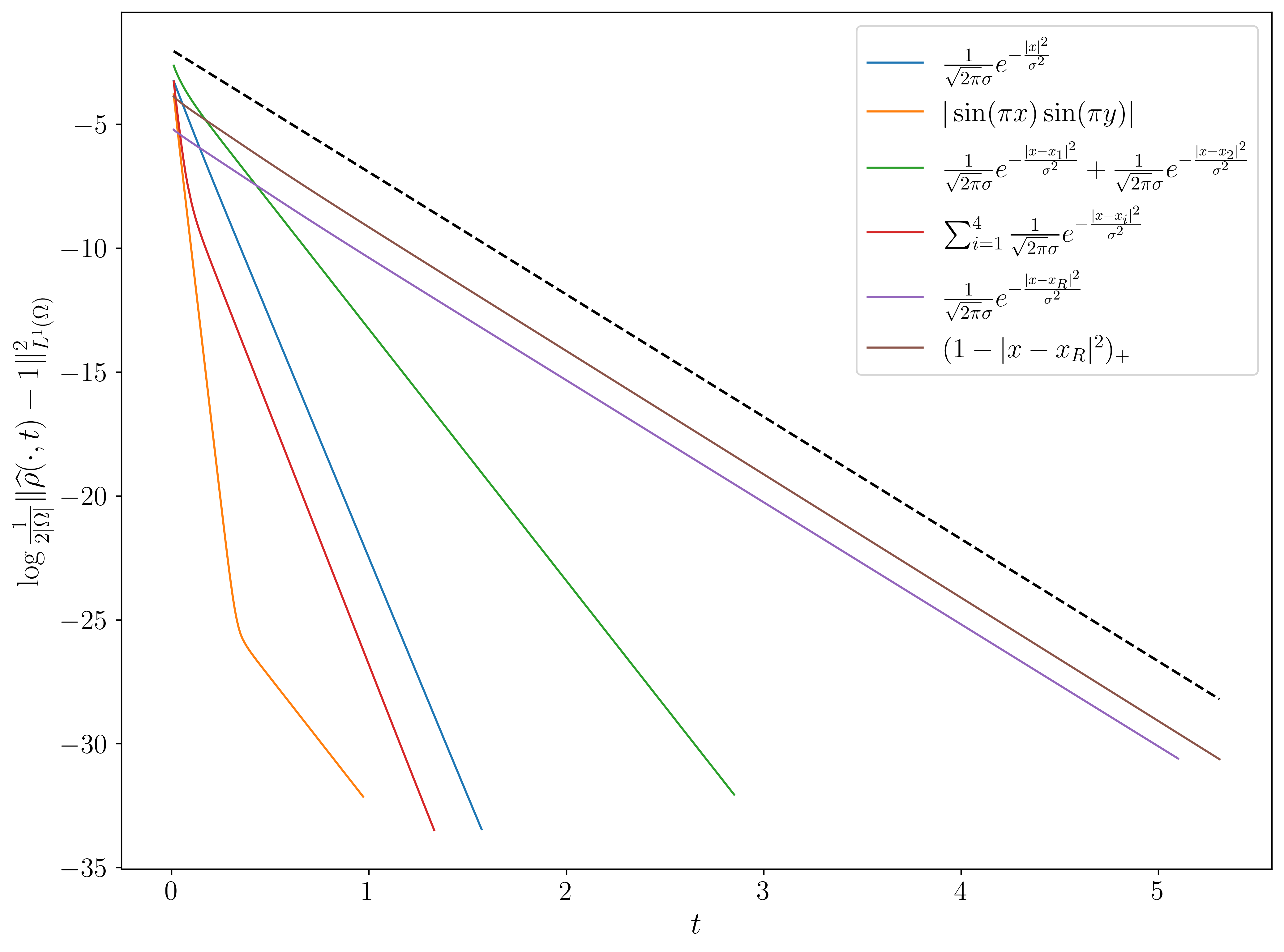}
\caption{The fractional porous medium equation \eqref{FullProbWSpace}; exponential decay to the equilibrium solution for different initial data $\rho_{0}$. The computational domain is $\Omega = (-1, 1)^2$, a mesh of $2^7 \times 2^7$ uniformly spaced mesh points was used with the time step taken as $\Delta t = 0.01$, the fractional order is $s = 1/2$, and $x_R = (1,1)$. The dashed black reference line has slope $-2 t / (C_\Omega \| \rho_0 \|_{L^\infty(\Omega)})$.} 
\label{Fig5}
\end{figure}

It is known that the comparison principle does not hold for the porous medium equation with a fractional potential \eqref{FullProbWSpace} with $\sigma=0$; cf. Section 6 of \cite{Caffarelli2011}. This was demonstrated by numerical simulations in one space dimension in \cite{DelTeso2023}. We show the same phenomenon in a two-dimensional simulation in Figure \ref{Fig4}. We choose a small constant value of $\sigma$ and two initial data: $u_{1}(x) = g(x-2)$ and $u_{2}(x) = g(x-2) + 2 g(x+2)$, where $g(x) \propto \mathrm{e}^{-x^2/C}$ for a positive constant $C$. We see in Figure \ref{Fig4} that while the chosen initial data are initially ordered, $u_1 \leq u_2$, the order is lost as the solution evolves over time. This fact is due to the non-locality of the fractional Laplacian: when the density $u_2$ diffuses, one part of it (the one given initially by $g(x-2)$) is dragged towards the other one (given initially by $2g(x+2)$) that has a higher value, and in this way the comparison with $u_1$, which is concentrated in part of the domain only, is lost.

Finally, we validate our scheme by checking the asymptotics in time of solutions to the fractional porous medium equation stated in \cite[Theorem 4.7]{carrillosuli2024}, where exponential decay of weak solutions of \eqref{FullProbWSpace} to the equilibrium solution was proved, with a rate that depends on the $L^{\infty}(\Omega)$ norm of the initial datum $\rho_{0}$. We report in Figure \ref{Fig5} the rates of convergence to the equilibrium solution for different initial data $\rho_{0}$ with the same $L^{\infty}(\Omega)$ norm, where $\Omega=(-1,1)^2$. We observe that coincidence with the theoretically predicted decay rate to the equilibrium solution is best for initial data that are the furthest from the equilibrium state: the Gaussian and a blob centered at a corner of the square domain $\Omega$.  

\begin{figure}
	\centering

 \begin{subfigure}{0.45\textwidth}			\includegraphics[width=\textwidth]{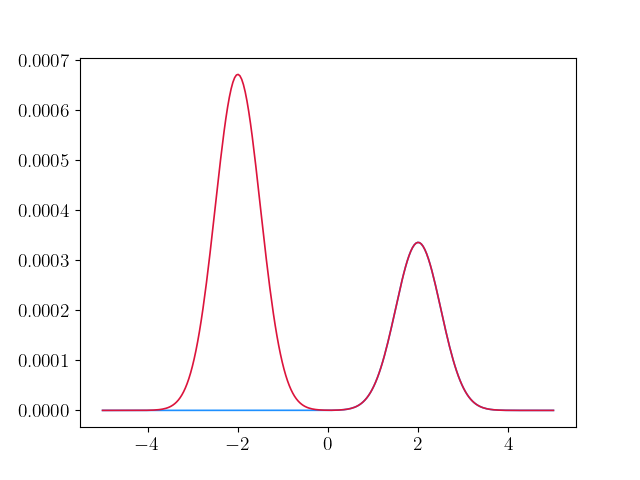}
		\caption{$t=0$, section $x_2=0$}
	\end{subfigure}
	\hfill
	\begin{subfigure}{0.45\textwidth}
		\includegraphics[width=\textwidth]{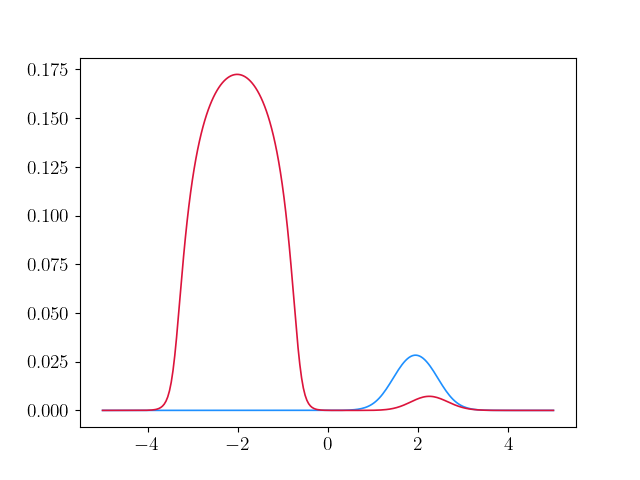}
		\caption{$t=4$, section $x_2=0$}
	\end{subfigure}
\caption{The fractional porous medium equation \eqref{FullProbWSpace}; failure of the comparison principle for fractional order $s=3/4$. The computational domain is $\Omega = (-R, R)^2$, $R=5$, $\sigma = 10^{-5}$; a uniform mesh of $2^8 \times 2^8$ uniformly spaced mesh points was used with the time step taken as $\Delta t = 0.01$.}
  \end{figure}

\begin{figure}
\centering
  \begin{subfigure}{0.7\textwidth}			\includegraphics[width=\textwidth]{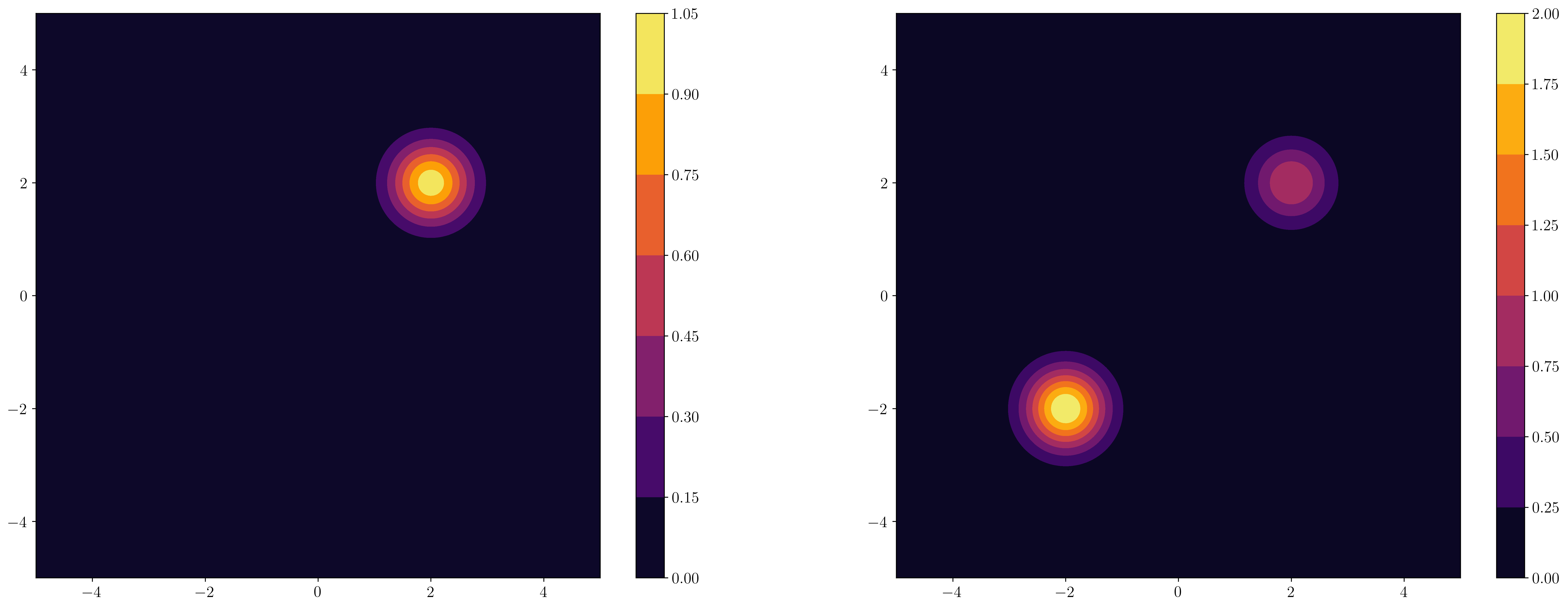}
		\caption{$t=0$}
	\end{subfigure}
	\hfill
	\begin{subfigure}{0.7\textwidth}
		\includegraphics[width=\textwidth]{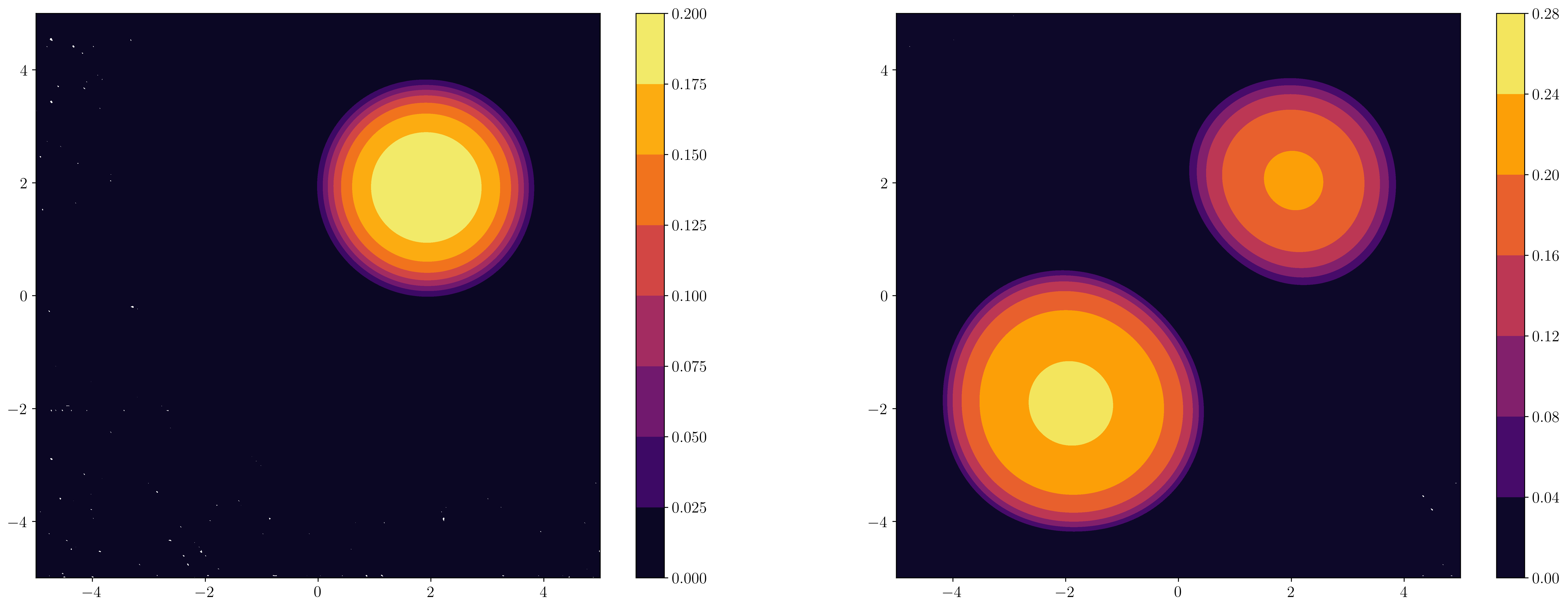}
		\caption{$t=4$}
	\end{subfigure}

\caption{The fractional porous medium equation \eqref{FullProbWSpace}; failure of the comparison principle for fractional order $s=3/4$. The computational domain is $\Omega = (-R, R)^2$, $R=5$, $\sigma = 10^{-5}$; a uniform mesh of $2^8 \times 2^8$ uniformly spaced mesh points was used
 with the time step taken as $\Delta t = 0.01$.}
\label{Fig4}
\end{figure}

\subsection{The fractional Keller--Segel model}

The standard Keller--Segel model is a well-established mathematical model used to describe chemotaxis, that is, the directed movement of cells or organisms in response to chemical gradients.
In recent years, there has been significant interest in studying nonlocal generalizations of the model, with the standard Laplacian replaced by the fractional Laplacian; see, for instance, \cite{Biler1999, BilerKarch2011,  Lafleche2019} and the references therein. This variant is motivated by the need to model anomalous diffusion, such as Lévy flights or long-range interactions, often observed in biological and ecological systems. 
The model is given by 
\begin{equation}\label{fractionalKellerSegel}
\frac{\partial \rho}{\partial t} = \Delta \rho - \nabla \cdot (\rho \nabla c), \quad \text{where } (-\Delta)^{s} c = \rho^{\ast}.
\end{equation}
Blow-up in the context of the Keller--Segel model refers to the unbounded growth of the solution in finite time, signifying the collapse of the population into a singular structure. The conditions under which blow-up occurs are critical for understanding the balance between chemotactic attraction and diffusive dispersion. 

For the classical Keller--Segel equation, that is, for $s=1$, it has been shown that in $\mathbb{R}^{d}$ blow-up of the solution can occur if the initial mass is larger than a critical value. In the seminal paper \cite{Jager1992} blow-up was proved for the classical model in $\mathbb{R}^2$ when the initial mass exceeds a critical threshold, which was later determined in \cite{Dolbeault2004}, with the result further completed and improved in \cite{Doulbeaut2006}, where it was shown that, given $M := \int_{\Omega} \rho_{0} \dx$, if $M>8\pi$ then the density blows up in finite time, while for $M<8\pi$ blow-up does not take place. Additional information on blow-up profiles was obtained in \cite{Herrero1997,Velazquez2002}. In $\mathbb{R}^d$, for $d\geq 3$, the blow-up depends not only on the initial mass of the solution, but also on how the initial mass density is concentrated; we refer to \cite{CPZ04} for further details. 

For the fractional regime $0< s< 1$ the study of potential blow-up for $\rho$ is a challenging undertaking. The blow-up conditions for the fractional Keller--Segel model depend, in any number of dimensions $d\geq 1$, not only on the initial mass $M$, but also on the relative concentration of the initial density. It is in fact known that on $\mathbb{R}^{d}$, under some initial mass concentration criteria, the solution ceases to exist in finite time; see \cite[Proposition 4.2]{Biler1999} and \cite[Theorem 3.7]{Lafleche2019} for a more general class of models in which non-standard diffusion is also present. 

We test in two space dimensions the behaviour of solutions to \eqref{fractionalKellerSegel} using our numerical approximation of the fractional Laplacian. We explore the blow-up by choosing several initial values $\rho_{0}$, which are all Gaussians $\rho_{0}(x) \propto \exp(-|x|^{2} / 2 \sigma^{2})$ with different masses $M = \int_{\Omega} \rho_{0} \dx$ and variances $\sigma$, leading to a range of initial concentrations. Our results are reported in Figure \ref{Fig6} and Figure \ref{Fig7}. For all of our simulations we chose a dynamic time step $\Delta t_n$, defined as follows: in a fully discrete approximation (in space and time), if $\rho_h^n$ is the numerical approximation of the solution to \eqref{fractionalKellerSegel} at time $t_n = \sum_{k=1}^n \Delta t_{k}$, then the following time step $\Delta t_{n+1}$ is defined by $\Delta t_{n+1} := 1/\max(\| \rho_h^n\|_{L^{\infty}(\Omega)}, |\rho_h^n|_{H^1(\Omega)})$. In our simulations numerical blow-up is declared when the adaptive time step $\Delta t$ becomes smaller than a threshold value $\delta \ll 1$. 
We observe strong correlation between blow-up and the initial concentration of the density: for two initial data, having the same mass $M$, blow-up can happen or not depending on the extent to which the density is concentrated at the origin. Secondly, we notice that blow-up is more likely to occur for small values of the fractional order $s \in (0,1)$. 

Blow-up results for the fractional Keller--Segel model were previously proved in the literature with the fractional Laplacian defined through its Riesz integral representation on $\mathbb{R}^d$ or with its truncation to a bounded domain in the case of a homogenenous Dirichlet boundary condition. The novelty of our work is in numerically testing the blow-up in the case of the spectral representation of the fractional Laplacian. In addition, we provide in Proposition \ref{BlowUpProp} a proof of existence of solutions to the fractional Keller--Segel model, based on the spectral definition of the fractional Dirichlet Laplacian, that exhibit finite-time blow-up. Let us consider the fractional Keller--Segel model
\begin{equation} \label{FracKS}
\frac{\partial \rho}{\partial t} = \Delta \rho - \nabla \cdot (\rho \nabla c), \quad \text{where } (-\Delta_{\mathrm{D}})^{s} c = \rho, \quad \frac{1}{2} < s < 1,
\end{equation}
together with its weak formulation 
\begin{gather}
\text{Find } \rho \in L^2(0,T;V)  \text{ with $\frac{\partial \rho}{\partial t} \in L^\infty(0,T;V')$ such that} \nonumber \\
\Big\langle \frac{\partial \rho}{\partial t}, \phi \Big\rangle = - \int_{\Omega} \nabla \rho \cdot \nabla \phi \dx + \int_{\Omega} \rho \nabla c \cdot \nabla \phi \dx \quad \text{for all } \phi \in V \text{ and a.e. $t \in (0,T]$}, \label{BasicWeakForm}
\end{gather}
for $V=H^1(\Omega)$, subject to the initial condition $\rho(x, 0) = \rho_{0}(x)$, where $\rho_{0} \in L^{\infty}(\Omega)$ and $\rho_{0}(x) \geq 0$ for a.e. $x \in \Omega$ and where
\begin{equation} \label{MyFracPois1}
    (-\Delta_{\mathrm{D}})^{s} c = \rho \textrm{ in } \Omega, \quad \frac{1}{2} < s <1. 
\end{equation}
We note that taking $\phi \equiv 1$ as a test function in \eqref{BasicWeakForm} implies conservation of mass of $\rho$ in time, namely $\frac{\dd}{\dd t} M(t) = 0$, with $M(t) = \int_{\Omega} \rho(x, t) \dx$. 

\begin{prop} \label{BlowUpProp}
Let $\Omega$ be a ball of radius $R$ centred at zero
in $\mathbb{R}^{d}$ with $d= 2,3$. Let $\frac{1}{2}<s<1$, and let $\rho$ be a nonnegative solution to \eqref{FracKS}. Then, there exists a nonnegative function $\rho_{0} \in L^\infty(\Omega)$ such that the corresponding solution blows up in finite time.
\end{prop}
\begin{proof}
Consider the second moment function $\Phi(t) := \int_{\Omega} |x|^{2} \rho(x, t) \dx > 0$. Clearly $0<\Phi(0) < \infty$.
The idea is to follow the evolution of the function $\Phi(t)$ and to prove that $\Phi$ satisfies a differential inequality of the form
\begin{equation} \label{BUcriterion}
\frac{\dd}{\dd t} \Phi(t) \leq f(\Phi(t)),
\end{equation}
where $f$ is a continuous nondecreasing function such that $f(0)<0$. Once we have  shown \eqref{BUcriterion}, we shall define $\Phi^{\dagger} \coloneqq \inf \{ \Phi > 0\text{ such that } f(\Phi) =0 \} \in (0, +\infty]$, for any sufficiently smooth solution of \eqref{FracKS} with initial second moment satisfying $\Phi(0) < \Phi^{\dagger}$; then, there exists a time $T^{\dagger} < \infty $ such that $\lim_{t \nearrow T^{\dagger}} \Phi(t) = 0$. Vanishing of the second moment will then imply blow-up of the solution.  Using the weak formulation of our problem and integration by parts, we have that 
\begin{equation} \label{EqBU1}
\frac{\dd \Phi}{\dd t} = 2dM - 2 \int_{\partial \Omega} \rho x \cdot n + 2 \int_{\Omega} \rho \nabla c \cdot x \dx,
\end{equation}
where $M = \int_{\Omega} \rho(x, t) \dx$ and $n$ is the unit outward normal vector to $\partial \Omega$. 
As the domain $\Omega$ is an open ball centred at zero, it follows that $n = x/|x|$ (notice that we could have assumed $\Omega$ to be any bounded convex set with the centre of the coordinate system contained in the interior of $\Omega$ to ensure that $x \cdot n \geq 0$) and therefore the second term on the left-hand side \eqref{EqBU1} can be discarded.

For the third term on the left-hand side of \eqref{EqBU1} we shall use Lemma \ref{PohoLemma}, first noticing that
\begin{align} \label{EqBU2}
\int_{\Omega} \rho \nabla c \cdot x \dx &= \int_{\Omega} (\nabla c \cdot x) (-\Delta_{\mathrm{D}})^{s} c \dx. 
\end{align}

Using Lemma \ref{PohoLemma} on the right-hand side of \eqref{EqBU2}, we immediately have 
\begin{align} \label{EqBu2.1}
\int_{\Omega} (\nabla c \cdot x) (-\Delta_{\mathrm{D}})^{s} c \dx &\leq - \frac{1}{2}(d-2s) \int_{\Omega} c (-\Delta_{\mathrm{D}})^{s}c \dx =  - \frac{1}{2}(d-2s) \int_{\Omega} c \rho^{} \dx. 
\end{align}

We now recall the explicit expression for $c$, which is the solution of \eqref{MyFracPois1}, given by 
\begin{equation} \label{ExpFracPoiSol}
c(x) = \frac{1}{\Gamma(s)} \int_{0}^{\infty} \int_{\Omega} W_{t}^{\Omega}(x, y) \frac{1}{t^{1-s}} \rho(y) \dy \dt,
\end{equation} 
where $W_{t}^{\Omega}$ denotes the distributional heat kernel for the heat equation on $\Omega$ with homogeneous Dirichlet boundary condition. We note that the assumption that $\rho$ is nonnegative implies the nonnegativity of $c$. 

We proceed by considering the following auxiliary problem for $\lambda>1$:
\begin{equation*}
(-\Delta_{\textrm{D}}^{\lambda \Omega})^s \overline{c} =  \overline{\rho} \quad \textrm{in } \lambda \Omega \coloneqq \{ \lambda x, x \in \Omega\},
\end{equation*}
where $\overline{\rho}$ is the natural extension by zero outside $\Omega$, namely $\overline{\rho}(y) = \rho(y)$ if $y\in\Omega$ and $\overline{\rho}(y) = 0$ if $y\notin \Omega$. By Lemma \ref{ConvLemma} we have $\overline{c} \to c$ in $\mathbb{H}^s(\Omega)$ as $\lambda \to 1$. This implies strong and hence weak convergence in $L^2(\Omega)$. In particular $(\overline{c}, \rho) \to (c, \rho)$ as $\lambda \to 1$. Therefore, there exists a $\lambda_1 > 1$ such that $\frac{1}{2} (\overline{c}, \rho) < (c, \rho) < \frac{3}{2} (\overline{c}, \rho)  $. Fix a $\lambda$ such that $1< \lambda < \lambda_1$; we can then rewrite \eqref{EqBu2.1} as
\begin{align*}
\int_{\Omega} (\nabla c \cdot x) (-\Delta_{\mathrm{D}})^{s} c \dx &\leq - \frac{1}{2}(d-2s) \int_{\Omega} c \rho \dx = - \frac{1}{2}(d-2s) (c, \rho) \\
& \leq - \frac{1}{4}(d-2s) (\overline{c}, \rho) = - \frac{1}{4}(d-2s) \int_{\Omega} \overline{c}\rho \dx. 
\end{align*}

Since $\overline{\rho}$ is equal to zero outside $\Omega$, for $\overline{c}$ we have an expression analogous to  \eqref{ExpFracPoiSol} with $W_t^{\lambda \Omega}$ in place of $W_t^\Omega$. Therefore we have
\begin{align*}
\int_{\Omega} (\nabla c \cdot x) (-\Delta_{\textrm{D}})^{s} c \dx &\leq -\frac{1}{2} (d - 2s) \frac{1}{\Gamma(s)} \int_{\Omega} \int_{0}^{\infty} \int_{\Omega} \rho(x) \rho(y) W_{t}^{\lambda \Omega}(x, y) \frac{1}{t^{1-s}} \dy \dt \dx.
\end{align*}
Now that the distributional Dirichlet heat kernel refers to the domain $\lambda \Omega$, which strictly contains $\Omega$, we can use the exponential bounds that are known to hold for the distributional heat kernel in a domain $\Omega$, which is strictly inside the domain $\lambda \Omega$:
\begin{equation*}
c_{1} \frac{\mathrm{e}^{-|x-y|^{2}/(c_{3} t)}}{t^{d/2}} \leq  W_{t}^{\lambda \Omega}(x, y) \leq c_{2} \frac{\mathrm{e}^{-|x-y|^{2}/(c_{4} t)}}{t^{d/2}}, \quad x, y \in \Omega \subset \lambda \Omega, \quad t>0,
\end{equation*}
for positive constants $c_{1}, c_{2}, c_{3}$ and $c_4$ (see \cite{caffarelli2016fractional}). We then have, after integration in time, with a similar computation as in Lemma A.4 in \cite{carrillosuli2024} that
\begin{align*}
\int_{\Omega} (\nabla c \cdot x) (-\Delta)^{s} c \dx &\leq - C (d- 2s)\frac{\Gamma(d/2-s)}{\Gamma(s)} \int_{\Omega} \int_{\Omega} \rho(x) \rho(y) |x-y|^{-(d-2s)} \dx \dy,
\end{align*}
where $C>0$ is a constant. 
Therefore, for some positive constant $C_{d,s}>0$, it follows that
\begin{equation*}
\frac{\dd \Phi}{\dd t} \leq 2dM + 2 \int_{\Omega} \rho \nabla c \cdot x \dx \leq 2dM - C_{d, s} (d-2s) \int_{\Omega} \int_{\Omega} \rho(x) \rho(y) |x-y|^{-(d-2s)} \dx \dy.
\end{equation*}

Using H\"{o}lder's inequality, denoting $\alpha = d - 2s$ and $I_{\alpha} = \int_{\Omega} \int_{\Omega} |x-y|^{-\alpha} \rho(x) \rho(y) \dx \dy$, we have 
\begin{align*}
M^{2} &= \int_{\Omega} \int_{\Omega} \rho(x) \rho(y) \dx \dy \leq \bigg( \int_{\Omega} \int_{\Omega} \rho(x) \rho(y) |x-y|^{2} \dx \dy \bigg)^{\frac{\alpha}{\alpha + 2}} \bigg( \int_{\Omega} \int_{\Omega} \rho(x) \rho(y) |x-y|^{-\alpha} \dx \dy \bigg)^{\frac{2}{\alpha + 2}} \\
&= \bigg( \int_{\Omega} \int_{\Omega} \rho(x) \rho(y) (|x|^{2} - x \cdot y - y \cdot x + |y|^{2}) \dx \dy\bigg)^{\alpha/(\alpha+2)} I_{\alpha}^{\frac{2}{\alpha + 2}} \\
&= \bigg( 2M \Phi - 2 \bigg| \int_{\Omega} x \rho(x) \dx \bigg|^{2} \bigg)^{\alpha/(\alpha + 2)} I_{\alpha}^{\frac{2}{\alpha + 2}};
\end{align*}
hence $2^{-\alpha/2} M^{\alpha/2 + 2} \Phi^{-\alpha/2} \leq I_{\alpha}$.
Therefore, we have in the end
\begin{equation} \label{EqBU3}
\frac{\dd \Phi}{\dd t}  \leq 2dM - C_{d, s} (d-2s) M^{\frac{d}{2} - s + 2} \Phi^{-\frac{d}{2} + s}
\end{equation}
This means that if at time $t=0$ we have 
\begin{equation*}
\Phi(0)^{\frac{d}{2} - s} < \frac{C_{d,s} (d - 2s) M^{\frac{d}{2} - s + 2}}{C_{d,s} (d - 2s) M_0^{\frac{d}{2} - s + 2} + 2dM_0},
\end{equation*}
then the right-hand side of expression \eqref{EqBU3} is strictly negative, and decreases for $t\geq 0$.

\end{proof}

\begin{figure}[H]
\centering
\begin{subfigure}{0.53 \textwidth}
\includegraphics[width = \textwidth]{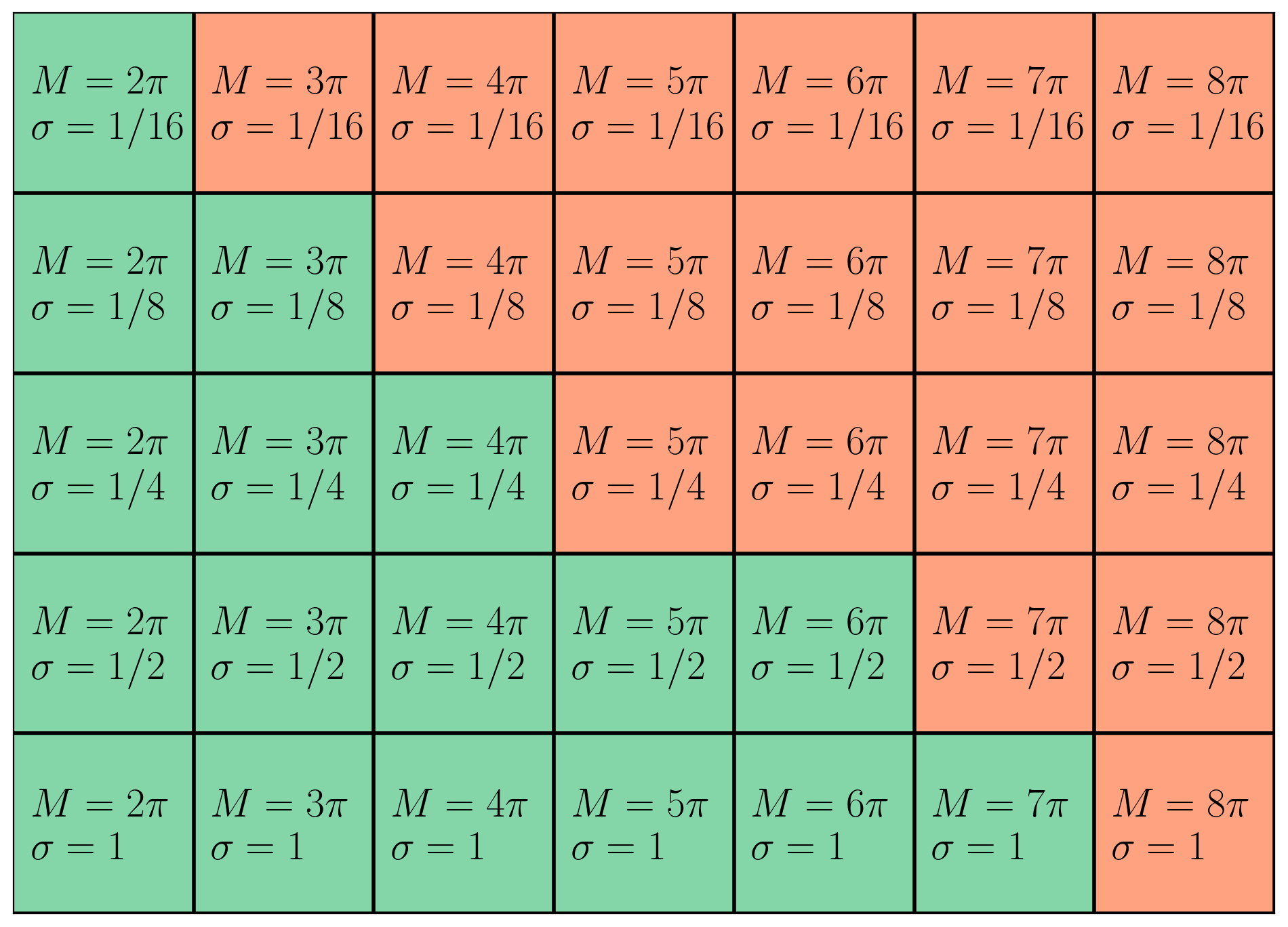}
\caption{$s = 0.75$}
\end{subfigure}
 \hfill
\begin{subfigure}{0.455 \textwidth}
\centering
\includegraphics[width = \textwidth]{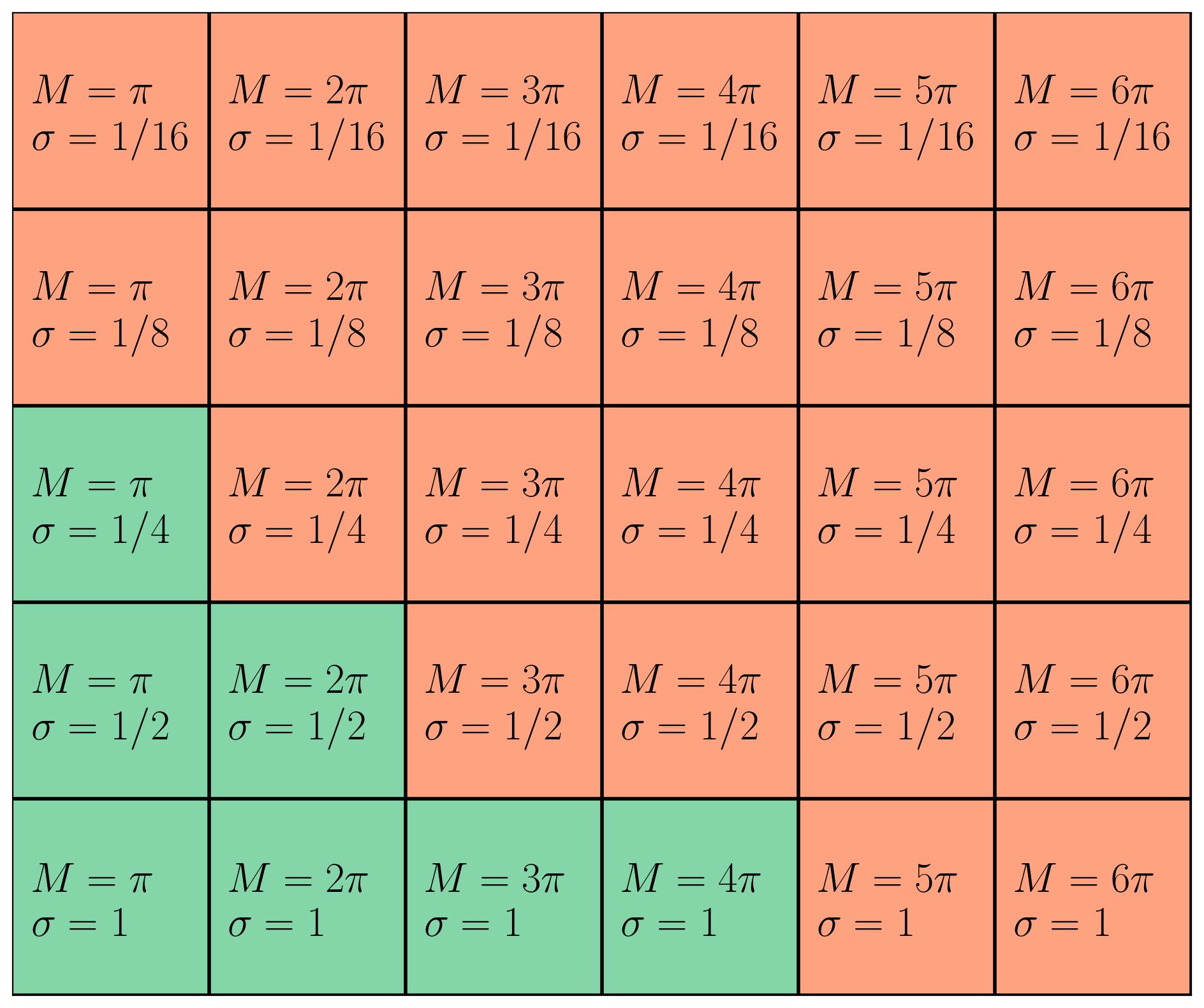}
\caption{$s = 0.5 + \epsilon$}
\end{subfigure}
\caption{The fractional Keller--Segel model \eqref{fractionalKellerSegel}; blow-up, depending on mass and concentration of the initial datum. The numerical tests were performed on the unit disc $\Omega = B_{1}(0)$; a uniform mesh of 37325 nodes was used for two different fractional orders, $s = 0.75$ in $(a)$ and $s = 0.5 + \epsilon$ in $(b)$, with different initial data having different masses $M$ and concentrations, given by the parameter $\sigma$. Each cell corresponds to a numerical simulation in which either there is no numerical blow-up (green cell) or numerical blow-up occurs (orange cell). The threshold value for numerical blow-up was set to $\delta = 10^{-7}$.}
\label{Fig6}
\end{figure}

\begin{figure}[H]
\centering
\begin{subfigure}{0.49\textwidth}
\includegraphics[width = \textwidth]{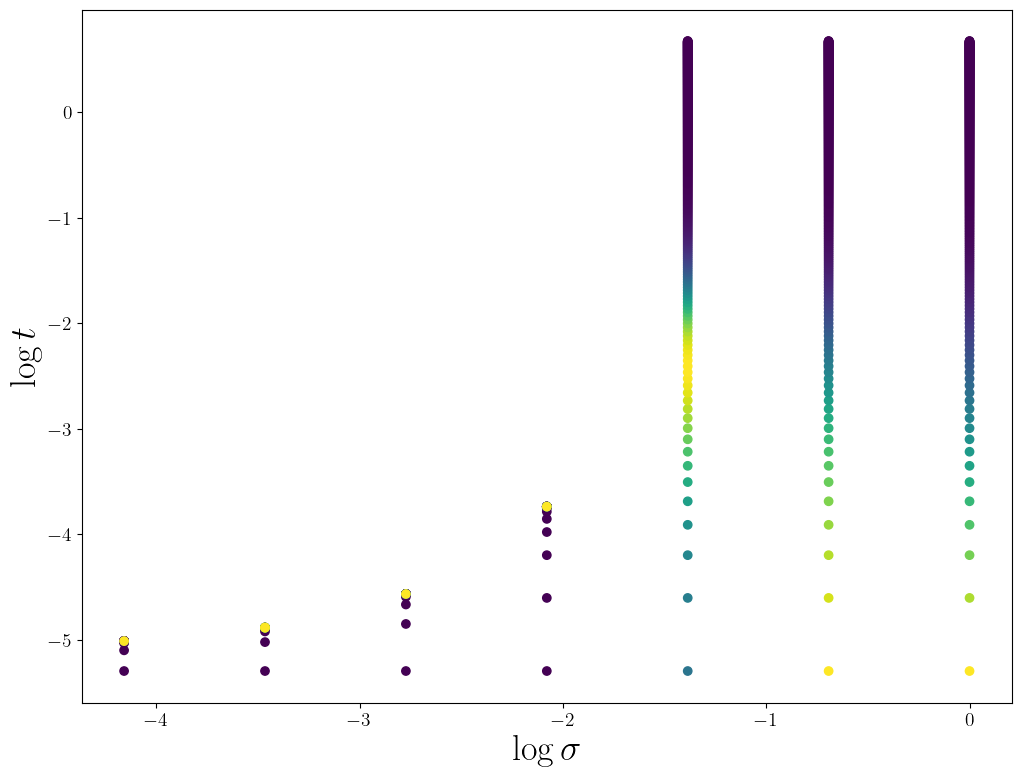}  
\end{subfigure}
\hfill
\begin{subfigure}{0.49\textwidth}
\centering
\includegraphics[width = \textwidth]{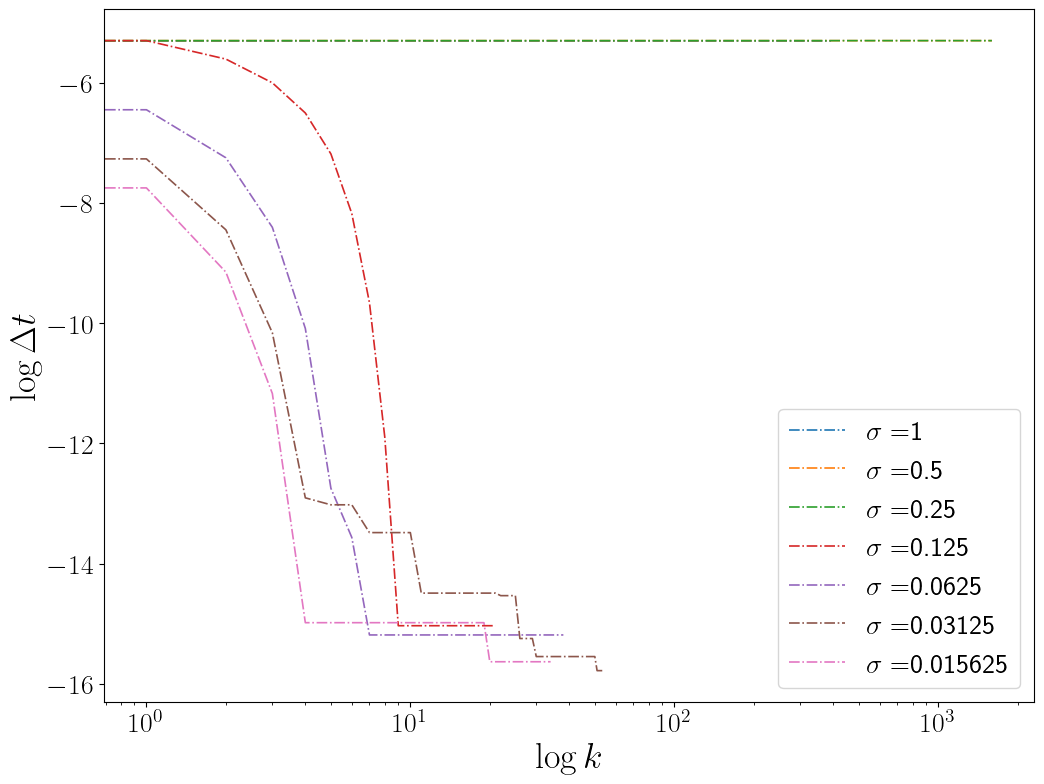}
\end{subfigure}
\caption{The fractional Keller--Segel model \eqref{fractionalKellerSegel}; blow-up depending on the concentration of the initial datum. Numerical tests were performed on the unit disc $\Omega = B_{1}(0)$; a uniform mesh of 37325 mesh points was used for fractional order $s= 0.75$, with initial data having the same mass $M = 4\pi$ and different concentrations, given by the parameter $\sigma$. The threshold value for numerical blow-up was set to $\delta = 10^{-7}$.}
\label{Fig7}
\end{figure}

\section*{Acknowledgements}
This work was supported by the Advanced Grant Nonlocal-CPD (Nonlocal PDEs for Complex Particle Dynamics: Phase Transitions, Patterns and Synchronization) of the European Research Council Executive Agency (ERC) under the European Union’s Horizon 2020 research and innovation programme (grant agreement No. 883363). JAC was also partially supported by the EPSRC grant numbers EP/T022132/1 and EP/V051121/1. YN was supported by EPSRC grants EP/Y010086/1 and EP/Y030990/1. We are grateful to Professor Christoph Schwab (ETH Z\"{u}rich)
for his helpful comments on the preprint version of this paper.

\FloatBarrier

\bibliographystyle{abbrv}
\bibliography{./fractional-rational-FE.bib}

\renewcommand{\appendixname}{Appendix}
\addappendix
\label{sec:appendix} We present two auxiliary results: the first, Lemma \ref{ScaleLemma}, is a scaling property of the spectral fractional Laplacian, that is valid under both a homogeneous Dirichlet and a homogeneous Neumann boundary condition; the second, Lemma \ref{PohoLemma}, is a Pohozaev-type estimate, valid for the Dirichlet spectral fractional Laplacian. For future reference, we recall the definition of the fractional-order Sobolev space $H^s(\Omega)$, for $s \in (0,1)$ and $\Omega \subseteq \mathbb{R}^d$:
\begin{equation}
    H^s(\Omega) := \left\{ u \in L^2(\Omega) \textrm{ such that } |u|_{H^s(\Omega)}^2 := \int_\Omega \int_\Omega \frac{|u(x) - u(y)|^2}{|x - y|^{d+2s}} \dx \dy < \infty \right\}, 
\end{equation}
equipped with the norm $\|\cdot\|_{H^s(\Omega)}$ defined by $\|u\|_{H^s(\Omega)}:=(\|u\|^2 + |u|^2_{H^s(\Omega)})^{\frac{1}{2}}$.
Let $H^s_0(\Omega) \coloneqq \overline{C^\infty_0(\Omega)}^{H^s(\Omega)}$, that is, the closure of the set of infinitely many times
continuously differentiable functions with compact support contained in $\Omega$, with respect to the topology induced by the norm $\|\cdot\|_{H^s(\Omega)}$.

\begin{lemma}[Scaling property for the spectral fractional Laplacian] \label{ScaleLemma}
Let $\Omega$ be a bounded Lipschitz domain, and for $s \in (0,1)$ let $(-\Delta_{\mathcal{B}}^{\Omega})^{s}$ be the spectral fractional Laplacian on $\Omega$ with boundary condition $\mathcal{B}$, defined in \eqref{SpectralFracLap}. Let $u \in \mathbb{H}^{s}_{\mathcal{B}}(\Omega)$,  where $\mathbb{H}^{s}_{\mathcal{B}}(\Omega)$ signifies the space $\mathbb{H}^s(\Omega)$ defined in \eqref{DirFracLap} in the case of the 
Dirichlet Laplacian, and the space $\mathbb{H}^s(\Omega)$ defined in \eqref{NeuFracLap} in the case of the Neumann Laplacian. Let $\lambda > 0$, and assume that if $\lambda>1$ then we have an extension $\overline{u}(x')$ of $u$ to $\lambda \Omega$, with $x'= \lambda x \in \lambda \Omega$ for some $x \in \Omega$, such that $\overline{u} \in \mathbb{H}^{s}_{\mathcal{B}}(\lambda \Omega)$. Then,
\begin{equation}
(-\Delta_{\mathcal{B}}^{\Omega})^{s} u_{\lambda}(x) = \lambda^{2s} [ (-\Delta_{\mathcal{B}}^{\lambda \Omega})^{s} \overline{u} ](\lambda x).
\end{equation}
\end{lemma}
\begin{proof}
Let $H_{\mathcal{B}}^1(\Omega)$ denote $H^1_0(\Omega)$ in the case of a homogeneous Dirichlet boundary condition, and $H^1_\ast(\Omega)$ in the case of a homogeneous Neumann boundary condition. 
Let us further notice that if $\{ \phi_{k}(x) \}_{k\geq 1}$, $x \in \Omega$ is a set of $L^{2}(\Omega)$-orthonormal eigenfunctions of the Laplacian in $\Omega$ with boundary condition $\mathcal{B}(\phi_{k}) = 0$, with corresponding eigenvalues $\{ \lambda_{k} \}_{k\geq 1}$, that is 
$(\nabla \phi_{k}, \nabla \phi)_{L^2(\Omega^d)} = \lambda_{k} (\phi_{k}, \phi)_{L^2(\Omega)}$ for all  $\phi \in H^{1}_{\mathcal{B}}(\Omega)$, then $\{ \phi_{k} (x'/\lambda)\}_{k \geq 1}$, $x' \in \lambda \Omega$ are a set of eigenfunctions for the Laplacian on $\lambda \Omega$ with corresponding eigenvalues $\{ \lambda_k / \lambda^2 \}_{k\geq 1}$. This is because for $\phi' \in H^{1}_{\mathcal{B}}(\lambda \Omega)$, using the change of variables $x'=\lambda x$, one has 
\begin{align*}
\int_{\lambda \Omega} \nabla_{x'} \phi_{k}(x'/\lambda) \cdot \nabla_{x'} \phi'(x') \dx' &= \lambda^{d} \int_{\Omega} \nabla_{x'} \phi_{k}(x) \cdot \nabla_{x'} \phi'(\lambda x) \dx = \frac{\lambda^{d}}{\lambda^{2}} \int_{\Omega} \nabla_{x} \phi_{k}(x) \cdot \nabla_{x} \phi'(\lambda x) \dx \\
&=  \frac{\lambda^{d}}{\lambda^{2}} \lambda_{k} \int_{\Omega} \phi_{k}(x) \phi'( \lambda x) \dx = \frac{\lambda_{k}}{\lambda^{2}} \int_{\lambda \Omega} \phi_{k}(x'/\lambda ) \phi'(x') \dx'.
\end{align*}
If we need $L^{2}(\lambda \Omega)$-orthonormal eigenfunctions for the Laplacian on $\lambda \Omega$ we need to choose $\phi'_{k}(x') := \frac{1}{\lambda^{d/2}} \phi_{k}(x'/\lambda)$ because $\int_{\lambda \Omega} \phi_{k}(x'/\lambda) \phi_{j}(x'/\lambda) \dx' = \lambda^{d} \int_{\Omega} \phi_{k}(x) \phi_{j}(x) \dx = \lambda^{d} \delta_{j,k}$.

With these considerations in place we now observe that
\begin{align*}
(-\Delta^{\Omega}_{\mathcal{B}})^{s} u_{\lambda}(x) &= \sum_{k=1}^{\infty} \lambda_{k}^{s} \bigg( \int_{\Omega} \phi_{k}(x) \overline{u}(\lambda x) \dx \bigg) \phi_{k}(x) = \sum_{k=1}^{\infty} \lambda_{k}^{s} \bigg( \int_{\lambda \Omega} \overline{u}(x') \frac{1}{\lambda^{d/2}} \phi_{k}(x'/\lambda) \dx' \bigg) \frac{1}{\lambda^{d/2}} \phi_{k}(x'/\lambda) \\
&= \lambda^{2s} \sum_{k=1}^{\infty} \bigg(\frac{\lambda_{k}}{\lambda^{2}} \bigg)^{s} \bigg( \int_{\lambda \Omega} \overline{u}(x') \phi'_{k}(x') \dx' \bigg) \phi'_{k}(x') = \lambda^{2s} [(-\Delta_{\mathcal{B}}^{\lambda \Omega})^{s} \overline{u}](x'),
\end{align*}
where we have used the change of variable $x' = \lambda x$.
\end{proof}

An inequality analogous to the one stated in the next lemma for the fractional spectral Dirichlet Laplacian is known to hold for the integral representation of the fractional Laplacian in the case of an exterior homogeneous Dirichlet boundary condition  (i.e., $u\equiv 0$ on $\mathbb{R}^d \setminus \Omega$) for a bounded star-shaped domain $\Omega$ (see \cite{RosOton2014, boulenger2016}). 

\begin{lemma}[Pohozaev-type estimate] \label{PohoLemma}
Suppose that $\Omega \subset \mathbb{R}^{d}$ is a bounded Lipschitz domain, with $d=2,3$, and let $1/2 < s< 1$. Then, for all $u \in H^{1}_{0}(\Omega)$ with $(-\Delta_{\textrm{D}})^{s}u \in L^2(\Omega)$, we have 
\begin{equation}
\int_{\Omega} ( x \cdot \nabla u) (-\Delta_{\textrm{D}})^{s} u \dx \leq \frac{1}{2}(2s - d) \int_{\Omega} u (-\Delta_{\textrm{D}})^{s} u \dx.
\end{equation}
\end{lemma}

\begin{proof}

Consider the natural extension of $u \in H^1_0(\Omega)$ to the whole of $\mathbb{R}^{d}$:
\begin{equation} \label{DirExt}
    \overline{u} := \left \{\begin{array}{ll} u  & \textrm{in } \Omega, \\ 0 & \textrm{in } \mathbb{R}^{d} \setminus \Omega. \end{array} \right. 
\end{equation}
As $\textrm{supp}\, \overline{u} \subset\subset M$ for every (for example, bounded open Lipschitz) domain $M \supset \Omega$, it follows that $\overline{u}\in H^1_0(M)$. Therefore, in particular, $\overline{u} \in H^s_0(M)$
for all $s \in (1/2,1)$. 
With $\mathbb{H}^s(\Omega)$ defined with respect to the Dirichlet Laplacian, as is the case here, it is known that, for $s \in (1/2,1)$, $H^s_0(\Omega) = \mathbb{H}^s(\Omega)$ (see Remark 2.1 in \cite{caffarelli2016fractional} and the paragraph preceding Theorem 2.5 therein). 
Hence, $\overline{u} \in \mathbb{H}^s(M)$. Suppose that $\lambda > 1$. To simplify the notation, from now on we will refer to $\overline{u}$ simply as $u$ and each time we use the notation $u_{\lambda}$ (where $u_\lambda(x) := u(\lambda x)$), we will mean $\overline{u}_{\lambda}$ (where $\overline{u}_\lambda(x) := \overline{u}(\lambda x)$).  Clearly, $u_\lambda \in H^1_0(\lambda\Omega)$ for all $\lambda >1$. 

Since, by assumption, $u \in H^{1}_0(\Omega)$, we have that 
\begin{equation} \label{GradConvLambda}
\frac{u_{\lambda} - u}{\lambda - 1} \rightharpoonup x \cdot \nabla u \quad \text{weakly in } L^{2}(\Omega) \text{ as } \lambda \searrow 1;
\end{equation}
cf. Lemma 4.2 in \cite{rosoton2015} for a proof. Since, again by assumption, $(-\Delta_{\textrm{D}})^{s} u \in L^{2}(\Omega)$, it follows by passing to the weak limit that 
\begin{equation*}
\int_{\Omega} (x \cdot \nabla u) (-\Delta_{\textrm{D}})^{s} u \dx = \ddlambda \bigg|_{\lambda\searrow 1} \int_{\Omega} u_{\lambda} (-\Delta_{\textrm{D}})^{s} u \dx.
\end{equation*}
Next, as a preparation for the application of the scaling property stated in Lemma \ref{ScaleLemma}, we rewrite the expression above in terms of the notation of Lemma \ref{ScaleLemma} as 
\begin{equation*}
\int_{\Omega} (x \cdot \nabla u) (-\Delta_{\textrm{D}}^{\Omega})^{s} u \dx = \ddlambda \bigg|_{\lambda\searrow 1} \int_{\Omega} u_{\lambda} (-\Delta_{\textrm{D}}^{\Omega})^{s} u \dx.
\end{equation*}
As $u_{\lambda} \in H^{1}_0(\lambda \Omega)$ and $u \in H^{1}_0(\Omega)$, the scaling property of the spectral fractional Laplacian stated in Lemma \ref{ScaleLemma} implies that
\begin{align*}
\int_{\Omega} u_{\lambda}(x) (-\Delta_{\textrm{D}}^{\Omega})^{s} u(x) \dx &= \int_{\Omega}(-\Delta_{\textrm{D}}^{\Omega})^{s} u_{\lambda}(x) u(x) \dx = \lambda^{2s} \int_{\Omega} [(-\Delta_{\textrm{D}}^{\lambda \Omega})^{s} u](\lambda x) u(x)  \dx \\
&= \lambda^{2s -d} \int_{ \lambda \Omega} [(-\Delta_{\textrm{D}}^{\lambda \Omega})^{s} u](x') u(x'/\lambda) \dx',
\end{align*}
where we applied the change of variable $x' = \lambda x$. We then have 
\begin{align}
\int_{\Omega} (x \cdot \nabla u) (-\Delta_{\textrm{D}})^{s} u \dx &= \ddlambda \bigg|_{\lambda\searrow 1} \lambda^{2s -d} \int_{\lambda \Omega} [(-\Delta_{\textrm{D}}^{\lambda \Omega})^{s} u](x') u(x'/\lambda) \dx' \nonumber \\
&=(2s - d) \int_{\Omega}  (-\Delta_{\textrm{D}}^{ \Omega})^{s} u(x) u(x) \dx +\ddlambda \bigg|_{\lambda\searrow 1} I_{\lambda} \nonumber \\
&= (2s - d) \int_{\Omega} u (-\Delta_{\textrm{D}}^\Omega)^{s} u dx +  \ddlambda \bigg|_{\lambda\searrow 1} I_{\lambda}, \label{Ilambda0}
\end{align}
with $I_{\lambda} :=  \int_{\lambda \Omega} [(-\Delta_{\textrm{D}}^{\lambda \Omega})^{s} u](x') u(x'/\lambda) \dx'$. 
We recall that 
\begin{equation*}
[(-\Delta_{\mathcal{B}}^{\lambda \Omega})^{s} u](x') = \sum_{k=1}^{\infty} \frac{\lambda_{k}^{s}}{\lambda^{2s}} u_k^{\lambda} \phi_{k}'(x'), \quad \textrm{where} \quad  u_k^\lambda = \int_{\lambda \Omega} u(x) \phi_k'(x') \dx' \quad \textrm{and} \quad \phi'_k(x') = \frac{1}{\lambda^{d/2}} \phi_k(x'/\lambda)
\end{equation*}
are the scaled eigenfunctions $\{ \phi'_{k} \}_{k \geq 1}$, defined in the proof of Lemma \ref{ScaleLemma}, which are $L^2(\lambda \Omega)$-orthonormal.  
Let $\tilde{u}_k := \int_{\lambda \Omega} u(x') \phi_k(x'/\lambda) \dx'$; hence, we have 
\begin{equation} \label{ulambdau}
\tilde{u}_k = \int_{\lambda \Omega} u(x') \phi_k(x'/\lambda) \dx' = \lambda^{d/2} \int_{\lambda\Omega} u(x') \phi_k'(x') \dx' = \lambda^{d/2} u_k^\lambda.
\end{equation}
By expanding the expression for $I_{\lambda}$, it follows that
\begin{align} \label{Ilambda1}
    I_{\lambda} &= \sum_{k, j= 1}^{\infty} \lambda_k^s \lambda^{-2s} \bigg( \frac{1}{\lambda^{d/2}} \int_{\lambda \Omega} u(x) \phi_{k}(x'/\lambda) \dx' \bigg) u_j \int_{\lambda \Omega} \frac{1}{\lambda^{d/2}} \phi_k(x'/\lambda)  \phi_j(x'/\lambda) \dx' \\ &=\lambda^{-2s} \sum_{k=1}^{\infty} \lambda_{k}^{s} u_k \widetilde{u}_{k} = \frac{1}{2} \lambda^{-2s} \sum_{k=1}^{\infty} \lambda_{k}^{2s} u_{k}^{2}  + \frac{1}{2} \lambda^{-2s} \sum_{k=1}^{\infty} \tilde{u}_{k}^{2} - \lambda^{-2s} \sum_{k=1}^{\infty} \lambda_{k}^{s} (u_k - \tilde{u}_k)^2, \nonumber 
\end{align}
where in the last line we have used the equality $ab = \frac{1}{2} a^2 + \frac{1}{2} b^2 - ( a - b)^2$. 

We have that
$\lambda^{-2s} \sum_{k=1}^{\infty}  \tilde{u}_{k}^{2} = \lambda^{d-2s} \sum_{k=1}^{\infty} (u_k^\lambda)^2 = \lambda^{d-2s} \| u \|_{L^2(\lambda \Omega)}^2 = \lambda^{d-2s} \| u \|_{L^2(\Omega)}^2$,
where the last equality follows from the fact that $u$ has been extended by $0$ outside $\Omega$ (cf. \eqref{DirExt}). Furthermore, 
\begin{align*}
u_k - \tilde{u}_k = \int_{\Omega} u(x) \phi_k(x) \dx - \int_{\lambda \Omega} u(x') \phi_k(x'/\lambda) \dx' = \int_{\Omega} (u(x) - \lambda^d u(\lambda x) ) \phi_k(x) \dx,
\end{align*}
and therefore \eqref{Ilambda1} 
is equivalent to 
\begin{align} \label{Ilambda2}
    I_{\lambda} &= \frac{1}{2} \lambda^{-2s} \| u \|_{\mathbb{H}^{2s}(\Omega)}^2 + \frac{1}{2} \lambda^{d-2s} \| u \|_{L^2(\Omega)}^2 - \lambda^{-2s}  \| u - \lambda^d   u_\lambda \|_{\mathbb{H}^s(\Omega)}^2 =I _{\lambda, 1} + I_{\lambda, 2},
\end{align}
with $I_{\lambda, 1} \coloneqq \frac{1}{2} \lambda^{-2s} \| u \|_{\mathbb{H}^{2s}(\Omega)}^2 + \frac{1}{2} \lambda^{d-2s} \| u \|_{L^2(\Omega)}^2$ and  $I_{\lambda, 2} \coloneqq - \lambda^{-2s}  \| u - \lambda^d   u_\lambda \|_{\mathbb{H}^s(\Omega)}^2$.

Next, observe that
\begin{equation} \label{Htheta}
    u(x) - u(\lambda x) = \int_0^1 \ddtheta u(H_\theta(x)) \textrm{d}\theta = (1-\lambda) \int_0^1 \nabla u(H_\theta(x)) \cdot x \textrm{d}\theta, \quad \textrm{where} \quad H_\theta(x) = \theta x + (1-\theta)\lambda x.
\end{equation}
We shall use this equality to prove that $\| u - u_\lambda\|_{L^2(\Omega)} \to 0$ as $\lambda \searrow 1$. Note, to this end, that
\begin{align*}
    \| u - u_\lambda\|_{L^2(\Omega)}^2 &= (1-\lambda)^2 \int_\Omega \left(  \int_0^1 \nabla u (H_\theta(x)) \cdot x \dd\theta \right)^2 \dx \leq (1-\lambda)^2 \int_\Omega \int_0^1 |\nabla u (H_\theta(x)) \cdot x|^2 \dd\theta \dx \\
    &= (1-\lambda)^2 \int_0^1  \int_{M_\theta^\lambda} |\nabla u(x') \cdot x'|^2 (C_\theta^\lambda)^2 C_\theta^\lambda \dx'  \dd\theta = (1-\lambda)^2 \int_0^1 \| \nabla u \cdot x'\|_{L^2(M_\theta^\lambda)}^2 (C_\theta^\lambda)^3 \dd \theta,
\end{align*}
where in the inequality in the first line we have applied the Cauchy--Schwarz inequality and in the passage from the first to the second line we have used the change of variables $x':=\theta x + (1-\theta)\lambda x$,  $C_\theta^\lambda := 1/(\theta + (1-\theta)\lambda)$ and $M_\theta^\lambda := (C_\theta^\lambda)^{-1} \Omega$. Since $(C_\theta^\lambda)^{-1} \in [1, \lambda]$ for $\lambda>1$ sufficiently close to 1 we have $M_\theta^\lambda \subset M$ for a fixed $M$ (contained, for example, in $2\Omega$) and recalling the natural extension \eqref{DirExt} of $u$, it follows that $\| \nabla u \cdot x' \|_{L^2(M_\theta^\lambda)} \leq\| \nabla u \cdot x' \|_{L^2(M)}$. Therefore, because $C_\theta^\lambda \leq 1$, we finally have 
\begin{align} \label{L2boundulambda}
     \| u - u_\lambda\|_{L^2(\Omega)}^2 &\leq (1-\lambda)^2  \| |\nabla u \cdot x' |\|_{L^2(M)}^2.
\end{align}
As a consequence, $\| u - u_\lambda\|_{L^2(\Omega)} \to 0$ as $\lambda \searrow 1$.

Next, we shall prove that for $\lambda$ close to 1, there exists a constant $C(u)$, dependent on $u$, but independent of $\lambda$, such that $| u - u_\lambda|_{H^1(\Omega)} \leq C(u)$. Then, by \eqref{L2boundulambda} and using a function space interpolation inequality (analogous to Lemma A.2 in \cite{carrillosuli2024}), we will have for $s \in (1/2,1)$ that
\begin{align} \label{Hsboundulambda}
    \| u - u_\lambda\|_{\mathbb{H}^s(\Omega)} \leq C \|u - u_\lambda\|_{L^{2}(\Omega)}^{1-s} \; \| u -u_\lambda \|_{\mathbb{H}^1(\Omega)}^{s} \leq C(u) (1-\lambda)^{1-s},
\end{align}
which will then imply that $\| u-u_\lambda\|_{\mathbb{H}^s(\Omega)} \to 0$  as $\lambda \searrow 1$. 
Assuming that $\lambda < 2$, for example, the existence of such a constant $C(u)$ follows by the application of the Cauchy--Schwarz and triangle inequalities together with the fact that for the natural extension of $u$ we have $|u|_{H^1(M)} < \infty$ for any subset $M \supset \Omega$:
\begin{align*}
    |u - u_\lambda|_{H^1(\Omega)}^2 &= \int_\Omega |\nabla(u(x) -u(\lambda x)|^2 \dx = \int_\Omega |\nabla u(x) -\lambda (\nabla u)(\lambda x)|^2 \dx
    \\ & = \int_\Omega |\nabla u|^2 \dx + \lambda^2 \int_\Omega |(\nabla u)(\lambda x)|^2 \dx - 2 \lambda \int_\Omega \nabla u(x) \cdot \nabla u(\lambda x) \dx \\
    &\leq |u|_{H^1(\Omega)}^2 + |u|_{H^1(M)}^2 + 2 |u|_{H^1(\Omega)} |u|_{H^1(M)}=C(u). 
\end{align*}

Thus we have shown  that $\| u-u_\lambda\|_{\mathbb{H}^s(\Omega)} \to 0$  as $\lambda \searrow 1$. Hence, by the triangle inequality,
\begin{align*}
    \| u - \lambda^d u_\lambda\|_{\mathbb{H}^s(\Omega)} &\leq (1-\lambda^d) \| u \|_{\mathbb{H}^s(\Omega)} + \lambda^d \| u - u_\lambda\|_{\mathbb{H}^s(\Omega)} 
\end{align*}
and therefore we also have $\lim_{\lambda \searrow 1}  \| u - \lambda^d u_\lambda\|_{\mathbb{H}^s(\Omega)} = 0$. In addition,  

\begin{align}
\frac{\dd}{\dd \lambda} \bigg|_{\lambda \searrow 1} I_{\lambda, 2} &=  \ddlambda  \bigg|_{\lambda \searrow 1} \bigg(  -\lambda^{-2s} \sum_{k=1}^{\infty} \lambda_k^s \bigg( \int_{\Omega} (u(x)  - \lambda^d u(\lambda x)) \phi_k(x) \dx \bigg)^2 \bigg) \nonumber \\
&= 2s \lim_{\lambda \searrow 1} \| u - \lambda^d u_\lambda\|_{\mathbb{H}^s(\Omega)}^2  - \ddlambda \bigg|_{\lambda \searrow 1}  \sum_{k=1}^{\infty} \lambda_k^s \bigg( \int_{\Omega} (u(x)  - \lambda^d u(\lambda x)) \phi_k(x) \dx \bigg)^2 \nonumber \\
&= - \ddlambda \bigg|_{\lambda \searrow 1}  \sum_{k=1}^{\infty} \lambda_k^s \bigg( \int_{\Omega} (u(x)  - \lambda^d u(\lambda x)) \phi_k(x) \dx \bigg)^2 \nonumber \\
&= - \ddlambda \bigg|_{\lambda \searrow 1} \| u - \lambda^d u_\lambda \|_{\mathbb{H}^s(\Omega)}^2. \label{I2limit1} 
\end{align}
Because the limit in $\lambda$ is taken for $\lambda>1$ decreasing toward 1 and $\lim_{\lambda \searrow 1}  \| u - \lambda^d u_\lambda\|_{\mathbb{H}^s(\Omega)} = 0$, we see that the expression appearing on the right-hand side of \eqref{I2limit1} is nonpositive.

\color{black}
In this way, we have from \eqref{Ilambda2} that 
\begin{align*}
\ddlambda \bigg|_{\lambda\searrow 1} I_{\lambda} &= \frac{\dd}{\dd\lambda} \bigg|_{\lambda\searrow 1} I_{\lambda, 1} + \frac{\dd}{\dd\lambda} \bigg|_{\lambda\searrow 1} I_{\lambda, 2} \\
&\leq \frac{\dd}{\dd\lambda} \bigg|_{\lambda\searrow 1} I_{\lambda, 1} = -s \|u\|_{\mathbb{H}^{2s}(\Omega)}^2 + \frac{d-2s}{2} \| u \|_{L^2(\Omega)}^2 \leq  \frac{d-2s}{2} \| u \|_{\mathbb{H}^{s}(\Omega)}^2. 
\end{align*}
Returning to \eqref{Ilambda0} we then arrive at the desired final result.

\end{proof}

\begin{lemma} \label{ConvLemma}
Assume the same hypotheses as in Lemma \ref{ScaleLemma} and consider the problems 
\begin{equation}
(-\Delta_{\mathcal{B}}^{\Omega})^s u = f, \quad f \in H^1(\Omega),
\qquad \textrm{and} \qquad
(-\Delta_{\mathcal{B}}^{\lambda \Omega})^s \overline{u} = \overline{f}, \quad \overline{f} \in H^1(\lambda \Omega),
\end{equation}
where $\overline{f}$ is an extension or restriction $f$ on $\lambda \Omega$ respectively if $\lambda >1$ or $\lambda < 1$, such that $\overline{f} = f$ on $\Omega$. Then, $\overline{u} \to u$ in $\mathbb{H}^s(\Omega)$ as $\lambda \to 1$.
\end{lemma}
\begin{proof}
Similarly as in Lemma \ref{ScaleLemma}, we shall use the notation $f_\lambda(x) = \overline{f}(\lambda x)$. By the same argument as in Lemma \ref{ScaleLemma} we have
$((-\Delta_{\mathcal{B}}^{\lambda \Omega})^{-s}\overline{f})(x) = \lambda^{2s} (-\Delta^{\Omega}_{\mathcal{B}})^{-s} (f_\lambda(x/\lambda)) $. Hence, 
\begin{equation*}
\overline{u}(x) - u(x) = ((-\Delta_{\mathcal{B}}^{\lambda \Omega})^{-s}\overline{f})(x) - (-\Delta^{\Omega}_{\mathcal{B}})^{-s}f(x) =  (-\Delta^{\Omega}_{\mathcal{B}})^{-s} (\lambda^{2s} \overline{f}(x) - f(x)).
\end{equation*}
Therefore, 
\begin{align*}
\| \overline{u} - u \|_{\mathbb{H}^s(\Omega)}^2 &= ((-\Delta_{\mathcal{B}}^{\Omega})^s (\overline{u}-u),\overline{u} - u)  = ( \lambda^{2s} \overline{f} - f, (-\Delta^{\Omega}_{\mathcal{B}})^{-s} ( \lambda^{2s} \overline{f} - f)) \\
&\leq \|  \lambda^{2s} \overline{f} - f \|_{L^2(\Omega)} \| (-\Delta^{\Omega}_{\mathcal{B}})^{-s} ( \lambda^{2s} \overline{f} - f) \|_{L^2(\Omega)} \\
&\leq C \|  \lambda^{2s} \overline{f} - f \|_{L^2(\Omega)}^2,
\end{align*}
where in the passage from the first to the second line we have used the Cauchy--Schwarz inequality, and in the passage from the second to the third line we have used the bound on the solution of a fractional Poisson equation with spectral fractional Laplacian in terms of the source term (see the Appendix in \cite{carrillosuli2024}).
Using the proof of Lemma 4.2 in \cite{rosoton2015}, 
we have that $\|  \lambda^{2s} \overline{f} - f \|_{L^2(\Omega)}^2$ tends to zero as $\lambda \to 1$. This completes the proof. \end{proof} 
\end{document}